\documentclass{article}
\usepackage[utf8]{inputenc}
\usepackage{amsmath,amsthm, amssymb, amsfonts}
\usepackage{todonotes}
\usepackage{kantlipsum}
\usepackage{wrapfig}
\usepackage{paralist}
\usepackage{wasysym}
\usepackage{verbatim}
\usepackage{fullpage}
\usepackage[pdfencoding=auto]{hyperref}
\usepackage{graphicx}
\usepackage{wrapfig}
\usepackage{mathtools}
\usepackage{mathrsfs}
\usepackage{enumitem}
\usepackage{dsfont}
\usepackage{units}
\usepackage[]{algorithm2e}
\usepackage{ltablex}
\usepackage{varwidth}

\usepackage{xcolor}

\allowdisplaybreaks[1]
\usepackage{sistyle} 
\newcommand\bSI[1]{{\small[\SI{}{#1}]}}

\makeatletter
\newlength\unitwdth
\newlength\numwdth
\newlength\tdima
\newcommand\SIdescr[2]{%
    \setlength\tdima{\linewidth}%
    \addtolength\tdima{\@totalleftmargin}%
    \addtolength\tdima{-\dimen\@curtab}%
    \addtolength\tdima{-\unitwdth}%
    \addtolength\tdima{-\numwdth}%
    \parbox[t]{\tdima}{%
        #1
        \leaders\hbox{$\m@th\mkern \@dotsep mu\hbox{\tiny.}\mkern \@dotsep mu$}%
        \hfill
        \ifhmode\strut\fi
        \makebox[0pt][l]{%
            \makebox[\unitwdth][l]{}%
            \makebox[\numwdth][r]{#2}}}}
\makeatother

\newcommand{\Z}{\mathbb{Z}}
\newcommand{\N}{\mathbb{N}}

\newcommand{\R}{\mathbb{R}}

\newcommand{\CC}{\mathbb{C}}

\newcommand{\fa}{ \quad \text{ for all }}

\newcommand{\eps}{\varepsilon}
\newcommand{\Realization}{\mathrm{R}}
\newcommand{\sgn}{\operatorname{sign}}

\let\emptyset\varnothing

\newcommand{\prob}{\sigma}

\newcommand{\CalF}{\mathcal{F}}
\newcommand{\CalG}{\mathcal{G}}

\newcommand{\Schwartz}{\mathcal{S}}

\DeclareMathOperator*{\supp}{supp}

\newcommand{\with}{\,:\,}
\newcommand{\identity}{\mathrm{id}}
\DeclareMathOperator{\spn}{span}
\newcommand{\Indicator}{{\mathds{1}}}

\newcommand{\ArchitectureSpace}{\mathcal{RNN}_\varrho^\Omega}

\newcommand{\FirstN}[1]{\{1,\dots,#1\}}

\newcommand{\conc}{{\raisebox{2pt}{\tiny\newmoon} \,}}

\newcommand{\cN}{\mathcal{NN}}
\newcommand{\cRN}{\mathcal{RNN}}

\DeclareMathOperator{\Lip}{Lip}

\newtheorem{theorem}{Theorem}[section]
\newtheorem*{theorem*}{Theorem}
\newtheorem{remark}[theorem]{Remark}
\newtheorem{definition}[theorem]{Definition}
\newtheorem{proposition}[theorem]{Proposition}
\newtheorem{lemma}[theorem]{Lemma}
\newtheorem{corollary}[theorem]{Corollary}

\newtheorem*{remark*}{Remark}
\newtheorem*{proposition*}{Proposition}

\numberwithin{equation}{section}

\definecolor{darkcandyapplered}{rgb}{0.64, 0.0, 0.0}

\title{Topological properties of the set of functions\\
       generated by neural networks of fixed size}
\author{Philipp Petersen\footnotemark[2] \footnotemark[1] \\ Universit\"at Wien
\and Mones Raslan\footnotemark[3] \footnotemark[1] \\ TU Berlin
\and Felix Voigtlaender\footnotemark[4] \footnotemark[1]\\ KU Eichst\"att--Ingolstadt}

\usepackage[ddmmyyyy,hhmmss]{datetime}
\newdateformat{daymonthyeardate}{%
  \THEDAY.\THEMONTH.\THEYEAR}
\usepackage{fancyhdr}
\usepackage{lastpage}
\usepackage{etoolbox}
\begin{document}
\maketitle

\begin{abstract}
  We analyze the topological properties of the set of
  functions that can be implemented by neural networks of a fixed size.
  Surprisingly, this set has many undesirable properties.
  It is highly non-convex, except possibly for a few exotic activation functions.
  Moreover, the set is not closed with respect to $L^p$-norms,
  $0 < p < \infty$, for all practically-used activation functions, and also not
  closed with respect to the $L^\infty$-norm for all practically-used activation
  functions except for the ReLU and the parametric ReLU.
  Finally, the function that maps a family of weights to the function computed by the
  associated network is not inverse stable for every practically used activation function.
  In other words, if $f_1, f_2$ are two functions realized by neural networks
  and if $f_1, f_2$ are close in the sense that $\|f_1 - f_2\|_{L^\infty} \leq \eps$ for $\eps > 0$,
  it is, regardless of the size of $\eps$, usually not possible to find weights $w_1, w_2$
  close together such that each $f_i$ is realized by a neural network with weights $w_i$.
  Overall, our findings identify potential causes for issues in the
  training procedure of deep learning such as no guaranteed convergence,
  explosion of parameters, and slow convergence.
\end{abstract}

\smallskip


\renewcommand{\thefootnote}{\fnsymbol{footnote}}
\footnotetext[1]{All three authors contributed equally to this work.}
\footnotetext[2]{Institut für Mathematik,
Universit\"at Wien,
Oskar-Morgenstern-Platz 1,
1090 Vienna,
Austria,
e-mail: \texttt{philipp.petersen@univie.ac.at}}
\footnotetext[3]{Institut für Mathematik,
Technische Universit\"at Berlin,
Straße des 17.~Juni 136,
10623 Berlin,
Germany,
e-mail: \texttt{raslan@math.tu-berlin.de}}
\footnotetext[4]{Department of Scientific Computing,
Catholic University of Eichstätt-Ingolstadt,
Kollegiengebäude I Bau B,
Ostenstraße 26,
85072 Eichstätt,
Germany,
e-mail: \texttt{felix.voigtlaender@ku.de}}

\renewcommand{\thefootnote}{\arabic{footnote}}

\section{Introduction}\label{sec:Introduction}

\textit{Neural networks}, introduced in 1943 by McCulloch and Pitts \cite{MP43},
are the basis of every modern machine learning algorithm based on \emph{deep learning}
\cite{Goodfellow-et-al-2016, LeCun2015DeepLearning, schmidhuber2015deep}.
The term \emph{deep learning} describes a variety of methods that are
based on the data-driven manipulation of the weights of a neural network.
Since these methods perform spectacularly well in practice, they have become the
state-of-the-art technology for a host of applications including image
classification \cite{Huang2017DenselyCC, simonyan2014very, Krizhevsky2012Imagenet},
speech recognition \cite{hinton2012deep, dahl2012context, wu2016stimulated},
game intelligence \cite{silver2017mastering, usunier2016episodic, yannakakis2017artificial},
and many more.

This success of deep learning has encouraged many scientists to pick up research
in the area of neural networks after the field had gone dormant for decades.
In particular, quite a few mathematicians have recently investigated the
properties of different neural network architectures, hoping that this can
explain the effectiveness of deep learning techniques.
In this context, mathematical analysis has mainly been conducted in the context
of statistical learning theory \cite{Cucker02onthe}, where the overall success of a learning method
is determined by the approximation properties of the underlying function class,
the feasibility of optimizing over this class, and the generalization capabilities of the class,
when only training with finitely many samples.

In the \emph{approximation theoretical} part of deep learning research,
one analyzes the expressiveness of deep neural network architectures.
The universal approximation theorem
\cite{Cybenko1989, Hornik1989universalApprox,PinkusUniversalApproximation}
demonstrates that neural networks can approximate \emph{any} continuous function,
as long as one uses networks of increasing complexity for the approximation.
If one is interested in approximating more specific function classes
than the class of all continuous functions, then
one can often quantify more precisely how large the networks have to be to
achieve a given approximation accuracy for functions from the restricted class.
Examples of such results are \cite{Barron1993, Mhaskar:1996:NNO:1362203.1362213,
Mhaskar1993, YAROTSKY2017103, boelcskeiNeural, PetV2018OptApproxReLU}.
Some articles \cite{Montufar:2014:NLR:2969033.2969153, cohen2016expressive,
pmlr-v70-safran17a, PetV2018OptApproxReLU, YarotskyPhaseDiagram} study in particular in which
sense \emph{deep} networks have a superior expressiveness
compared to their shallow counterparts, thereby partially
explaining the efficiency of networks with many layers in deep learning.

Another line of research studies the training procedures employed in deep learning.
Given a set of training samples, the training process is an \emph{optimization problem}
over the parameters of a neural network, where a loss function is minimized.
The loss function is typically a non-linear, non-convex function of the weights of the network,
rendering the optimization of this function highly challenging
\cite{blum1989training, judd1987learning, bartlett2002hardness}.
Nonetheless, in applications, neural networks are often trained successfully
through a variation of stochastic gradient descent.
In this regard, the energy landscape of the problem was studied and found to allow convergence
to a global optimum, if the problem is sufficiently overparametrized; see
\cite{nguyen2017loss,allen2018convergence,venturi2018neural,chizat2018global,freeman2016topology}.

The third large area of mathematical research on deep neural networks
is analyzing the so-called \emph{generalization} error of deep learning.
In the framework of statistical learning theory \cite{Cucker02onthe,MohriFoundations},
the discrepancy between the \emph{empirical} loss and the \emph{expected} loss of a classifier
is called the generalization error.
Specific bounds for this error for the class of deep neural networks were analyzed for instance in
\cite{AnthonyBartlett, bartlett2002rademacher}, and in more specific settings for instance in
\cite{bartlett2017spectrally, bartlett2019nearly}.

\medskip

In this work, we study neural networks from a different point of view.
Specifically, we study the \emph{structure of the set of functions implemented by neural networks
of fixed size}.
These sets are naturally (non-linear) subspaces of classical function spaces like $L^p(\Omega)$
and $C(\Omega)$ for compact sets $\Omega$.

Due to the size of the networks being fixed, our analysis is inherently non-asymptotic.
Therefore, our viewpoint is fundamentally different from the analysis in the framework
of statistical learning theory. Indeed, in approximation theory,
the expressive power of networks growing in size is analyzed.
In optimization, one studies the convergence properties of iterative algorithms---usually that
of some form of stochastic gradient descent.
Finally, when considering the generalization capabilities of deep neural networks,
one mainly studies how and with which probability the empirical loss of a classifier converges
to the expected loss, for increasing numbers of random training samples and depending on the sizes
of the underlying networks.

Given this strong delineation to the classical fields, we will see that our point of view
yields interpretable results describing phenomena in deep learning
that are not directly explained by the classical approaches.
We will describe these results and their interpretations in detail
in Subsections~\ref{sec:NonConvIntro}--\ref{sec:NonStableIntro}.

We will use standard notation throughout most of the paper without explicitly introducing it.
We do, however, collect a list of used symbols and notions in Appendix~\ref{sub:Notation}.
To not interrupt the flow of reading, we have deferred several auxiliary results
to Appendix~\ref{sec:auxResult} and all proofs and related statements
to Appendices~\ref{app:Convex}--\ref{app:InvStab}.

Before we continue, we formally introduce the notion of spaces of neural networks of fixed size.

\subsection*{Neural networks of fixed size: basic terminology}

To state our results, it will be necessary to distinguish between a \emph{neural network}
as a set of weights and the associated function implemented by the network,
which we call its \emph{realization}.
To explain this distinction, let us fix numbers $L, N_0, N_1, \dots, N_{L} \in \N$.
We say that a family $\Phi = \big( (A_\ell,b_\ell) \big)_{\ell = 1}^L$ of matrix-vector tuples
of the form $A_\ell \in  \R^{N_{\ell} \times N_{\ell-1}}$ and $b_\ell \in \R^{N_\ell}$
is a \textbf{neural network}.
We call $S\coloneqq(N_0, N_1, \dots, N_L)$ the \textbf{architecture} of $\Phi$;
furthermore $N(S)\coloneqq \sum_{\ell = 0}^L N_\ell$ is called the \textbf{number of neurons of $S$}
and $L = L(S)$ is the \textbf{number of layers of $S$}.
We call $d\coloneqq N_0$ the \textbf{input dimension} of $\Phi$ and throughout this introduction
we assume that the \textbf{output dimension} $N_L$ of the networks is equal to one.
For a given architecture $S$, we denote by $\cN(S)$
the \textbf{set of neural networks with architecture $S$}.

Defining the realization of such a network $\Phi = \big( (A_\ell,b_\ell) \big)_{\ell=1}^L$
requires two additional ingredients: a so-called \textbf{activation function}
$\varrho : \R \to \R$, and a domain of definition $\Omega \subset \R^{N_0}$.
Given these, the \textbf{realization of the network} $\Phi = \big( (A_\ell,b_\ell) \big)_{\ell=1}^L$
is the function
\begin{align*}
  \Realization_\varrho^\Omega \left( \Phi \right) :
  \Omega \to \R , \ \
  x \mapsto x_L \, ,
\end{align*}
where $x_L$ results from the following scheme:
\begin{equation*}
  \begin{split}
    x_0 &\coloneqq x, \\
    x_{\ell} &\coloneqq \varrho(A_{\ell} \, x_{\ell-1} + b_\ell),
    \quad \text{ for } \ell = 1, \dots, L-1,\\
    x_L &\coloneqq A_{L} \, x_{L-1} + b_{L},
  \end{split}
\end{equation*}
and where $\varrho$ acts componentwise;
that is, $\varrho(x_1,\dots,x_d) := (\varrho(x_1),\dots,\varrho(x_d))$.
In what follows, we study topological properties of sets of realizations of
neural networks with a \emph{fixed size}.
Naturally, there are multiple conventions to specify the size of a network.
We will study the
\textbf{set of realizations of networks with a given architecture $S$ and activation function $\varrho$};
that is, the set
$\cRN_{\varrho}^{\Omega}(S) \coloneqq \{\Realization_\varrho^\Omega(\Phi) \colon \Phi \in \cN(S) \}$.
In the context of machine learning, this point of view is natural,
since one usually prescribes the network architecture,
and during training only adapts the weights of the network.

Before we continue, let us note that the set $\cN(S)$ of all neural networks
(that is, the network weights) with a fixed architecture forms a finite-dimensional vector space,
which we equip with the norm
\[
  \|\Phi\|_{\cN(S)}
  \coloneqq \| \Phi \|_{\mathrm{scaling}} + \max_{\ell = 1,\dots,L} \|b_\ell\|_{\max}
  \qquad \text{for} \qquad \Phi = \big( (A_\ell, b_\ell) \big)_{\ell=1}^L \in \cN (S) ,
\]
where $\| \Phi \|_{\mathrm{scaling}} \coloneqq \max_{\ell = 1,\dots,L } \| A_\ell \|_{\max}$.
If the specific architecture of $\Phi$ does not matter,
we simply write $\|\Phi\|_{\mathrm{total}}\coloneqq  \|\Phi\|_{\cN(S)}$.
In addition, if $\varrho$ is continuous, we denote the \textbf{realization map} by
\begin{equation}
  \Realization^{\Omega}_{\varrho} :
  \cN(S) \to     C(\Omega ; \R^{N_L}), ~
  \Phi                 \mapsto \Realization^{\Omega}_{\varrho} (\Phi).
  \label{eq:RealizationMapping}
\end{equation}

While the activation function $\varrho$ can in principle be chosen arbitrarily, a couple
of particularly useful activation functions have been established in the literature.
We proceed by listing some of the most common activation functions,
a few of their properties, as well as references to articles using these
functions in the context of deep learning.
We note that all activation functions listed below are non-constant, monotonically increasing,
\emph{globally} Lipschitz continuous functions.
This property is much stronger than the assumption of \emph{local} Lipschitz continuity
that we will require in many of our results.
Furthermore, all functions listed below belong to the class $C^\infty(\R\setminus \{0\}).$

\begin{centering}
\renewcommand{\arraystretch}{2}
\begin{tabularx}{\linewidth}{|l|l|l|l|}
  \hline
    \textbf{Name}
  & \textbf{Given by}
  & \vtop{\hbox{\strut \textbf{Smoothness}/} \hbox{\strut \textbf{Boundedness}}}
  & \textbf{Cit.}
  \\ \hline
    rectified linear unit (ReLU)
  & $\max\{0,x\}$
  & $C(\R)$ / Unbounded
  & \cite{NairHinton}
  \\ \hline
    parametric ReLU
  & $\max\{ax,x\}$ for some $a \geq 0$, $a \neq 1$
  & $C(\R)$ / Unbounded
  & \cite{ParReLu}
  \\ \hline
    exponential linear unit
  & $x\cdot \chi_{x\geq 0}(x) + (\exp(x)-1)\cdot \chi_{x<0}(x)$
  & $C^1(\R)$ / Unbounded
  & \cite{ExpLU}
  \\ \hline
    softsign
  & $\frac{x}{1+|x|}$
  & $C^1(\R)$ / Bounded
  & \cite{Bergstra+2009}
  \\ \hline
    \vtop{\hbox{\strut inverse square root}\hbox{\strut linear unit }}
  & $x\cdot \chi_{x\geq 0}(x) + \frac{x}{\sqrt{1+ax^2}}\cdot \chi_{x<0}(x)$
    for $a>0$
  & $C^2(\R)$ / Unbounded
  & \cite{ISRLU}
  \\ \hline
    inverse square root unit
  & $\frac{x}{\sqrt{1 + a x^2}}$ for some $a>0$
  & Analytic / Bounded
  & \cite{ISRLU}
  \\ \hline
    sigmoid / logistic
  & $\frac{1}{1+\exp(-x)}$
  & Analytic / Bounded
  & \cite{HaykinLogFunction}
  \\ \hline
    tanh
  & $\frac{\exp(x)-\exp(-x)}{\exp(x)+\exp(-x)}$
  & Analytic / Bounded
  & \cite{MaasTanh}
  \\ \hline
    arctan
  & $\arctan(x)$
  & Analytic / Bounded
  & \cite{ArctanLiao}
  \\ \hline
    softplus
  & $\ln(1+\exp(x))$
  & Analytic / Unbounded
  & \cite{GlorotBordesBengio}
  \\ \hline
   \caption{Commonly-used activation functions and their properties}
   \label{tab:ActFunctions}
  \end{tabularx}
\end{centering}


In the remainder of this introduction, we discuss our results concerning the topological properties
of the sets of realizations of neural networks with fixed architecture and their interpretation
in the context of deep learning.
Then, we give an overview of related work.
We note at this point that it is straightforward to generalize all of the results in this paper
to neural networks for which one only prescribes the total number of neurons and layers
and not the specific architecture.

For simplicity, we will always assume in the remainder of this introduction that
$\Omega \subset \R^{N_0}$ is compact with non-empty interior.

\subsection{Non-convexity of the set of realizations}\label{sec:NonConvIntro}

We will show in Section~\ref{sec:Shape} (Theorem~\ref{thm:NoConvexityEver})
that, for a given architecture $S$, the set $\cRN_{\varrho}^{\Omega}(S)$ \emph{is not convex},
except possibly when the activation function is a polynomial, which is clearly not the case for
any of the activation functions that are commonly used in practice.


In fact, for a large class of activation functions (including the ReLU
and the standard sigmoid activation function), the set $\cRN_{\varrho}^{\Omega}(S)$
turns out to be \emph{highly non-convex} in the sense that for every
$r \in [0,\infty)$, the set of functions having uniform distance at most $r$ to
any function in $\cRN_{\varrho}^{\Omega}(S)$ is not convex.
We prove this result in Theorem~\ref{thm:epsconv} and Remark~\ref{rem:EpsilonConvexity}.

This non-convexity is undesirable, since for non-convex sets,
there do not necessarily exist well-defined projection operators onto them.
In classical statistical learning theory \cite{Cucker02onthe},
the property that the so-called regression function can be \emph{uniquely} projected
onto a convex (and compact) hypothesis space greatly simplifies the learning problem; see
\cite[Section~7]{Cucker02onthe}.
Furthermore, in applications where the realization of a network---%
rather than its set of weights---is the quantity of interest (for example when a network is used
as an Ansatz for the solution of a PDE, as in \cite{lagaris1998artificial, weinan2017deep}),
our results show that the Ansatz space is non-convex.
This non-convexity is inconvenient if one aims for a convergence proof of the
underlying optimization algorithm, since one cannot apply convexity-based fixed-point theorems.
Concretely, if a neural network is optimized by stochastic gradient descent
so as to satisfy a certain PDE, then it is interesting to see if there even exists a network
so that the iteration stops.
In other words, one might ask whether gradient descent on the set of neural networks
(potentially with bounded weights) has a fixed point.
If the space of neural networks were convex and compact,
then the fixed-point theorem of Schauder would guarantee the existence of such a fixed point.

\medskip{}

\subsection{(Non-)closedness of the set of realizations}
\label{sec:NonCloseIntro}

For any fixed architecture $S$, we show in Section~\ref{sec:Closed} (Theorem~\ref{thm:nonclosed})
that $\cRN_{\varrho}^{\Omega}(S)$ \emph{is not a closed subset of $L^p (\mu)$ for $0 < p < \infty$},
under very mild assumptions on the measure $\mu$ and the activation function $\varrho$.
The assumptions concerning $\varrho$ are satisfied for all activation functions used in practice.

For the case $p = \infty$, the situation is more involved: For all activation
functions that are commonly used in practice---\emph{except for the (parametric) ReLU}---%
the associated sets $\cRN_{\varrho}^{\Omega}(S)$ \emph{are non-closed
also with respect to the uniform norm}; see Theorem~\ref{thm:GeneralNonClosednessC}.
For the (parametric) ReLU, however, the question of closedness
of the sets $\cRN_{\varrho}^{\Omega}(S)$ remains mostly open.
Nonetheless, in two special cases, we prove in Section~\ref{sec:ClosReLUNet}
that the sets $\cRN_{\varrho}^{\Omega}(S)$ are closed.
In particular, for neural network architectures with two layers only,
Theorem~\ref{thm:closedReLU} establishes the \emph{closedness of $\cRN_{\varrho}^{\Omega}(S)$,
where $\varrho$ is the (parametric) ReLU}.

A practical consequence of the observation of non-closedness can be identified with the help
of the following argument that is made precise in Subsection~\ref{sec:ConsOfNonClosed}:
We show that the set
\[
  \left\{
    \mathrm{R}_\varrho^\Omega (\Phi)
    \,:\,
    \Phi = ((A_\ell,b_\ell))_{\ell=1}^L \text{ has architecture } S
    \text{ with } \|A_\ell\| + \|b_\ell\| \leq C
  \right\}
\]
of realizations of neural networks with a fixed architecture and
all affine linear maps bounded in a suitable norm, \emph{is} always closed.
As a consequence, we observe the following phenomenon of exploding weights:
If a function $f$ is such that it does not have a best approximation in $\cRN^\Omega_\varrho(S)$,
that is, if there does not exist $f^*\in \cRN^\Omega_\varrho(S)$ such that
\[
  \|f^* - f\|_{L^p(\mu)}
  = \tau_f
  \coloneqq \inf_{g \in {\cRN_{\varrho}^{\Omega}(S)}}
              \|f-g\|_{L^p(\mu)},
\]
then for any sequence of networks $(\Phi_n)_{n \in \N}$ with architecture $S$ satisfying
$\|f - \Realization_\varrho^\Omega (\Phi_n)\|_{L^p(\mu)} \to \tau_f$,
the weights of the networks $\Phi_n$ cannot remain uniformly bounded as $n \to \infty$.
In words, if $f$ does not have a best approximation in the set of neural networks of fixed size,
then every sequence of realizations approximately minimizing the distance to $f$
will have exploding weights.
Since $\cRN_{\varrho}^{\Omega}(S)$ is not closed, there do exist functions $f$ which do not have
a best approximation in $\cRN_{\varrho}^{\Omega}(S)$.

Certainly, the presence of large coefficients will make the numerical optimization
increasingly unstable.
Thus, exploding weights in the sense described above are highly undesirable in practice.

The argument above discusses an approximation problem in an $L^p$-norm.
In practice, one usually only minimizes ``empirical norms''.
We will demonstrate in Proposition~\ref{prop:ExplodingWeights} that also in this situation,
for increasing numbers of samples, the weights of the neural networks
that minimize the empirical norms necessarily explode under certain assumptions.
Note that the set-up of having a fixed architecture and a potentially unbounded
number of training samples is common in applications where neural networks are trained
to solve partial differential equations.
There, training samples are generated during the training process
\cite{lagaris1998artificial, weinan2018deep}.

\subsection{Failure of inverse stability of the realization map}
\label{sec:NonStableIntro}

As our final result, we study (in Section~\ref{sec:InverseStability})
the stability of the realization map $\Realization_\varrho^\Omega$
introduced in Equation~\eqref{eq:RealizationMapping},
which maps a family of weights to its realization.
Even though this map will turn out to be continuous
from the finite dimensional parameter space to $L^p (\Omega)$ for any
$p \in (0,\infty]$, we will show that it is \emph{not} inverse stable.
In other words, for two realizations that are very close
in the uniform norm, there do not always exist network weights associated with these
realizations that have a small distance.
In fact, Theorem~\ref{thm:InverseStability} even shows that there exists a sequence of realizations
of networks converging uniformly to $0$, but such that every sequence of weights
with these realizations is necessarily unbounded.

For both of these results---continuity and no inverse stability---we only need
to assume that the activation function $\varrho$ is Lipschitz continuous and not constant.

These properties of the realization map pinpoint a potential problem
that can occur when training a neural network:
Let us consider a regression problem, where a network is iteratively updated by
a (stochastic) gradient descent algorithm trying to minimize a loss function.
It is then possible that at some iterate the loss function exhibits a very small error,
even though the associated network \emph{parameters} have a large distance to the optimal parameters.
This issue is especially severe since a small error term leads to small steps
if gradient descent methods are used in the optimization.
Consequently, convergence to the very distant optimal weights will be
slow even if the energy landscape of the optimization problem
happens to be free of spurious local minima.

\subsection{Related work}

\paragraph{Structural properties:}

The aforementioned properties of non-convexity and non-closedness have, to some extent,
been studied before.
Classical results analyze the spaces of shallow neural networks,
that is, of $\mathcal{RNN}_\varrho^\Omega (S)$ for $S = (d, N_0, 1)$, so that $L = 2$.
For such sets of shallow networks, a property that has been extensively studied
is to what extent $\mathcal{RNN}_\varrho^\Omega (S)$ has the \emph{best approximation property}.
Here, we say that $\mathcal{RNN}_\varrho^\Omega (S)$ has the best approximation property,
if for every function $f \in L^p(\Omega)$, $1 \leq p \leq \infty$,
there exists a function $F(f) \in \mathcal{RNN}_\varrho^\Omega (S)$
such that $\| f - F(f) \|_{L^p} = \inf_{g \in \ArchitectureSpace (S)} \| f - g \|_{L^p}$.
In \cite{kainen1999approximation} it was shown that even if a minimizer always exists,
the map $f \mapsto F(f)$ is necessarily discontinuous.
Furthermore, at least for the Heaviside activation function, there does exist a (non-unique)
best approximation; see \cite{kainen2000best}.

Additionally, \cite[Proposition~4.1]{Girosi1990} demonstrates, for shallow networks as before,
that for the logistic activation function $\varrho(x) = (1 + e^{-x})^{-1}$,
the set $\mathcal{RNN}_\varrho^\Omega (S)$ does not have the best approximation property
in $C(\Omega)$.
In the proof of this statement, it was also shown that $\mathcal{RNN}_\varrho^\Omega (S)$
is not closed.
Furthermore, it is claimed that this result should hold for every non-linear activation function.
The previously mentioned result of \cite{kainen2000best} and Theorem~\ref{thm:closedReLU} below
disprove this conjecture for the Heaviside and ReLU activation functions, respectively.

\paragraph{Other notions of (non-)convexity:}

In deep learning, one chooses a loss function
${\mathcal{L}: C(\Omega) \to [0,\infty)}$, which is then minimized over the set of neural networks
$\cRN_{\varrho}^{\Omega}(S)$ with fixed architecture $S$.
A typical loss function is the empirical square loss, that is,
\[
  E_N(f) \coloneqq \frac{1}{N} \sum_{i = 1}^N |f(x_i) - y_i|^2,
\]
where $(x_i,y_i)_{i=1}^N \subset \Omega \times \R$, $N \in \N$.
In practice, one solves the minimization problem \emph{over the weights of the network};
that is, one attempts to minimize the function
$\mathcal{L} \circ \Realization_\varrho^\Omega : \cN(S) \to [0,\infty)$.
In this context, to assess the hardness of this optimization problem,
one studies whether $\mathcal{L} \circ \Realization_\varrho^\Omega$ is convex,
the degree to which it is non-convex, and if one can find remedies to alleviate the problem
of non-convexity, see for instance
\cite{BaldiHornik,NguyenHein17,venturi2018neural,Rotskoff2018NeuralNA, zhang2017convexified,bach2017breaking, jacot2018neural,
freeman2016topology,mei2018mean}.

It is important to emphasize that this notion of non-convexity describes properties
\emph{of the loss function}, in contrast to the non-convexity \emph{of the sets of functions}
that we analyze in this work.

\section{Non-convexity of the set of realizations}
\label{sec:Shape}

In this section, we analyze the convexity of the set of all neural network realizations.
In particular, we will show that this set is highly non-convex for all practically used
activation functions listed in Table~\ref{tab:ActFunctions}.
First, we examine the convexity of the set $\cRN^{\Omega}_{\varrho}(S)$:

\begin{theorem}\label{thm:NoConvexityEver}
  Let $S = (d, N_1, \dots, N_L)$ be a neural network architecture with $L \in \N_{\geq 2}$
  and let $\Omega \subset \R^d$ with non-empty interior.
  Moreover, let $\varrho: \R \to \R$ be locally Lipschitz continuous.

  If $\cRN_\varrho^\Omega(S)$ is convex, then $\varrho$ is a polynomial.
\end{theorem}

\begin{remark*}
  (1) It is easy to see that all of the activation functions in Table~\ref{tab:ActFunctions}
  are locally Lipschitz continuous, and that none of them is a polynomial.
  Thus, the associated sets of realizations are never convex.

  \medskip{}

  (2) In the case where $\varrho$ is a polynomial, the set $\cRN_\varrho^\Omega(S)$
  might or might not be convex.
  Indeed, if $S = (1, N, 1)$ and $\varrho(x) = x^m$, then it is not hard to see that
  $\cRN_\varrho^\Omega(S)$ is convex if and only if $N \geq m$.
\end{remark*}

\begin{proof}
The detailed proof of Theorem~\ref{thm:NoConvexityEver}
is the subject of Appendix~\ref{app:NoConvex}.
Let us briefly outline the proof strategy:
\begin{itemize}
  \item[1.] We first show in Proposition~\ref{prop:starshaped} that
            $\cRN_\varrho^\Omega(S)$ is \emph{closed under scalar multiplication},
            hence \emph{star-shaped with respect to the origin, i.e., 0 is a center.}%
            \footnote{{A subset $A$ of some vector space $V$ is called \textbf{star-shaped},
            if there exists some $f\in A$ such that for all $g \in A$, also
            $\{\lambda f + (1 - \lambda)g \colon \lambda \in [0,1]\} \subset A$.
            The vector $f$ is called a \textbf{center of $A$}.}}


  \item[2.]  Next, using the local Lipschitz continuity of $\varrho$,
             we establish in Proposition~\ref{prop:Centres} that the maximal
             \emph{number of linearly independent centers} of the set $\cRN_\varrho^\Omega(S)$
             is finite.
             Precisely, it is bounded by the number of parameters of the underlying neural networks,
             given by $\sum_{\ell = 1}^L (N_{\ell-1} + 1) N_{\ell}$.

  \item[3.] A direct consequence of Step~2 is that if  $\cRN_\varrho^\Omega(S)$ is convex,
            then it can only contain a finite number of linearly independent functions;
            see Corollary~\ref{cor:NoConvexity}.

  \item[4.] Finally, using that $\cRN_\varrho^{\R^d}(S)$ is a
            \emph{translation-invariant subset of} $C(\R^d)$,
            we show in Proposition~\ref{prop:NetworksUsuallyProduceManyLinearlyIndependentFunctions}
            that $\cRN_\varrho^{\R^d}(S)$ (and hence also $\cRN_\varrho^{\Omega}(S)$)
            contains \emph{infinitely many linearly independent functions},
            if $\varrho$ is not a polynomial.
            \qedhere
\end{itemize}
\end{proof}

In applications, the non-convexity of $\cRN^\Omega_\varrho(S)$ might not be as problematic
as it first seems.
If, for instance, the set $\cRN^\Omega_\varrho(S) + B_\delta(0)$ of functions that can be approximated
up to error $\delta > 0$ by a neural network with architecture $S$ was convex,
then one could argue that the non-convexity of $\cRN^\Omega_\varrho(S)$ was not severe.
Indeed, in practice, neural networks are only trained to minimize a certain empirical loss function,
with resulting bounds on the generalization error which are typically of size
$\eps = \mathcal{O}(m^{-1/2})$, with $m$ denoting the number of training samples.
In this setting, one is not really interested in ``completely minimizing''
the (empirical) loss function, but would be content with finding a function
for which the empirical loss is $\eps$-close to the global minimum.
Hence, one could argue that one is effectively working with a hypothesis space of the form
$\cRN^\Omega_\varrho(S) + B_\delta(0)$, containing all functions
that can be represented up to an error of $\delta$ by neural networks of architecture $S$.

To quantify this potentially more relevant notion of convexity of neural networks,
we define, for a subset $A$ of a vector space $\mathcal{Y}$, the \textbf{convex hull of $A$} as
\[
  \mathrm{co}(A)
  \coloneqq \bigcap_{B \subset \mathcal{Y} \text{ convex and } B \supset A}
              B \, .
\]
For $\eps > 0$, we say that a subset $A$ of a normed vector space $\mathcal{Y}$ is
$\eps$\textbf{-convex in} $(\mathcal{Y},\|\cdot\|_{\mathcal{Y}})$, if
\[
  \text{co}(A) \subset A + B_\eps(0) \, .
\]
Hence, the notion of $\eps$-convexity asks whether the convex hull of a set is contained
in an enlargement of this set.
%
Note that if $\cRN_\varrho^\Omega(S)$ is dense in $C(\Omega)$,
then its closure is trivially $\eps$-convex for all $\eps > 0$.
Our main result regarding the $\eps$-convexity of neural network sets shows that
this is the only case in which $\overline{\cRN_\varrho^\Omega(S)}$
is $\eps$-convex for any $\eps > 0$.

\begin{theorem}\label{thm:epsconv}
  Let $S = (d, N_1, \dots, N_{L-1}, 1)$ be a neural network architecture with $L \geq 2$,
  and let $\Omega \subset \R^d$ be compact.
  Let $\varrho: \R \to \R$ be continuous but not a polynomial,
  and such that $\varrho'(x_0) \neq 0$ for some $x_0 \in \R$.


  Assume that $\cRN_\varrho^{\Omega}(S)$ is \emph{not} dense in $C(\Omega)$.
  Then there does not exist \emph{any} $\eps > 0$ such that $\overline{\cRN_\varrho^{\Omega}(S)}$
  is $\eps$-convex in $\big( C(\Omega), \|\cdot\|_{\sup} \big)$.
\end{theorem}

\begin{remark*}
  All closures in the theorem are taken in $C(\Omega)$.
\end{remark*}

\begin{proof}
  The proof of this theorem is the subject of Appendix~\ref{app:epsconv}.
  It is based on showing that if $\overline{\cRN_\varrho^{\Omega}(S)}$ is $\eps$-convex
  for some $\eps > 0$, then in fact $\overline{\cRN_\varrho^{\Omega}(S)}$ is convex,
  which we then use to show that $\overline{\cRN_\varrho^{\Omega}(S)}$ contains all realizations
  of two-layer neural networks with activation function $\varrho$.
  As shown in \cite{PinkusUniversalApproximation}, this implies that $\cRN_\varrho^{\Omega}(S)$
  is dense in $C(\Omega)$, since $\varrho$ is not a polynomial.
\end{proof}

\begin{remark}\label{rem:EpsilonConvexity}
  While it is certainly natural to expect that
  $\overline{\cRN_\varrho^{\Omega}(S)} \neq C(\Omega)$ should hold for most
  activation functions $\varrho$, giving a reference including large classes
  of activation functions such that the claim holds is not straightforward.
  We study this problem more closely in Appendix~\ref{app:nonDenseNetworkSets}.

  To be more precise, from 
  Proposition~\ref{prop:MostActivationFunctionsDoNotYieldDenseRealisations}
  it follows that the ReLU, the parametric ReLU, the exponential linear unit,
  the softsign, the sigmoid, and the $\tanh$ yield realization sets
  $\cRN_\varrho^{\Omega}(S)$ which are \emph{not} dense in $L^p(\Omega)$
  and in $C(\Omega)$.

  The only activation functions listed in Table~\ref{tab:ActFunctions} for which
  we do \emph{not} know whether any of the realization sets $\cRN_\varrho^{\Omega}(S)$
  is dense in $L^p(\Omega)$ or in $C(\Omega)$ are:
  the inverse square root linear unit, the inverse square root unit,
  the softplus, and the $\arctan$ function.
  Of course, we expect that also for these activation functions, the resulting sets
  of realizations are never dense in $L^p(\Omega)$ or in $C(\Omega)$.

  Finally, we would like to mention that if $\Omega \subset \R^d$ has non-empty
  interior and if the input dimension satisfies $d \geq 2$, then it follows from the results in
  \cite{MaiorovRidgeFunctionsBestApproximation} that
  if $S = (d, N_1, 1)$ is a \emph{two-layer architecture},
  then $\cRN_\varrho^{\Omega}(S)$ is \emph{not} dense in $C(\Omega)$ or $L^p(\Omega)$.
\end{remark}

\section{(Non-)Closedness of the set of realizations}
\label{sec:Closed}

Let $\emptyset \neq \Omega \subset \R^d$ be compact with non-empty interior.
In the present section, we analyze whether the neural network realization set
$\cRN^\Omega_\varrho (S)$ with $S=(d,N_1,\dots,N_{L-1},1)$ is closed in $C(\Omega)$,
or in $L^p (\mu)$, for $p \in (0, \infty)$ and any measure $\mu$
satisfying a mild ``non-atomicness'' condition.
For the $L^p$-spaces, the answer is simple:
Under very mild assumptions on the activation function $\varrho$,
we will see that $\cRN^\Omega_\varrho (S)$ is never closed in $L^p (\mu)$.
In particular, this holds for \emph{all} of the activation functions
listed in Table~\ref{tab:ActFunctions}.
Closedness in $C(\Omega)$, however, is more subtle:
For this setting, we will identify several different classes of activation functions
for which the set $\cRN^\Omega_\varrho (S)$ is \emph{not} closed in $C(\Omega)$.
As we will see, these classes of activation functions cover all those functions listed in
Table~\ref{tab:ActFunctions}, \emph{except for the ReLU and the parametric ReLU}.
For these two activation functions, we were unable to determine whether the set
$\cRN^\Omega_\varrho (S)$ is closed in $C(\Omega)$ in general, but we conjecture this to be true.
Only for the case $L = 2$, we could show that these sets are indeed closed.


Closedness of $\cRN^\Omega_\varrho (S)$
is a highly desirable property as we will demonstrate in Section~\ref{sec:ConsOfNonClosed}.
Indeed, we establish that if $X = L^p(\mu)$ or $X = C(\Omega)$,
then, for all functions $f \in X$ \emph{that do not possess a best approximation
within $\mathcal{R} = \cRN^\Omega_\varrho (S)$},
the weights of approximating networks necessarily explode.
In other words, if $(\mathrm{R}_{\varrho}^{\Omega}(\Phi_n))_{n \in \N} \subset \mathcal{R}$
is such that $\|f - \mathrm{R}_{\varrho}^{\Omega}(\Phi_n)\|_{X}$
converges to $\inf_{g\in \mathcal{R}}\|f-g \|_{X}$ for $n \to \infty$,
then $\|\Phi_n\|_{\mathrm{total}} \to \infty$.
Such functions without a best approximation in $\mathcal{R}$
necessarily exist if $\mathcal{R}$ is not closed. Moreover, even in practical applications, where empirical error terms instead of $L^p(\mu)$ norms are minimized, the absence of closedness implies exploding weights as we show in Proposition \ref{prop:ExplodingWeights}.

Finally, we note that for simplicity, all ``non-closedness'' results in this section are
formulated for compact rectangles of the form $\Omega = [-B,B]^d$ only;
but our arguments easily generalize to any compact set $\Omega\subset \R^d$ with non-empty interior.

\subsection{Non-closedness in \texorpdfstring{$L^p(\mu)$}{Lᵖ(μ)}}
\label{subsec:NonClosedLp}

We start by examining the closedness with respect to $L^p$-norms for $p \in (0, \infty)$.
In fact, for all $B > 0$ and all widely used activation functions
(including all activation functions presented in Table~\ref{tab:ActFunctions}),
the set $\cRN_{\varrho}^{[-B,B]^d}(S)$ is \emph{not} closed in $L^p(\mu)$,
for any $p \in (0, \infty)$ and any ``sufficiently non-atomic'' measure $\mu$ on $[-B,B]^d$.
To be more precise, the following is true:

\begin{theorem}\label{thm:nonclosed}
Let $S = (d, N_1, \dots, N_{L-1}, 1)$ be a neural network architecture with $L \in \N_{\geq 2}$.
Let $\varrho: \R \to \R$ be a function satisfying the following conditions:
\begin{enumerate}
  \item[(i)] $\varrho$ is continuous and increasing;

  \item[(ii)] There is some $x_0 \in \R$ such that $\varrho$ is differentiable at $x_0$
              with $\varrho'(x_0) \neq 0$;

  \item[(iii)] There is some $r > 0$ such that
              $\varrho |_{(-\infty, -r) \cup (r, \infty)}$ is differentiable;

  \item[(iv)] \emph{At least one} of the following two conditions is satisfied:
              \begin{enumerate}
                \item[(a)] There are $\lambda, \lambda' \geq 0$
                           with $\lambda \neq \lambda'$ such that
                           $\varrho'(x) \to \lambda$ as $x \to \infty$,
                           and $\varrho'(x) \to \lambda'$ as $x \to -\infty$,
                           and we have $N_{L-1} \geq 2$.

                \item[(b)] $\varrho$ is bounded.
              \end{enumerate}
\end{enumerate}

Finally, let $B > 0$ and let $\mu$ be a finite Borel measure on $[-B, B]^d$
for which the support $\supp \mu$ is uncountable.
Then the set $\cRN_{\varrho}^{[-B,B]^d}(S)$ is \emph{not} closed in $L^p(\mu)$
for any $p \in (0, \infty)$.
More precisely, there is a function $f \in L^\infty (\mu)$ which satisfies
$f \in \overline{\cRN_{\varrho}^{[-B,B]^d}(S)} \setminus \cRN_{\varrho}^{[-B,B]^d}(S)$
for all $p \in (0,\infty)$, where the closure is taken in $L^p(\mu)$.
\end{theorem}

\begin{remark*}
  If $\supp \mu$ is countable, then $\mu = \sum_{x \in \supp \mu} \mu(\{ x \}) \, \delta_x$
  is a countable sum of Dirac measures, meaning that $\mu$ is \emph{purely atomic}.
  In particular, if $\mu$ is \emph{non-atomic}
  (meaning that $\mu(\{ x \}) = 0$ for all $x \in [-B,B]^d$), then $\supp \mu$ is uncountable
  and the theorem is applicable.
\end{remark*}

\begin{proof}
  For the proof of the theorem, we refer to Appendix~\ref{app:nonclosed}.
  The main proof idea consists in the approximation of a (discontinuous) step function
  which cannot be represented by a neural network with continuous activation function.
\end{proof}

\begin{corollary}\label{cor:StuffNeverClosedInLp}
  The assumptions concerning the activation function $\varrho$ in Theorem~\ref{thm:nonclosed}
  are satisfied for all of the activation functions listed in Table~\ref{tab:ActFunctions}.
  In any case where $\varrho$ is bounded, one can take $N_{L-1} = 1$;
  otherwise, one can take $N_{L-1} = 2$.
\end{corollary}

\begin{proof}
   For a proof of this statement, we refer to Appendix~\ref{app:StuffNeverClosedInLp}.
\end{proof}

\subsection{Non-closedness in \texorpdfstring{$C([-B,B]^d)$}{C([-B,B]ᵈ)}
            for many widely used activation functions}\label{subsec:NonClosedC}

We have seen in Theorem~\ref{thm:nonclosed} that under
reasonably mild assumptions on the activation function $\varrho$---which are satisfied
for all commonly used activation functions---the set $\cRN_\varrho^{[-B,B]^d}(S)$
is not closed in any $L^p$-space where $p \in (0, \infty)$.
However, the argument of the proof of Theorem~\ref{thm:nonclosed} breaks down
if one considers closedness with respect to the $\| \cdot \|_{\sup}$-norm.
Therefore, we will analyze this setting more closely in this section.
More precisely, in Theorem~\ref{thm:GeneralNonClosednessC},
we present several criteria regarding the activation function $\varrho$
which imply that the set 
 $\cRN_\varrho^{[-B,B]^d}(S)$ is \emph{not} closed in $C ([-B,B]^d)$.
We remark that in all these results, $\varrho$ will be assumed to be at least $C^1$.
Developing similar criteria for non-differentiable functions is an interesting topic
for future research.

Before we formulate Theorem~\ref{thm:GeneralNonClosednessC}, we need the following notion:
We say that a function $f: \R \to \R$ is \textbf{approximately homogeneous of order $(r,q) \in \N_0^2$}
if there exists $s > 0$ such that $|f(x) - x^r| \leq s$ for all $x \geq 0$ and
$|f(x) - x^q| \leq s$ for all $x \leq 0$. Now the following theorem holds:

\begin{theorem}\label{thm:GeneralNonClosednessC}
Let $S = (d, N_1, \dots, N_{L-1}, 1)$ be a neural network architecture
with $L\in \N_{\geq 2}$,  let $B > 0$, 
and let $\varrho : \R \to \R$.
Assume that \emph{at least one} of the following three conditions is satisfied:
\begin{itemize}
    \item[(i)] $N_{L-1}\geq 2$ and $\varrho \in C^1(\R) \setminus C^\infty(\R).$

    \item[(ii)] $N_{L-1}\geq 2$ and $\varrho$ is bounded, analytic, and not constant.

    \item[(iii)] $\varrho$ is approximately homogeneous of order $(r,q)$
                 for certain $r,q \in \N_0$ with $r \neq q$,
                 and $\varrho \in C^{\max\{r,q\}}(\R)$.
\end{itemize}
Then the set $\cRN_\varrho^{[-B,B]^d}(S)$ is \emph{not} closed in the space $C([-B,B]^d)$.
\end{theorem}

\begin{proof}
  For the proof of the statement, we refer to Appendix~\ref{app:GeneralNonClosednessC}.
  In particular, the proof of the statement under Condition~(i) can be found
  in Appendix~\ref{app:SmoothnessNoClosedness}.
  Its main idea consists of the uniform approximation of $\varrho'$
  (which cannot be represented by neural networks with activation function $\varrho$,
  due to its lack of sufficient regularity) by neural networks.
  For the proof of the statement under Condition~(ii), we refer to Appendix~\ref{app:closedAnalytic}.
  The main proof idea consists in the uniform approximation of an unbounded
  analytic function which cannot be represented by a neural network
  with activation function $\varrho$, since $\varrho$ itself is bounded.
  Finally, the proof of the statement under Condition~(iii) can be found
  in Appendix~\ref{app:closedHomogeneous}.
  Its main idea consists in the approximation of the function
  $x\mapsto (x)^{\max\{r,q\}}_+ \not \in C^{\max\{r,q\}}.$
\end{proof}

\begin{corollary}\label{cor:GeneralNonClosednessC}
  Theorem~\ref{thm:GeneralNonClosednessC} applies to all activation functions
  listed in Table~\ref{tab:ActFunctions} \emph{except} for the ReLU and the parametric ReLU.
  To be more precise,
  \begin{itemize}
    \item[(1)] Condition~(i) is fulfilled by the function $x\mapsto \max\{0, x \}^k$ for $k \geq 2$,
               and by the exponential linear unit, the softsign function,
               and the inverse square root linear unit.

    \item[(2)] Condition (ii) is fulfilled by the inverse square root unit, the sigmoid function,
               the $\tanh$ function, and the $\arctan$ function.

    \item[(3)] Condition (iii) (with $r = 1$ and $q = 0$) is fulfilled by the softplus function.
  \end{itemize}
\end{corollary}

\begin{proof}
  For the proof of this statement, we refer to Appendix~\ref{app:ProofCorollaryNonClosed}.
  In particular, for the proof of (1), we refer to Appendix~\ref{app:LimitedSmoothnessNonClosedness},
  the proof of (2) is clear and for the proof of (3), we refer to Appendix~\ref{app:Softplus}.
\end{proof}

\subsection{The phenomenon of exploding weights}
\label{sec:ConsOfNonClosed}

We have just seen that the realization set $\cRN_\varrho^{[-B,B]^d}(S)$ is not closed
in $L^p(\mu)$ for any $p \in (0,\infty)$ and every practically relevant activation function.
Furthermore, for a variety of activation functions, we have seen that $\cRN_\varrho^{[-B,B]^d}(S)$
is not closed in $C([-B,B]^d)$.
The situation is substantially different if the weights are taken from a compact subset:

\begin{proposition}\label{prop:bdWeights}
  Let $S = (d, N_1, \dots, N_L)$ be a neural network architecture,
  let $\Omega \subset \R^d$ be compact, let furthermore $p \in (0,\infty)$,
  and let $\varrho: \R \to \R$ be continuous.
  For $C > 0$, let
  \[
    \Theta_C := \big\{ \Phi\in \cN(S):\|\Phi\|_{\mathrm{total}}\leq C \big\}.
  \]
  Then the set $\mathrm{R}_{\varrho}^{\Omega}(\Theta_C)$ is compact in $C(\Omega)$
  as well as in $L^p(\mu)$, for any finite Borel measure $\mu$ on $\Omega$
  and any $p \in (0, \infty)$.
\end{proposition}

\begin{proof}
  The proof of this statement is based on the continuity of the realization map
  and can be found in Appendix~\ref{app:bdWeights}.
\end{proof}

Proposition~\ref{prop:bdWeights} helps to explain the phenomenon of exploding network weights
that is sometimes observed during the training of neural networks.
Indeed, let us assume that $\mathcal{R} := \cRN_\varrho^{[-B,B]^d}(S)$ is \emph{not} closed in
$\mathcal{Y}$, where $\mathcal{Y} \coloneqq L^p (\mu)$ for a Borel measure $\mu$ on $[-B,B]^d$,
or $\mathcal{Y} \coloneqq C([-B,B]^d)$; as seen in Sections~\ref{subsec:NonClosedLp}
and \ref{subsec:NonClosedC}, this is true under mild assumptions on $\varrho$.
Then, it follows that there exists a function $f \in \mathcal{Y}$ which \emph{does not have
a best approximation in $\mathcal{R}$}, meaning that there does not exist any $g \in \mathcal{R}$
satisfying
\[
  \| f - g \|_{\mathcal{Y}}
  = \inf_{h \in \mathcal{R}}
      \| f - h \|_{\mathcal{Y}}
  \eqqcolon M \, ;
\]
in fact, one can take any $f \in \overline{\mathcal{R}} \setminus \mathcal{R}$.
Next, recall from Proposition~\ref{prop:bdWeights} that the subset of $\mathcal{R}$
that contains only realizations of networks with uniformly bounded weights is compact.

Hence, we conclude the following:
For every sequence
\(
  (f_n)_{n \in \N}
 = \big( \Realization_\varrho^{[-B,B]^d} (\Phi_n) \big)_{n \in \N}
 \subset \mathcal{R}
\)
satisfying ${\| f - f_n \|_{\mathcal{Y}} \to M}$, we must have $\|\Phi_n\|_{\mathrm{total}} \to \infty$,
since otherwise, by compactness, $(f_n)_{n \in \N}$ would have a subsequence that converges
to some $h \in \mathrm{R}_{\varrho}^{\Omega}(\Theta_C) \subset \mathcal{R}$.
In other words, \emph{the weights of the networks $\Phi_n$ necessarily explode}.

The argument above only deals with the approximation problem in $C([-B,B]^d)$ or in $L^p(\mu)$
for $p \in (0, \infty)$.
In practice, one is often not concerned with these norms, but instead wants to minimize
an \emph{empirical loss function} over $\mathcal{R}$.
For the empirical square loss, this loss function takes the form
\[
  E_N(f) \coloneqq \frac{1}{N} \sum_{i = 1}^N |f(x_i) - y_i|^2,
\]
for $\big( (x_i, y_i) \big)_{i=1}^N \subset \Omega \times \R$ drawn i.i.d.~according to
a probability distribution $\prob$ on $\Omega \times \R$.
By the strong law of large numbers, for each fixed measurable function $f$,
the empirical loss function converges almost surely to the \emph{expected loss}
\begin{equation}\label{eq:expectedRisk}
  \mathcal{E}_\prob (f)
  \coloneqq \int_{\Omega\times \R}
              \left|f(x) - y\right|^2
            d \, \prob (x,y).
\end{equation}
This expected loss is related to an $L^2$ minimization problem.
Indeed, \cite[Proposition~1]{Cucker02onthe} shows that there is a measurable function
$f_\prob : \Omega \to \R$---called the \emph{regression function}---such that the expected risk
from Equation~\eqref{eq:expectedRisk} satisfies
\begin{equation}\label{eq:expectedRiskDecomposition}
  \mathcal{E}_\prob (f)
  = \mathcal{E}_\prob (f_\prob)
    + \int_{\Omega}
        \left| f(x) - f_\prob(x) \right|^2
      d \prob_\Omega(x)
  \quad \text{for each measurable} \quad f : \Omega \to \R .
\end{equation}
Here, $\prob_\Omega$ is the marginal probability distribution of $\prob$ on $\Omega$,
and $\mathcal{E}_\prob (f_\prob)$ is called the \emph{Bayes risk};
it is the minimal expected loss achievable by any (measurable) function.

In this context of a statistical learning problem,
we have the following result regarding exploding weights:

\begin{proposition}\label{prop:ExplodingWeights}
  Let $d \in \N$ and $B, K > 0$.
  Let $\Omega \coloneqq [-B,B]^d$, and let $\prob$ be a Borel probability measure
  on $\Omega \times [-K,K]$.
  Further, let $S = (d, N_1, \dots, N_{L-1}, 1)$ be a neural network architecture
  and $\varrho : \R \to \R$ be locally Lipschitz continuous.
  Assume that the regression function $f_\prob$ is such that there does not exist
  a best approximation of $f_\prob$ in $\cRN_\varrho^{\Omega}(S)$ with respect to
  $\| \cdot \|_{L^2(\prob_\Omega)}$.
  Let $\big( (x_i, y_i) \big)_{i \in \N} \overset{\mathrm{i.i.d.}}{\sim} \prob$;
  all probabilities below will be with respect to this family of random variables.

  If $(\Phi_N)_{N \in \N} \subset \cN(S)$ is a random sequence of neural networks
  (depending on $\big( (x_i, y_i) \big)_{i \in \N}$) that satisfies
  \begin{equation}\label{eq:ConvergenceOfMinimizers}
    \mathbb{P}
      \left(
        E_N \left(\Realization_\varrho^{\Omega}(\Phi_N)\right)
        - \inf_{f \in \cRN_\varrho^{\Omega}(S)}
            E_N(f)
        > \eps
      \right) \to 0,
    \quad \text{as } N \to \infty
    \quad \text{for all } \eps > 0 ,
  \end{equation}
  then $\| \Phi_N \|_{\mathrm{total}} \to \infty$ in probability as $N \to \infty$.
\end{proposition}

\begin{remark*}
  A compact way of stating Proposition~\ref{prop:ExplodingWeights} is that,
  if $f_\prob$ has no best approximation in $\cRN_\varrho^{\Omega}(S)$
  with respect to $\| \cdot \|_{L^2(\prob_\Omega)}$,
  then the weights of the minimizers (or approximate minimizers) of the empirical square loss
  explode for increasing numbers of samples.

  Since $\prob$ is unknown in applications, it is indeed possible that $f_\prob$
  has no best approximation in the set of neural networks.
  As just one example, this is true if $\prob_\Omega$ is any Borel probability measure
  on $\Omega$ and if $\prob$ is the distribution of $(X, g(X))$, where $X \sim \prob_\Omega$
  and $g \in L^2(\prob_\Omega)$ is bounded and satisfies
  $g \in \overline{\cRN_\varrho^{\Omega}(S)} \setminus \cRN_\varrho^{\Omega}(S)$,
  with the closure taken in $L^2(\prob_\Omega)$.
  The existence of such a function $g$ is guaranteed by Theorem~\ref{thm:nonclosed}
  if $\supp \prob_\Omega$ is uncountable.
\end{remark*}

\begin{proof}
  For the proof of Proposition~\ref{prop:ExplodingWeights},
  we refer to Appendix~\ref{app:ExplodingWeightsProof}.
  The proof is based on classical arguments of statistical learning theory
  as given in \cite{Cucker02onthe}.
\end{proof}

\subsection{Closedness of ReLU networks in \texorpdfstring{$C([-B,B]^d)$}{C([-B,B]ᵈ)}}
\label{sec:ClosReLUNet}

In this subsection, we analyze the closedness of sets of realizations of
neural networks with respect to the ReLU or the parametric ReLU activation
function in $C(\Omega)$, mostly for the case $\Omega = [-B, B]^d$.
We conjecture that the set of (realizations of) ReLU networks of a fixed complexity is closed
in $C(\Omega)$, but were not able to prove such a result in full generality.
In two special cases, namely when the networks have only two layers, or when
at least the \emph{scaling} weights are bounded, we can show that the associated set of
ReLU realizations is closed in $C(\Omega)$; see below.

We begin by analyzing the set of realizations with uniformly bounded scaling weights
and possibly unbounded biases, before proceeding with the analysis of two layer ReLU networks.

For $\Phi = \big( (A_1,b_1),\dots,(A_L,b_L) \big) \in \cN(S)$ satisfying
$\|\Phi\|_{\mathrm{scaling}} \leq C$ for some $C>0$, we say that the
network $\Phi$ has \textbf{$C$-bounded scaling weights}.
Note that this does \emph{not} require the biases $b_\ell$
of the network to satisfy $|b_\ell| \leq C$.

Our first goal in this subsection is to show that if $\varrho$ denotes the ReLU,
if $S = (d, N_1,\dots,N_L)$, if $C > 0$, and if $\Omega \subset \R^d$ is measurable and bounded,
then the set
\[
  \cRN_{\varrho}^{\Omega,C}(S)
  := \left\{
      \Realization_\varrho^{\Omega} (\Phi)
      \with
      \Phi \in \cN(S) \text{ with } \|\Phi\|_{\mathrm{scaling}} \leq C
     \right\}
\]
is closed in $C(\Omega;\R^{N_L})$ and in $L^p (\mu; \R^{N_L})$ for arbitrary $p \in [1, \infty]$.
Here, and in the remainder of the paper, we use the norm
$\|f\|_{L^p(\mu;\R^{N_L})} = \|\, |f| \,\|_{L^p(\mu)}$ for vector-valued $L^p$-spaces.
The norm on $C(\Omega; \R^{N_L})$ is defined similarly.
The difference between the following proposition and
Proposition~\ref{prop:bdWeights} is that in the following proposition,
the ``shift weights'' (the biases) of the networks can be potentially unbounded.
Therefore, the resulting set is merely closed, not compact.

\begin{proposition}\label{prop:BoundedScalingWeightsClosedness}
  Let $S = (d, N_1, \dots, N_L)$ be a neural network architecture,
  let $C > 0$, and let $\Omega \subset \R^d$ be Borel measurable and bounded.
  Finally, let $\varrho : \R \to \R, x \mapsto \max \{ 0, x \}$ denote the ReLU function.

  Then the set $\cRN_{\varrho}^{\Omega,C}(S)$ is closed in $L^p (\mu;\R^{N_L})$
  for every $p \in [1,\infty]$ and any finite Borel measure $\mu$ on $\Omega$.
  If $\Omega$ is compact, then $\cRN_{\varrho}^{\Omega,C}(S)$ is also closed in $C(\Omega;\R^{N_L})$.
\end{proposition}

\begin{remark*}
  In fact, the proof shows that each subset of
  $\cRN_{\varrho}^{\Omega,C}(S)$ which is bounded in $L^1 (\mu; \R^{N_L})$ (when $\mu(\Omega) > 0$)
  is precompact in $L^p(\mu; \R^{N_L})$ and in $C(\Omega; \R^{N_L})$.
\end{remark*}

\begin{proof}
  For the proof of the statement, we refer to Appendix~\ref{app:BoundedScalingWeightsClosedness}.
  The main idea is to show that for every sequence $(\Phi_n)_{n \in \N} \subset \mathcal{NN}(S)$
  of neural networks with uniformly bounded \emph{scaling} weights
  and with $\| \Realization_\varrho^\Omega (\Phi_n) \|_{L^1(\mu)} \leq M$,
  there exist a subsequence $(\Phi_{n_k})_{k\in\N}$ of $(\Phi_n)_{n\in\N}$ and neural networks
  $(\widetilde{\Phi}_{n_k})_{k\in\N}$ with uniformly bounded scaling weights \emph{and} biases such
  that
  \(
    \Realization^\Omega_\varrho \big( \widetilde{\Phi}_{n_k} \big)
    = \Realization^\Omega_\varrho \big( {\Phi}_{n_k} \big)
  \).
  The rest then follows from Proposition~\ref{prop:bdWeights}.
\end{proof}

As our second result in this section, we show that the set of realizations
of \emph{two-layer} neural networks with arbitrary scaling weights and biases
is closed in $C([-B,B]^d),$ if the activation is the parametric ReLU.
It is a fascinating question for further research whether this also holds for deeper networks.

\begin{theorem}\label{thm:closedReLU}
  Let $d, N_0 \in \N,$ let $B > 0$, and let $a \geq 0$.
  Let $\varrho_a:\R \to \R, x \mapsto \max \{x, ax \}$ be the parametric ReLU.
  Then $\cRN_{\varrho_a}^{[-B,B]^d}((d,N_0,1))$ is closed in $C([-B,B]^d)$.
\end{theorem}

\begin{proof}
  For the proof of the statement, we refer to Appendix~\ref{app:closedReLU};
  here we only sketch the main idea:
  First, note that each $f \in \cRN_{\varrho_a}^{[-B,B]^d}((d,N_0,1))$ is of the form
  $f(x) = c + \sum_{i=1}^{N_0} \varrho_a (\langle \alpha_i , x \rangle + \beta_i)$.
  The proof is based on a careful---and quite technical---analysis of the
  \emph{singularity hyperplanes} of the functions $\varrho_a (\langle \alpha_i, x \rangle + \beta_i)$,
  that is, the hyperplanes $\langle \alpha_i, x \rangle + \beta_i = 0$ on which these functions
  are not differentiable.
  More precisely, given a \emph{uniformly convergent} sequence
  $(f_n)_{n \in \N} \subset \cRN_{\varrho_a}^{[-B,B]^d}((d,N_0,1))$,
  we analyze how the singularity hyperplanes of the functions $f_n$ behave as $n \to \infty$,
  in order to show that the limit is again of the same form as the $f_n$.
  For more details, we refer to the actual proof.
\end{proof}

\section{Failure of inverse instability of the realization map}
\label{sec:InverseStability}

In this section, we study the properties of the realization map $\Realization^{\Omega}_{\varrho}$.
First of all, we observe that the realization map is continuous.

\begin{proposition}\label{prop:RealizationContinuity}
  Let $\Omega \subset \R^d$ be compact and let $S = (d, N_1, \dots, N_L)$
  be a neural network architecture.
  If the activation function $\varrho : \R \to \R$ is continuous, then the realization map from
  Equation~\eqref{eq:RealizationMapping} is continuous.
  If $\varrho$ is locally Lipschitz continuous, then so is $\Realization^{\Omega}_{\varrho}$.

  Finally, if $\varrho$ is globally Lipschitz continuous, then there is a constant
  $C = C(\varrho, S) > 0$ such that
  \[
    \mathrm{Lip}\big(\Realization^{\Omega}_{\varrho} (\Phi)\big)
    \leq C \cdot \| \Phi \|_{\mathrm{scaling}}^L
    \qquad \text{ for all } \Phi \in \cN(S) \, .
  \]
\end{proposition}

\begin{proof}
  For the proof of this statement, we refer to Appendix~\ref{app:RealCont}.
\end{proof}

In general, the realization map is not injective; that is, there can be networks
$\Phi \neq \Psi$ but such that
$\Realization^{\Omega}_{\varrho} (\Phi) = \Realization^{\Omega}_{\varrho} (\Psi)$;
in fact, if for instance
\[
  \Phi = \big( (A_1,b_1), \dots, (A_{L-1}, b_{L-1}), (0,0) \big)
  \quad \text{and} \quad
  \Psi = \big( (B_1, c_1), \dots, (B_{L-1}, c_{L-1}), (0,0) \big) \, ,
\]
then the realizations of $\Phi,\Psi$ are identical.

In this section, our main goal is to determine whether, up to the failure of
injectivity, the realization map is a homeomorphism onto its range;
mathematically, this means that we want to determine whether the realization map
is a \emph{quotient map}.
We will see that this is \emph{not} the case.

To this end, we will prove for fixed $\Phi$ that even if $\Realization^{\Omega}_{\varrho} (\Psi)$
is very close to $\Realization^{\Omega}_{\varrho}(\Phi)$,
it is \emph{not} true in general that
\(
  \Realization^{\Omega}_{\varrho} (\Psi)
  = \Realization^{\Omega}_{\varrho} (\widetilde{\Psi})
\)
for network weights $\widetilde{\Psi}$ close to $\Phi$.
Precisely, this follows from the following theorem for $\Phi = 0$ and $\Psi = \Phi_n$.

\begin{theorem}\label{thm:InverseStability}
  Let $\varrho : \R \to \R$ be Lipschitz continuous,
  but not affine-linear.
  Let $S = (N_0,\dots,N_{L-1},1)$ be a network architecture with
  $L \geq 2$, with $N_0 = d$, and $N_1 \geq 3$.
  Let $\Omega \subset \R^d$ be bounded with nonempty interior.

  Then there is a sequence $(\Phi_n)_{n \in \N}$ of networks with architecture
  $S$ and the following properties:
  \begin{enumerate}
    \item We have $\Realization^{\Omega}_{\varrho} (\Phi_n) \to 0$ uniformly on $\Omega$.

    \item We have $\mathrm{Lip}(\Realization^{\Omega}_{\varrho}(\Phi_n)) \to \infty$
          as $n \to \infty$.
  \end{enumerate}
  Finally, if $(\Phi_n)_{n \in \N}$ is a sequence of networks with architecture
  $S$ and the preceding two properties, then the following holds:
  For each sequence of networks $(\Psi_n)_{n \in \N}$ with architecture $S$ and
  $\Realization^{\Omega}_{\varrho} (\Psi_n) = \Realization^{\Omega}_{\varrho} (\Phi_n)$,
  we have $\|\Psi_n\|_{\mathrm{scaling}} \to \infty$.
\end{theorem}

\begin{proof}
  For the proof of the statement, we refer to Appendix~\ref{app:InverseStability}.
  The proof is based on the fact that the Lipschitz constant of the realization of a network
  essentially yields a lower bound on the $\|\cdot\|_{\mathrm{scaling}}$ norm
  of every neural network with this realization.
  We construct neural networks $\Phi_n$ the realizations of which have small amplitude
  but high Lipschitz constants.
  The associated realizations uniformly converge to $0$,
  but every associated neural network must have exploding weights.
\end{proof}

We finally rephrase the preceding result in more topological terms:

\begin{corollary} \label{cor:Quotient}
  Under the assumptions of Theorem \ref{thm:InverseStability}, the
  realization map $\Realization^{\Omega}_{\varrho}$ from
  Equation \eqref{eq:RealizationMapping} is \emph{not} a quotient map
  when considered as a map onto its range.
\end{corollary}

\begin{proof}
  For the proof of the statement, we refer to Appendix \ref{app:Quotient}.
\end{proof}

\section*{Acknowledgements}
P.P.~and M.R.~were supported by the DFG Collaborative Research Center
TRR 109 ``Discretization in Geometry and Dynamics".
P.P.~was supported by a DFG Research Fellowship
"Shearlet-based energy functionals for anisotropic phase-field methods".
M.R.~is supported by the Berlin Mathematical School.
F.V.~acknowledges support from the European Commission through
DEDALE (contract no.~665044) within the H2020 Framework Program.

We would like to thank Dave L.~Renfro
for bringing the paper \cite{WardStructureOfNonEnumerableSetsOfPoints} to our attention.

\bibliographystyle{abbrv}
\bibliography{references}

\begin{thebibliography}{10}

\bibitem{allen2018convergence}
Z.~Allen-Zhu, Y.~Li, and Z.~Song.
\newblock A convergence theory for deep learning via over-parameterization.
\newblock In {\em Proc. of the 36th International Conference on Machine
  Learning}, volume~97, pages 242--252, 2019.

\bibitem{AmannEscher}
H.~Amann and J.~Escher.
\newblock {\em Analysis {III}}.
\newblock Birkh\"auser Verlag, Basel, 2009.

\bibitem{MR0169048}
P.~M. Anselone and J.~Korevaar.
\newblock Translation invariant subspaces of finite dimension.
\newblock {\em Proc. Amer. Math. Soc.}, 15:747--752, 1964.

\bibitem{AnthonyBartlett}
M.~Anthony and P.~L. Bartlett.
\newblock {\em Neural network learning: theoretical foundations}.
\newblock Cambridge University Press, Cambridge, 1999.

\bibitem{bach2017breaking}
F.~Bach.
\newblock Breaking the curse of dimensionality with convex neural networks.
\newblock {\em J. Mach. Learn. Res.}, 18(1):629--681, 2017.

\bibitem{BaldiHornik}
P.~Baldi and K.~Hornik.
\newblock Neural networks and principal component analysis: Learning from
  examples without local minima.
\newblock {\em Neural Netw.}, 2, 1988.

\bibitem{Barron1993}
A.~Barron.
\newblock Universal approximation bounds for superpositions of a sigmoidal
  function.
\newblock {\em IEEE Trans. Inf. Theory}, 39(3):930--945, 1993.

\bibitem{bartlett2002hardness}
P.~L. Bartlett and S.~Ben-David.
\newblock Hardness results for neural network approximation problems.
\newblock {\em Theor. Comput. Sci.}, 284(1):53--66, 2002.

\bibitem{bartlett2017spectrally}
P.~L. Bartlett, D.~J. Foster, and M.~J. Telgarsky.
\newblock Spectrally-normalized margin bounds for neural networks.
\newblock In {\em Adv. Neural Inf. Process. Syst.}, pages 6240--6249, 2017.

\bibitem{bartlett2019nearly}
P.~L. Bartlett, N.~Harvey, C.~Liaw, and A.~Mehrabian.
\newblock Nearly-tight {VC}-dimension and pseudodimension bounds for piecewise
  linear neural networks.
\newblock {\em J. Mach. Learn. Res.}, 20(63):1--17, 2019.

\bibitem{bartlett2002rademacher}
P.~L. Bartlett and S.~Mendelson.
\newblock Rademacher and {G}aussian complexities: {R}isk bounds and structural
  results.
\newblock {\em J. Mach. Learn. Res.}, 3(Nov):463--482, 2002.

\bibitem{Bergstra+2009}
J.~Bergstra, G.~Desjardins, P.~Lamblin, and Y.~Bengio.
\newblock Quadratic polynomials learn better image features.
\newblock Technical Report 1337, D{\'{e}}partement d'Informatique et de
  Recherche Op{\'{e}}rationnelle, Universit{\'{e}} de Montr{\'{e}}al, Apr.
  2009.

\bibitem{blum1989training}
A.~Blum and R.~Rivest.
\newblock Training a 3-node neural network is {NP}-complete.
\newblock In {\em Adv. Neural Inf. Process. Syst.}, pages 494--501, 1989.

\bibitem{boelcskeiNeural}
H.~B\"olcskei, P.~Grohs, G.~Kutyniok, and P.~C. Petersen.
\newblock Optimal approximation with sparsely connected deep neural networks.
\newblock {\em SIAM J. Math. Data Sci.}, 1:8--45, 2019.

\bibitem{ISRLU}
B.~Carlile, G.~Delamarter, P.~Kinney, A.~Marti, and B.~Whitney.
\newblock Improving deep learning by inverse square root linear units
  {(ISRLUs)}.
\newblock {\em arXiv preprint arXiv:1710.09967}, 2017.

\bibitem{chizat2018global}
L.~Chizat and F.~Bach.
\newblock On the global convergence of gradient descent for over-parameterized
  models using optimal transport.
\newblock In {\em Adv. Neural Inf. Process. Syst}, pages 3036--3046, 2018.

\bibitem{ExpLU}
D.-A. Clevert, T.~Unterthiner, and S.~Hochreiter.
\newblock Fast and accurate deep network learning by exponential linear units
  {(ELUs)}.
\newblock {\em arXiv preprint arXiv:1511.07289}, 2015.

\bibitem{cohen2016expressive}
N.~Cohen, O.~Sharir, and A.~Shashua.
\newblock On the expressive power of deep learning: A tensor analysis.
\newblock In {\em Conference on Learning Theory}, pages 698--728, 2016.

\bibitem{CohnMeasureTheory}
D.~L. Cohn.
\newblock {\em Measure theory}.
\newblock Birkh\"{a}user/Springer, New York, second edition, 2013.

\bibitem{Cucker02onthe}
F.~Cucker and S.~Smale.
\newblock On the mathematical foundations of learning.
\newblock {\em Bull. Am. Math. Soc.}, 39:1--49, 2002.

\bibitem{Cybenko1989}
G.~Cybenko.
\newblock {Approximation by superpositions of a sigmoidal function}.
\newblock {\em Math. Control Signal}, 2(4):303--314, 1989.

\bibitem{dahl2012context}
G.~E. Dahl, D.~Yu, L.~Deng, and A.~Acero.
\newblock Context-dependent pre-trained deep neural networks for
  large-vocabulary speech recognition.
\newblock {\em IEEE Audio, Speech, Language Process.}, 20(1):30--42, 2012.

\bibitem{DieudonneFoundations}
J.~Dieudonn\'{e}.
\newblock {\em Foundations of modern analysis}.
\newblock Pure and Applied Mathematics, Vol. X. Academic Press, New
  York-London, 1960.

\bibitem{weinan2017deep}
W.~E, J.~Han, and A.~Jentzen.
\newblock Deep learning-based numerical methods for high-dimensional parabolic
  partial differential equations and backward stochastic differential
  equations.
\newblock {\em Commun. Math. Stat.}, 5(4):349--380, 2017.

\bibitem{FollandRA}
G.~Folland.
\newblock {\em Real {A}nalysis}.
\newblock Pure and Applied Mathematics (New York). John Wiley \& Sons, Inc.,
  New York, second edition, 1999.

\bibitem{freeman2016topology}
C.~D. Freeman and J.~Bruna.
\newblock Topology and geometry of half-rectified network optimization.
\newblock In {\em 5th International Conference on Learning Representations,
  {ICLR} 2017, Toulon, France, April 24-26, 2017, Conference Track
  Proceedings}, 2017.

\bibitem{Girosi1990}
F.~Girosi and T.~Poggio.
\newblock Networks and the best approximation property.
\newblock {\em Biol. Cybern.}, 63(3):169--176, Jul 1990.

\bibitem{GlorotBordesBengio}
X.~Glorot, A.~Bordes, and Y.~Bengio.
\newblock Deep sparse rectifier neural networks.
\newblock In G.~Gordon, D.~Dunson, and M.~Dudík, editors, {\em Proceedings of
  the Fourteenth International Conference on Artificial Intelligence and
  Statistics}, volume~15 of {\em Proceedings of Machine Learning Research},
  pages 315--323, Fort Lauderdale, FL, USA, 11--13 Apr. 2011.

\bibitem{Goodfellow-et-al-2016}
I.~Goodfellow, Y.~Bengio, and A.~Courville.
\newblock {\em Deep Learning}.
\newblock MIT Press, 2016.
\newblock \url{http://www.deeplearningbook.org}.

\bibitem{GrafakosClassical}
L.~Grafakos.
\newblock {\em Classical {F}ourier analysis}, volume 249 of {\em Graduate Texts
  in Mathematics}.
\newblock Springer, New York, second edition, 2008.

\bibitem{HaykinLogFunction}
S.~Haykin.
\newblock {\em Neural Networks: A Comprehensive Foundation}.
\newblock Prentice Hall PTR, Upper Saddle River, NJ, USA, 2nd edition, 1998.

\bibitem{ParReLu}
K.~He, X.~Zhang, S.~Ren, and J.~Sun.
\newblock Delving deep into rectifiers: Surpassing human-level performance on
  imagenet classification.
\newblock In {\em Proc. of the IEEE international conference on computer
  vision}, pages 1026--1034, 2015.

\bibitem{hinton2012deep}
G.~Hinton, L.~Deng, D.~Yu, G.~E. Dahl, A.-R. Mohamed, N.~Jaitly, A.~Senior,
  V.~Vanhoucke, P.~Nguyen, T.~N. Sainath, et~al.
\newblock Deep neural networks for acoustic modeling in speech recognition: The
  shared views of four research groups.
\newblock {\em IEEE Signal Process. Mag.}, 29(6):82--97, 2012.

\bibitem{Hornik1989universalApprox}
K.~Hornik, M.~Stinchcombe, and H.~White.
\newblock Multilayer feedforward networks are universal approximators.
\newblock {\em Neural Netw.}, 2(5):359--366, 1989.

\bibitem{Huang2017DenselyCC}
G.~Huang, Z.~Liu, L.~van~der Maaten, and K.~Q. Weinberger.
\newblock Densely connected convolutional networks.
\newblock {\em 2017 IEEE Conference on Computer Vision and Pattern Recognition
  (CVPR)}, pages 2261--2269, 2017.

\bibitem{jacot2018neural}
A.~Jacot, F.~Gabriel, and C.~Hongler.
\newblock Neural tangent kernel: Convergence and generalization in neural
  networks.
\newblock In {\em Adv. Neural Inf. Process. Syst.}, pages 8571--8580, 2018.

\bibitem{judd1987learning}
J.~Judd.
\newblock Learning in networks is hard.
\newblock In {\em Proc. of IEEE International Conference on Neural Networks,
  1987}, volume~2, pages 685--692, 1987.

\bibitem{kainen2000best}
P.~Kainen, V.~Kurkov{\'a}, and A.~Vogt.
\newblock Best approximation by {H}eaviside perceptron networks.
\newblock {\em Neural Netw.}, 13(7):695--697, 2000.

\bibitem{kainen1999approximation}
P.~C. Kainen, V.~Kurkov{\'a}, and A.~Vogt.
\newblock Approximation by neural networks is not continuous.
\newblock {\em Neurocomputing}, 29(1-3):47--56, 1999.

\bibitem{Krizhevsky2012Imagenet}
A.~Krizhevsky, I.~Sutskever, and G.~Hinton.
\newblock Imagenet classification with deep convolutional neural networks.
\newblock In {\em Adv. Neural Inf. Process. Syst. 25}, pages 1097--1105. Curran
  Associates, Inc., 2012.

\bibitem{lagaris1998artificial}
I.~E. Lagaris, A.~Likas, and D.~I. Fotiadis.
\newblock Artificial neural networks for solving ordinary and partial
  differential equations.
\newblock {\em IEEE Trans. Neural Netw.}, 9(5):987--1000, 1998.

\bibitem{LeCun2015DeepLearning}
Y.~LeCun, Y.~Bengio, and G.~Hinton.
\newblock {Deep learning}.
\newblock {\em Nature}, 521(7553):436--444, 2015.

\bibitem{LeeTopologicalManifolds}
J.~M. Lee.
\newblock {\em Introduction to topological manifolds}, volume 202 of {\em
  Graduate Texts in Mathematics}.
\newblock Springer, New York, second edition, 2011.

\bibitem{PinkusUniversalApproximation}
M.~Leshno, V.~Y. Lin, A.~Pinkus, and S.~Schocken.
\newblock Multilayer feedforward networks with a nonpolynomial activation
  function can approximate any function.
\newblock {\em Neural Netw.}, 6(6):861--867, 1993.

\bibitem{ArctanLiao}
B.~Liao, C.~Ma, L.~Xiao, R.~Lu, and L.~Ding.
\newblock An arctan-activated {WASD} neural network approach to the prediction
  of {D}ow {J}ones {I}ndustrial {A}verage.
\newblock In {\em Advances in Neural Networks - {ISNN} 2017 - 14th
  International Symposium, {ISNN} 2017, Sapporo, Hakodate, and Muroran,
  Hokkaido, Japan, June 21-26, 2017, Proceedings, Part {I}}, pages 120--126,
  2017.

\bibitem{MaasTanh}
A.~Maas, Y.~Hannun, and A.~Ng.
\newblock Rectifier nonlinearities improve neural network acoustic models.
\newblock In {\em ICML Workshop on Deep Learning for Audio, Speech and Language
  Processing}, 2013.

\bibitem{MaiorovRidgeFunctionsBestApproximation}
V.~E. Maiorov.
\newblock Best approximation by ridge functions in {$L_p$}-spaces.
\newblock {\em Ukra\"{\i}n. Mat. Zh.}, 62(3):396--408, 2010.

\bibitem{MP43}
W.~McCulloch and W.~Pitts.
\newblock A logical calculus of ideas immanent in nervous activity.
\newblock {\em Bull. Math. Biophys.}, 5:115--133, 1943.

\bibitem{mei2018mean}
S.~Mei, A.~Montanari, and P.-M. Nguyen.
\newblock A mean field view of the landscape of two-layer neural networks.
\newblock {\em Proc. Natl. Acad. Sci. USA}, 115(33):E7665--E7671, 2018.

\bibitem{Mhaskar:1996:NNO:1362203.1362213}
H.~Mhaskar.
\newblock Neural networks for optimal approximation of smooth and analytic
  functions.
\newblock {\em Neural Comput.}, 8(1):164--177, 1996.

\bibitem{Mhaskar1993}
H.~N. Mhaskar.
\newblock Approximation properties of a multilayered feedforward artificial
  neural network.
\newblock {\em Adv. Comput. Math.}, 1(1):61--80, Feb. 1993.

\bibitem{MohriFoundations}
M.~Mohri, A.~Rostamizadeh, and A.~Talwalkar.
\newblock {\em Foundations of machine learning}.
\newblock Adaptive Computation and Machine Learning. MIT Press, Cambridge, MA,
  2012.

\bibitem{Montufar:2014:NLR:2969033.2969153}
G.~Mont\'{u}far, R.~Pascanu, K.~Cho, and Y.~Bengio.
\newblock On the number of linear regions of deep neural networks.
\newblock In {\em Proc. of the 27th International Conference on Neural
  Information Processing Systems}, NIPS'14, pages 2924--2932, Cambridge, MA,
  USA, 2014. MIT Press.

\bibitem{NairHinton}
V.~Nair and G.~Hinton.
\newblock Rectified linear units improve restricted {B}oltzmann machines.
\newblock In {\em Proc. of the 27th International Conference on Machine
  Learning}, ICML'10, pages 807--814, USA, 2010. Omnipress.

\bibitem{nguyen2017loss}
Q.~Nguyen and M.~Hein.
\newblock The loss surface of deep and wide neural networks.
\newblock In {\em Proc. of the 34th International Conference on Machine
  Learning-Volume 70}, pages 2603--2612. JMLR. org, 2017.

\bibitem{NguyenHein17}
Q.~Nguyen and M.~Hein.
\newblock The loss surface of deep and wide neural networks.
\newblock {\em arXiv preprint arXiv:1704.08045}, 2017.

\bibitem{PetV2018OptApproxReLU}
P.~Petersen and F.~Voigtlaender.
\newblock Optimal approximation of piecewise smooth functions using deep {ReLU}
  neural networks.
\newblock {\em Neural Netw.}, 108:296--330, 2018.

\bibitem{StackexchangeTwoSidedLimitPoint}
{\relax PhoemueX (\url{https://math.stackexchange.com/users/151552/phoemuex})}.
\newblock Uncountable closed set {$A$}, existence of point at which {$A$}
  accumulates "from two sides" of a hyperplane.
\newblock Mathematics Stack Exchange.
\newblock URL:\url{https://math.stackexchange.com/q/3513692} (version:
  2020-01-18).

\bibitem{Rotskoff2018NeuralNA}
G.~M. Rotskoff and E.~Vanden-Eijnden.
\newblock Neural networks as interacting particle systems: Asymptotic convexity
  of the loss landscape and universal scaling of the approximation error.
\newblock {\em ArXiv}, abs/1805.00915, 2018.

\bibitem{Rudin}
W.~Rudin.
\newblock {\em Real and complex analysis}.
\newblock McGraw-Hill Book Co., New York, third edition, 1987.

\bibitem{RudinFunkana}
W.~Rudin.
\newblock {\em Functional analysis}.
\newblock International Series in Pure and Applied Mathematics. McGraw-Hill,
  Inc., New York, second edition, 1991.

\bibitem{pmlr-v70-safran17a}
I.~Safran and O.~Shamir.
\newblock Depth-width tradeoffs in approximating natural functions with neural
  networks.
\newblock In {\em Proceedings of the 34th International Conference on Machine
  Learning}, volume~70 of {\em Proceedings of Machine Learning Research}, pages
  2979--2987, 2017.

\bibitem{schmidhuber2015deep}
J.~Schmidhuber.
\newblock Deep learning in neural networks: An overview.
\newblock {\em Neural Netw.}, 61:85--117, 2015.

\bibitem{silver2017mastering}
D.~Silver, T.~Hubert, J.~Schrittwieser, I.~Antonoglou, M.~Lai, A.~Guez,
  M.~Lanctot, L.~Sifre, D.~Kumaran, T.~Graepel, et~al.
\newblock Mastering chess and shogi by self-play with a general reinforcement
  learning algorithm.
\newblock {\em arXiv preprint arXiv:1712.01815}, 2017.

\bibitem{simonyan2014very}
K.~Simonyan and A.~Zisserman.
\newblock Very deep convolutional networks for large-scale image recognition.
\newblock {\em arXiv preprint arXiv:1409.1556}, 2014.

\bibitem{usunier2016episodic}
N.~Usunier, G.~Synnaeve, Z.~Lin, and S.~Chintala.
\newblock Episodic {E}xploration for {D}eep {D}eterministic {P}olicies for
  {S}tar{C}raft {M}icromanagement.
\newblock In {\em 5th International Conference on Learning Representations,
  {ICLR} 2017, Toulon, France, April 24-26, 2017, Conference Track
  Proceedings}, 2017.

\bibitem{venturi2018neural}
L.~Venturi, A.~S. Bandeira, and J.~Bruna.
\newblock Neural networks with finite intrinsic dimension have no spurious
  valleys.
\newblock {\em arXiv}, 1802.06384, 2018.

\bibitem{WardStructureOfNonEnumerableSetsOfPoints}
A.~J. Ward.
\newblock The {S}tructure of {N}on-{E}numerable {S}ets of {P}oints.
\newblock {\em J. London Math. Soc.}, 8(2):109--112, 1933.

\bibitem{weinan2018deep}
E.~Weinan and B.~Yu.
\newblock The deep {R}itz method: a deep learning-based numerical algorithm for
  solving variational problems.
\newblock {\em Communications in Mathematics and Statistics}, 6(1):1--12, 2018.

\bibitem{wu2016stimulated}
C.~Wu, P.~Karanasou, M.~J. Gales, and K.~C. Sim.
\newblock Stimulated deep neural network for speech recognition.
\newblock Technical report, University of Cambridge, 2016.

\bibitem{yannakakis2017artificial}
G.~N. Yannakakis and J.~Togelius.
\newblock {\em Artificial Intelligence and Games}.
\newblock Springer, 2017.

\bibitem{YAROTSKY2017103}
D.~Yarotsky.
\newblock Error bounds for approximations with deep {ReLU} networks.
\newblock {\em Neural Netw.}, 94:103--114, 2017.

\bibitem{YarotskyPhaseDiagram}
D.~Yarotsky and A.~Zhevnerchuk.
\newblock The phase diagram of approximation rates for deep neural networks.
\newblock {\em arXiv preprint arXiv:1906.09477}, 2019.

\bibitem{zhang2017convexified}
Y.~Zhang, P.~Liang, and M.~J. Wainwright.
\newblock Convexified convolutional neural networks.
\newblock In {\em Proc. of the 34th International Conference on Machine
  Learning-Volume 70}, pages 4044--4053. JMLR. org, 2017.

\end{thebibliography}

\appendix
\section{Notation}
\label{sub:Notation}

The symbol $\N$ denotes the \textbf{natural numbers} $\N = \{1,2,3,\dots\}$,
whereas $\N_0 = \{ 0 \} \cup \N$ stands for the natural numbers including zero.
Moreover, we set $\N_{\geq d} \coloneqq \{ n \in \N \colon n \geq d \}$ for $d \in \N$.
The number of elements of a set $M$ will be denoted by $|M| \in \N_0 \cup \{\infty\}$.
Furthermore, we write $\underline{n} \coloneqq \{k \in \N \,:\, k \leq n\}$ for $n \in \N_0$.
In particular, $\underline{0} = \emptyset$.

For two sets $A,B$, a map $f : A \to B$, and $C \subset A$, we write $f|_{C}$
for the \textbf{restriction of} $f$ \textbf{to} $C$.
For a set $A$, we denote by $\chi_A = \Indicator_A$ the
\textbf{indicator function of} $A$, so that $\chi_A (x) = 1$ if $x \in A$
and $\chi_A (x) = 0$ otherwise. For any $\R$-vector space $\mathcal{Y}$
we write $A + B \coloneqq \{ a + b \,:\, a \in A, b \in B \}$ and
$\lambda A \coloneqq \{\lambda a \,:\, a \in A\}$,
for $\lambda \in \R$ and subsets $A,B \subset \mathcal{Y}$.

The \textbf{algebraic dual space} of a $\mathbb{K}$-vector space $\mathcal{Y}$
(with $\mathbb{K} = \R$ or $\mathbb{K} = \CC$),
that is the space of all linear functions $\varphi: \mathcal{Y} \to \mathbb{K}$,
will be denoted by $\mathcal{Y}^\ast$.
In contrast, if $\mathcal{Y}$ is a \emph{topological} vector space, we denote by
$\mathcal{Y}'$ the \textbf{topological dual space} of $\mathcal{Y}$,
which consists of all functions $\varphi \in \mathcal{Y}^\ast$ that are continuous.

Given functions $(f_i)_{i \in \underline{n}}$ with $f_i : X_i \to Y_i$, we consider three
different types of products between these maps:
The \textbf{cartesian product} of $f_1, \dots, f_n$ is
\[
  f_1 \times \cdots \times f_n :
  X_1 \times \cdots \times X_n \to Y_1 \times \dots \times Y_n, \quad
  (x_1, \dots, x_n) \mapsto \big( f_1 (x_1), \dots, f_n (x_n) \big) \, .
\]
The \textbf{tensor product} of $f_1, \dots, f_n$ is defined if $Y_1, \dots, Y_n \subset \CC$,
and is then given by
\[
  f_1 \otimes \cdots \otimes f_n :
  X_1 \times \cdots \times X_n \to \CC,\quad
  (x_1,\dots,x_n) \mapsto f_1(x_1) \cdots f_n (x_n) \, .
\]
Finally, the \textbf{direct sum}
of $f_1, \dots, f_n$ is defined if $X_1 = \dots = X_n$, and given by
\[
  f_1 \oplus \cdots \oplus f_n :
  X_1 \to Y_1 \times \cdots \times Y_n, \quad
  x \mapsto \big( f_1 (x), \dots, f_n (x) \big) \, .
\]

\medskip{}

The \textbf{closure} of a subset $A$ of a topological space
will be denoted by $\overline{A}$, while the \textbf{interior} of $A$ is
denoted by $A^\circ$.
For a metric space $(\mathcal{U}, d)$,
we write $B_\eps(x)\coloneqq \{y \in \mathcal{U}: d(x,y)< \eps \}$ for the
$\eps$-\textbf{ball around} $x$, where $x \in \mathcal{U}$ and $\eps > 0$.
Furthermore, for a Lipschitz continuous function
$f : \mathcal{U}_1 \to \mathcal{U}_2$ between two metric spaces
$\mathcal{U}_1$ and $\mathcal{U}_2$, we denote by $\mathrm{Lip}(f)$
the smallest possible \textbf{Lipschitz constant} for $f$.

\medskip{}

For $d \in \N$ and a function $f: A \to \R^{d}$ or a vector
$v \in \R^{d}$, we denote for $j\in \{1, \dots, d\}$ the
$j$\textbf{-th component} of $f$ or $v$ by $(f)_j$ or $v_j$, respectively.
As an example, the \textbf{Euclidean scalar product} on $\R^d$
is given by $\langle x,y \rangle = \sum_{i=1}^d x_i \, y_i$.
We denote the \textbf{Euclidean norm} by $|x| := \sqrt{\langle x,x \rangle}$ for $x \in \R^d$.
For a matrix $A \in \R^{n \times d}$, let $\| A \|_{\max}
  \coloneqq \max_{i = 1,\dots,n} \,\,
       \max_{j = 1,\dots,d}
         | A_{i,j} |$.
The \textbf{transpose} of a matrix $A \in \R^{n \times d}$ will be
denoted by $A^T \in \R^{d \times n}$.
For $A \in \R^{n \times d}$, $i \in \{1,\dots,n\}$ and $j \in \{1,\dots,d\}$,
we denote by $A_{i,-} \in \R^d$ the $i$-th row of $A$
and by $A_{-,j} \in \R^n$ the $j$-th column of $A$.
The Euclidean \textbf{unit sphere} in $\R^d$ will be denoted by $S^{d-1} \subset \R^d$.

For $n \in \N$ and $\emptyset \neq \Omega \subset \R^d$, we denote by $C(\Omega; \R^n)$ the
space of all \textbf{continuous functions defined on $\Omega$ with values in $\R^n$}.
If $\Omega$ is compact, then $(C(\Omega; \R^n),\|\cdot\|_{\sup})$ denotes the Banach space
of $\R^n$- valued continuous functions equipped with the supremum norm, where we use the
Euclidean norm on $\R^n$.
If $n = 1$, then we shorten the notation to $C(\Omega)$.

We note that on $C(\Omega)$, the supremum norm coincides with the $L^\infty(\Omega)$-norm,
if for all $x\in \Omega$ and for all $\eps > 0$ we have that
$\lambda(\Omega\cap B_{\eps}(x))>0,$ where $\lambda$ denotes the Lebesgue measure
on $\R^d$.
For any nonempty set $U \subset \R$, we
say that a function $f : U \to \R$ is \textbf{increasing}
if $f(x) \leq f(y)$ for every $x,y \in U$ with $x < y$.
If even $f(x) < f(y)$ for all such $x,y$,
we say that $f$ is \textbf{strictly increasing}.
The terms ``decreasing'' and ``strictly decreasing'' are defined analogously.

The \textbf{Schwartz space} will be denoted by $\mathcal{S}(\R^d)$ and the
space of \textbf{tempered distributions} by $\mathcal{S}'(\R^d)$.
The associated bilinear \textbf{dual pairing} will be denoted by
$\langle \cdot,\cdot \rangle_{\mathcal{S}',\mathcal{S}}$.
We refer to \mbox{\cite[Sections 8.1--8.3 and 9.2]{FollandRA}} for more details
on the spaces $\mathcal{S}(\R^d)$ and $\mathcal{S}'(\R^d)$.
Finally, the \textbf{Dirac delta distribution}
$\delta_x$ at $x \in \R^d$ is given by
$\delta_x : C(\R^d) \to \R, f \mapsto f(x)$.

\section{Auxiliary results: Operations with neural networks}
\label{sec:auxResult}

This part of the appendix is devoted to auxiliary results that are connected
with basic operations one can perform with neural networks and which we will frequently make use of
in the proofs below.

We start by showing that one can ``enlarge'' a given neural network
in such a way that the realizations of the original network and the enlarged network coincide.
To be more precise, the following holds:

\begin{lemma}\label{lem:enlarge}
Let $d, L\in \N$, $\Omega\subset \R^d$, and $\varrho:\R\to \R$.
Moreover, let $\Phi = \big( (A_1, b_1), \dots, (A_L, b_L) \big)$ be a neural network with
architecture $(d, N_1, \dots, N_L)$ and let
$\widetilde{N}_1, \dots, \widetilde{N}_{L-1} \in \N$ such that
$\widetilde{N}_\ell\geq N_\ell$ for all $\ell = 1, \dots, L-1$.
Then, there exists a neural network $\widetilde{\Phi}$ with architecture
$(d,\widetilde{N}_1,\dots,\widetilde{N}_{L-1},N_L)$ and such that
$\Realization^{\Omega}_{\varrho}(\Phi)=\Realization^{\Omega}_{\varrho}(\widetilde{\Phi})$.
%
%
\end{lemma}

\begin{proof}
Setting $N_0 \coloneqq \widetilde{N}_0 \coloneqq d$, and $\widetilde{N}_L \coloneqq N_L$,
we define
\(
  \widetilde{\Phi}
  \coloneqq \left(
              (\widetilde{A}_1,\widetilde{b}_1),
              \dots,
              (\widetilde{A}_L,\widetilde{b}_L)
            \right)
\)
by
\begin{align*}
  \widetilde{A}_\ell
  \coloneqq \left(
       \begin{array}{l l}
         A_\ell
         & 0_{N_\ell \times (\widetilde{N}_{\ell-1}-N_{\ell-1})} \\
         0_{(\widetilde{N}_\ell-N_\ell)\times N_{\ell-1} }
         & 0_{
              (\widetilde{N}_\ell-N_\ell)
              \times (\widetilde{N}_{\ell-1}- N_{\ell-1})
             }
       \end{array}
     \right)
  \in \R^{\widetilde{N}_\ell\times \widetilde{N}_{\ell-1}}
  \qquad \text{and} \qquad
  \widetilde{b}_\ell
  \coloneqq \begin{pmatrix}
              b_\ell \\ 0_{\widetilde{N}_\ell-N_\ell}
            \end{pmatrix}
     \in \R^{\widetilde{N}_\ell},
\end{align*}
for $\ell = 1, \dots, L$.
Here, $0_{m_1 \times m_2}$ and $0_k$ denote the zero-matrix in $\R^{m_1 \times m_2}$
and the zero vector in $\R^k$, respectively.
Clearly, $\Realization_\varrho^\Omega(\widetilde{\Phi}) = \Realization_\varrho^\Omega(\Phi)$.
This yields the claim.
\end{proof}

Another operation that we can perform with networks is \emph{concatenation},
as given in the following definition.

\begin{definition}\label{def:Concatenation}
Let $L_1, L_2 \in \N$ and let $\vphantom{\sum_j}
\Phi^1 = \big( (A_1^1,b_1^1), \dots, (A_{L_1}^1,b_{L_1}^1) \big),
\Phi^2 = \big((A_1^2,b_1^2), \dots, (A_{L_2}^2,b_{L_2}^2) \big)$
be two neural networks such that the input layer of $\Phi^1$ has the same
dimension as the output layer of $\Phi^2$.
Then, $\Phi^1 \conc \Phi^2$ denotes the following $L_1+L_2-1$ layer network:
\[
  \Phi^1 \conc \Phi^2
  \coloneqq \big(
       (A_1^2,b_1^2),
       \dots,
       (A_{L_2-1}^2,b_{L_2-1}^2),
       (A_{1}^1 A_{L_2}^2, A_{1}^1 b^2_{L_2} + b_1^1),
       ({A}_{2}^1, b_{2}^1),
       \dots,
       (A_{L_1}^1, b_{{L_1}}^1)
     \big).
\]
Then, we call $\Phi^1 \conc \Phi^2$ the
\textbf{concatenation of} $\Phi^1$ \textbf{and} $\Phi^2$.
\end{definition}

One directly verifies that for every $\varrho: \R \to \R$ the definition of
concatenation is reasonable, that is, if $d_i$ is the dimension of the input
layer of $\Phi^i$, $i = 1,2$, and if $\Omega \subset \R^{d_2}$, then
$\Realization_\varrho^\Omega(\Phi^1 \conc \Phi^2)
= \Realization_\varrho^{\R^{d_1}}(\Phi^1) \circ \Realization_\varrho^\Omega(\Phi^2)$.
If $\Phi^2$ has architecture $(d, N_1, \dots, N_{L_2})$ and $\Phi^1$
has architecture $(N_{L_2}, \widetilde{N}_{1},\dots,\widetilde{N}_{L_1-1}, \widetilde{N}_{L_1})$,
then $\Phi^1 \conc \Phi^2$ has architecture
$(d, N_1, \dots, N_{L_2 - 1},  \widetilde{N}_{1}, \dots, \widetilde{N}_{L_1})$.
Therefore, $N(\Phi^1 \conc \Phi^2) = N(\Phi^1) + N(\Phi^2) - 2 N_{L_2}$.

We close this section by showing that under mild assumptions on $\varrho$---which are always
satisfied in practice---and on the network architecture,
one can construct a neural network which locally approximates the identity mapping
$\mathrm{id}_{\R^d}$ to arbitrary accuracy.
Similarly, one can obtain a neural network the realization of which
approximates the projection onto the $i$-th coordinate.
The main ingredient of the proof is the approximation
\(
  x \approx \frac{\varrho(x_0 + x) - \varrho(x_0)}{\varrho'(x_0)} ,
\)
which holds for $|x|$ small enough and where $x_0$ is chosen such that $\varrho'(x_0) \neq 0$.

\begin{proposition}\label{prop:Identity}
Let $\varrho : \R \to \R$ be continuous, and assume that there
exists $x_0 \in \R$ such that $\varrho$ is differentiable at $x_0$ with $\varrho' (x_0) \neq 0$.
Then, for every $\eps > 0, d \in \N, B > 0$ and every $L\in \N$ there exists a neural network
$\Phi_\eps^B\in \mathcal{NN}((d,d, \dots, d))$ with $L$ layers 
such that
\begin{itemize}
  \item $\left|\Realization_{\varrho}^{[-B,B]^d}(\Phi_{\eps}^B)(x)-x\right|
        \leq \eps$ for all $x\in [-B,B]^d$;

  \item $\Realization_{\varrho}^{[-B,B]^d}({\Phi}_{\eps}^B)(0)=0$;




  \item $\Realization_\varrho^{[-B,B]^d} (\Phi_\eps^B)$ is totally differentiable at $x = 0$
        with Jacobian matrix
        \(
          D \big( \Realization_\varrho^{[-B,B]^d} (\Phi_\eps^B) \big) (0)
          = \mathrm{id}_{\R^d};
        \)

  \item for $j \in \{1, \dots, d\}$,
        $\left(\Realization_{\varrho}^{[-B,B]^d}({\Phi}_{\eps}^B)\right)_j$
        is constant in all but the $j$-th coordinate.
\end{itemize}

Furthermore, for every $d, L \in \N$, $\eps > 0$, $B > 0$ and every
$i\in \{1,\dots,d\}$, one can construct a neural network
$\widetilde{\Phi}_{\eps,i}^B \in \mathcal{NN}((d,1, \dots, 1))$ with $L$ layers
such that
\begin{itemize}
  \item $\left|
           \Realization_{\varrho}^{[-B,B]^d}(\widetilde{\Phi}_{\eps,i}^B)(x)
           -x_i
         \right|\leq \eps$ for all $x\in [-B,B]^d$;

  \item $\Realization_{\varrho}^{[-B,B]^d}(\widetilde{\Phi}_{\eps,i}^B)(0)=0$;


  \item $\Realization_\varrho^{[-B,B]^d} (\Phi_{\eps,i}^B)$ is partially differentiable at $x = 0$,
        with
        \(
          \frac{\partial}{\partial x_i}\Big|_{x=0}
          \Realization_{\varrho}^{[-B,B]^d}(\widetilde{\Phi}_{\eps,i}^B)(x) = 1
        \);
        and


  \item $\Realization_{\varrho}^{[-B,B]^d}(\widetilde{\Phi}_{\eps,i}^B)$
        is constant in all but the $i$-th coordinate.
\end{itemize}

Finally, if $\varrho$ is increasing, then
$\big(\Realization_{\varrho}^{[-B,B]^d}({\Phi}_{\eps}^B)\big)_j$ and
$\Realization_{\varrho}^{[-B,B]^d}(\widetilde{\Phi}_{\eps,i}^B)$ are
monotonically increasing in every coordinate and for all $j\in \{1, \dots, d\}$.
\end{proposition}

\begin{proof}
We first consider the special case $L = 1$.
Here, we can take $\Phi_\eps^B \coloneqq ( (\mathrm{id}_{\R^d}, 0) )$
and ${\Phi_{\eps,i}^B \coloneqq ( (e_i ,0) )}$, with  $e_i \in \R^{1 \times d}$ denoting the
$i$-th standard basis vector in ${\R^d \cong \R^{1 \times d}}$.
In this case, $\Realization_\varrho^{[-B,B]^d} (\Phi_\eps^B) = \identity_{\R^d}$
and $\Realization_\varrho^{[-B,B]^d}(\Phi_{\eps,i}^B) (x) = x_i$ for all $x \in [-B,B]^d$,
which implies that all claimed properties are satisfied.
Thus, we can assume in the following that $L \geq 2$.

Without loss of generality, we only consider the case $\eps \leq 1$.
Define $\eps' \coloneqq \eps / (dL)$.
Let $x_0 \in \R$ be such that $\varrho$ is differentiable at $x_0$ with $\varrho' (x_0) \neq 0$.
We set $r_0 \coloneqq \varrho(x_0)$ and $s_0\coloneqq \varrho'(x_0)$.
Next, for $C > 0$, we define
\[
  \varrho_C : [-B-L\eps, B+L\eps] \to \R, \quad
  x \mapsto \frac{C}{s_0} \cdot \varrho\left(\frac{x}{C}
            + x_0\right)- \frac{C r_0}{s_0}.
\]
We claim that there is some $C_0 > 0$ such that $|\varrho_C(x) - x| \leq \eps'$
for all $x \in [-B-L\eps,B+L\eps]$ and all $C \geq C_0$. To see this, first note
by definition of the derivative that there is some $\delta > 0$ with
\[
  |\varrho(t + x_0) - r_0 - s_0 t|
  \leq \frac{|s_0| \cdot \eps'}{1+B+L} \cdot |t|
  \qquad \text{ for all } t \in \R \text{ with } | t | \leq \delta.
\]
Here we implicitly used that $s_0 = \varrho'(x_0) \neq 0$ to ensure that
the right-hand side is a \emph{positive} multiple of $|t|$.
Now, set $C_0 \coloneqq (B+L)/\delta$, and let $C \geq C_0$ be arbitrary.
Note because of $\eps' \leq \eps \leq 1$ that every $x \in [-B-L\eps,B+L\eps]$
satisfies $| x | \leq B+L$.
Hence, if we set $t \coloneqq x/C$, then $| t | \leq \delta$. Therefore,
\[
  | \varrho_C (x) - x |
  = \left| \frac{C}{s_0} \right|
    \cdot \big| \varrho(t + x_0) - r_0 - s_0 t \big|
  \leq \left| \frac{C}{s_0} \right|
       \cdot \frac{| s_0 | \cdot \eps'}{1+B+L}
       \cdot \left| \frac{x}{C} \right|
  \leq \eps' .
\]
Note that $\varrho_C$ is differentiable at $0$
with derivative $\varrho_C ' (0) = \frac{C}{s_0} \varrho' (x_0) \frac{1}{C} = 1$,
thanks to the chain rule.

\medskip{}

Using these preliminary observations, we now construct the neural networks $\Phi_\eps^B$
and $\Phi_{\eps,i}^B$.
Define ${\Phi_0^C \coloneqq \big( (A_1, b_1),(A_2,b_2) \big)}$, where
\[
  A_1 \coloneqq \frac{1}{C} \cdot \mathrm{id}_{\R^d} \in \R^{d \times d},
  \quad
  b_1 \coloneqq x_0 \cdot (1, \dots, 1)^T \in \R^d,
  \quad
  A_2 \coloneqq \frac{C}{s_0} \cdot \mathrm{id}_{\R^d} \in \R^{d \times d},
  \quad
  b_2 \coloneqq - \frac{C r_0}{s_0} \cdot (1,\dots,1)^T \in \R^d.
\]

%
Note $\Phi_0^C \in \cN((d,d,d))$.
To shorten the notation, let $\Omega\coloneqq [-B, B]^d$ and $J = [-B, B]$.
It is not hard to see that
${\Realization_{\varrho}^{\Omega}(\Phi_0^C) = \varrho_C|_J \times \cdots \times \varrho_C|_J}$,
where the cartesian product has $d$ factors.
We define ${\Phi_C \coloneqq \Phi_0^C \conc \Phi_0^C \conc \cdots \conc \Phi_0^C}$,
where we take $L-2$ concatenations (meaning $L-1$ factors, so that $\Phi_C = \Phi_0^C$ if $L = 2$).
We obtain $\Phi_C \in \cN ((d,\dots,d))$ (with $L$ layers) and 
\begin{equation}
  \Realization_{\varrho}^{\Omega}(\Phi_C)(x)
  = \big(
      \varrho_C \circ \varrho_C \circ \dots \circ \varrho_C(x_i)
    \big)_{i = 1, \dots, d}
  \quad \text{for all} \quad x \in \Omega ,
  \label{eq:IdentityTensorProductStructure}
\end{equation}
where $\varrho_C$ is applied $L-1$ times.

Since $|\varrho_C (x) -x | \leq \eps' \leq \eps$
for all $x \in [-B-L\eps,B+L\eps]$, it is not hard to see by induction that
\[
  |(\varrho_C \circ \cdots \circ \varrho_C) (x) - x|
  \leq t \cdot \eps' \leq t \cdot \eps
  \qquad \text{ for all } \, x \in [-B, B] ,
\]
where $\varrho_C$ is applied $t \leq L$ times.
Therefore, since $\eps' = \eps / (dL)$, we conclude for $C \geq C_0$ that
\[
  \left|\Realization_{\varrho}^{\Omega}(\Phi_{C})(x) - x \right| \leq \eps
  \quad \text{for all} \quad x \in \Omega .
\]
As we saw above, $\varrho_C$ is differentiable at $0$
with $\varrho_{C}(0) = 0$ and $\varrho_{C}'(0)=1$.
By induction, we thus get
$\frac{d}{dx}\big|_{x=0} (\varrho_C \circ \cdots \circ \varrho_C)(x) = 1$,
where the composition has at most $L$ factors.
Thanks to Equation~\eqref{eq:IdentityTensorProductStructure}, this shows that
$\Realization_\varrho^\Omega (\Phi_C)$ is totally differentiable at $0$, with
$D (\Realization_\varrho^\Omega (\Phi_C)) (0) = \mathrm{id}_{\R^d}$, as claimed.

Also by Equation \eqref{eq:IdentityTensorProductStructure}, we see
that for every $j \in \{1, \dots, d\}$,
$\big( \Realization_{\varrho}^{\Omega}(\Phi_{C})(x) \big)_j$ is constant in all
but the $j$-th coordinate.
Additionally, if $\varrho$ is increasing, then $s_0 > 0$, so that $\varrho_C$ is also increasing,
and hence $\big( \Realization_{\varrho}^{\Omega}(\Phi_{C}) \big)_j$ is increasing
in the $j$-th coordinate, since compositions of increasing functions are increasing.
Hence, $\Phi_\eps^B\coloneqq \Phi_{C}$ satisfies the desired properties.

\medskip{}

We proceed with the second part of the proposition.
We first prove the statement for $i = 1$.
Let $\widetilde{\Phi}_1^C \coloneqq \big( (A_1', b_1'),(A_2',b_2') \big)$, where
\begin{align*}
    A_1' \coloneqq \left(
              \begin{array}{cccc}
                \frac{1}{C} & 0 & \cdots & 0
              \end{array}
            \right) \in \R^{1 \times d},
    \quad
    b_1' \coloneqq x_0 \in \R^1 ,
    \quad
    A_2' \coloneqq \frac{C}{s_0} \in \R^{1 \times 1},
    \quad
    b_2' \coloneqq - \frac{C r_0}{s_0} \in \R^1.
\end{align*}
We have $\widetilde{\Phi}_1^C \in \cN((d, 1, 1))$.
Next, define $\widetilde{\Phi}_2^C \coloneqq \big( (A_1'', b_1''),(A_2'',b_2'') \big)$, where
\begin{align*}
  A_1'' \coloneqq \frac{1}{C} \in \R^{1 \times 1},
  \quad
  b_1'' \coloneqq x_0 \in \R^1,
  \quad
  A_2'' \coloneqq \frac{C}{s_0} \in \R^{1 \times 1},
  \quad
  b_2'' \coloneqq - \frac{C r_0}{s_0} \in \R^1 .
\end{align*}
We have $\widetilde{\Phi}_2^C \in \cN ((1, 1, 1))$. 
Setting $\widetilde{\Phi}_C \coloneqq \widetilde{\Phi}_2^C \conc \dots \conc
\widetilde{\Phi}_2^C \conc \widetilde{\Phi}_1^C$,
where we take $L-2$ concatenations (meaning $L-1$ factors), yields a neural network
$\widetilde{\Phi}_C \in \cN((d,1,\dots,1))$ (with $L$ layers) such that
%
\[
  \Realization_{\varrho}^{\Omega}(\widetilde{\Phi}_C)(x)
  \coloneqq \big( \varrho_C \circ \varrho_C \circ \dots \circ \varrho_C\big) (x_1)
  \quad \text{for all} \quad x \in \Omega ,
\]
where $\varrho_C$ is applied $L-1$ times.
%
Exactly as in the proof of the first part, this implies for $C \geq C_0$ that
\[
  \left|\Realization_{\varrho}^{\Omega}(\widetilde{\Phi}_{C})(x) - x_1 \right|
  \leq \eps
  \quad \text{for all} \quad x \in \Omega .
\]
Setting $\widetilde{\Phi}_{\eps,1}^B \coloneqq \widetilde{\Phi}_{C}$ and repeating
the previous arguments yields the claim for $i = 1$. Permuting the columns
of $A_1'$ yields the result for arbitrary $i \in \{1, \dots, d\}$.

Now, let $\varrho$ be increasing.
Then, $s_0 > 0$, and thus $\varrho_C$ is increasing for every $C > 0$.
Since $\Realization_{\varrho}^\Omega(\widetilde{\Phi}_{C})$ is the composition
of componentwise monotonically increasing functions, the claim regarding the monotonicity follows.
\end{proof}

\section{Proofs and results connected to Section~\ref{sec:Shape}}
\label{app:Convex}

\subsection{Proof of Theorem~\ref{thm:NoConvexityEver}}
\label{app:NoConvex}

We first establish the star-shapedness of the set of all realizations of
neural networks, which is a direct consequence of the fact that
the set is invariant under scalar multiplication.
The following proposition provides the details.

\begin{proposition}\label{prop:starshaped}
  Let $S = (d, N_1, \dots, N_L)$ be a neural network architecture, let $\Omega \subset \R^d$,
  and let $\varrho: \R \to \R$.
  Then, the set $\cRN_\varrho^\Omega(S)$
  is closed under scalar multiplication and is star-shaped with respect to the origin.
\end{proposition}

\begin{proof}
Let $f \in \cRN_\varrho^\Omega(S)$ and choose
$\Phi \coloneqq \big( (A_1,b_1), \dots, (A_L,b_L)\big) \in \cN(S)$ satisfying
${f = \Realization_\varrho^\Omega(\Phi)}$.
For ${\lambda \in \R}$, define
\(
  \widetilde{\Phi}
  \coloneqq \big(
       (A_1,b_1), \dots, (A_{L-1}, b_{L-1}), (\lambda A_{L}, \lambda b_L)
     \big)
\)
and observe that $\widetilde{\Phi} \in \cN(S)$ and furthermore
${\lambda f = \Realization^\Omega_\varrho(\widetilde{\Phi}) \in\cRN^\Omega_\varrho(S)}$.
This establishes the closedness of $\cRN^\Omega_\varrho(S)$ under scalar multiplication.


We can choose $\lambda = 0$ in the
argument above and obtain ${0 \in \cRN^\Omega_\varrho(S)}$.
For every $f \in \cRN^\Omega_\varrho(S)$ the line $\{\lambda f \colon \lambda \in [0,1]\}$
between $0$ and $f$ is contained in $\cRN^\Omega_\varrho(S)$,
since $\cRN^\Omega_\varrho(S)$ is closed under scalar multiplication.
We conclude that $\cRN^\Omega_\varrho(S)$ is star-shaped with respect to the origin.
\end{proof}

Our next goal is to show that $\cRN^\Omega_{\varrho}(S)$ cannot contain infinitely many
linearly independent centers.

As a preparation, we prove two related results which show that the class $\cRN_\varrho^\Omega(S)$
is ``small''.
The main assumption for guaranteeing this is that the activation function should be
locally Lipschitz continuous.

\begin{lemma}\label{lem:LipschitzImagesOfNetworkClass}
  Let $S = (d, N_1, \dots, N_L)$ be a neural network architecture, set $N_0 \coloneqq d$,
  and let $M \in \N$.
  Let $\varrho : \R \to \R$ be locally Lipschitz continuous.
  Let $\Omega \subset \R^d$ be compact, and let $\Lambda : C(\Omega; \R^{N_L}) \to \R^M$
  be locally Lipschitz continuous, with respect to the uniform norm on $C(\Omega; \R^{N_L})$.

  If $M > \sum_{\ell=1}^L (N_{\ell - 1} + 1) N_\ell$,
  then $\Lambda (\cRN_\varrho^\Omega (S)) \subset \R^M$ is a set of Lebesgue measure zero.
\end{lemma}

\begin{proof}
  Since $\varrho$ is locally Lipschitz continuous,
  Proposition~\ref{prop:RealizationContinuity} (which will be proved completely independently)
  shows that the realization map
  \[
    \Realization_\varrho^{\Omega}:
    \left(
      \cN(S), \|\cdot\|_{\cN(S)}
    \right)
    \to \big( C(\Omega; \R^{N_L}), \|\cdot\|_{\sup} \big)
  \]
  is locally Lipschitz continuous.
  Here, the normed vector space $\cN (S)$ is per definition isomorphic to
  \(
    \vphantom{\sum_j}
    \prod_{\ell = 1}^L
      \big(
        \R^{N_{\ell - 1} \times N_\ell} \times \R^{N_\ell}
      \big)
  \)
  and thus has dimension $D := \sum_{\ell=1}^L (N_{\ell - 1} + 1) N_\ell$,
  so that there is an isomorphism $J : \R^D \to \cN(S)$.

  As a composition of locally Lipschitz continuous functions, the map
  \[
    \Gamma :
    \R^M \to \R^M,
    (x_1, \dots, x_M) \mapsto \Lambda \Big(
                                        \Realization_\varrho^\Omega \big( J (x_1,\dots,x_D) \big)
                                      \Big)
  \]
  is locally Lipschitz continuous, and satisfies
  \(
    \Lambda \big( \cRN_\varrho^\Omega(S) \big)
    = \mathrm{ran} (\Gamma)
    = \Gamma(\R^D \times \{0\}^{M-D})
  \).
  But it is well known (see for instance \cite[Theorem~5.9]{AmannEscher}),
  that a locally Lipschitz continuous function between Euclidean spaces of the same dimension
  maps sets of Lebesgue measure zero to sets of Lebesgue measure zero.
  Hence, $\Lambda(\cRN_\varrho^\Omega(S)) \subset \R^M$ is a set of Lebesgue measure zero.
\end{proof}

As a corollary, we can now show that the class of neural network realizations
cannot contain a subspace of large dimension.

\begin{corollary}\label{cor:NeuralNetworkSetOnlyHasSmallSubspaces}
  Let $S = (d, N_1, \dots, N_L)$ be a neural network architecture, set $N_0 \coloneqq d$,
  and let $\varrho : \R \to \R$ be locally Lipschitz continuous.

  Let $\emptyset \neq \Omega \subset \R^d$ be arbitrary.
  If $V \subset C(\Omega; \R^{N_L})$ is a vector space
  with $V \subset \cRN_\varrho^\Omega (S)$,
  then ${\dim V \leq \sum_{\ell=1}^L (N_{\ell - 1} + 1) N_\ell}$.
\end{corollary}

\begin{proof}
Let $D \coloneqq \sum_{\ell=1}^L (N_{\ell - 1} + 1) N_\ell$.
Assume towards a contradiction that the claim of the corollary does not hold;
then there exists a subspace $V \subset C(\Omega; \R^{N_L})$ of dimension $\dim V = D + 1$
with $V \subset \cRN_{\varrho}^\Omega (S)$.
For $x \in \Omega$ and $\ell \in \underline{N_L}$,
let $\delta_x^{(\ell)} : C(\Omega; \R^{N_L}) \to \R, f \mapsto \big( f(x) \big)_\ell$.
Define
\(
  W
  \coloneqq \mathrm{span}
            \big\{
              \delta_x^{(\ell)} |_V \colon x \in \Omega, \ell \in \underline{N_L} \,
            \big\}
\),
and note that $W$ is a subspace of the finite-dimensional algebraic dual space $V^\ast$ of $V$.
In particular, $\dim W \leq \dim V^\ast = \dim V = D+1$, so that there are
$(x_1, \ell_1), \dots, (x_{D+1}, \ell_{D+1}) \in \Omega \times \underline{N_L}$ such that
$W = \mathrm{span} \big\{ \delta_{x_k}^{(\ell_k)} \colon k \in \underline{D+1} \big\}$.

\smallskip{}

We claim that the linear map
\[
  \Lambda_0 :
  V \to \R^{D+1},
  f \mapsto \big( [f(x_k)]_{\ell_k} \big)_{k \in \underline{D+1}}
\]
is surjective.
Since $\dim V = D+1 = \dim \R^{D+1}$, it suffices to show that $\Lambda_0$ is injective.
But if $\Lambda_0 f = 0$ for some $f \in V \subset C(\Omega; \R^{N_L})$,
and if $x \in \Omega$ and $\ell \in \underline{N_L}$ are arbitrary,
then $\delta_x^{(\ell)} = \sum_{k = 1}^{D+1} a_k \, \delta_{x_k}^{(\ell_k)}$ for certain
$a_1, \dots, a_{D+1} \in \R$.
Hence, $[f(x)]_\ell = \sum_{k=1}^{D+1} a_k [f(x_k)]_{\ell_k} = 0$.
Since $x \in \Omega$ and $\ell \in \underline{N_L}$ were arbitrary, this means $f \equiv 0$.
Therefore, $\Lambda_0$ is injective and thus surjective.

Now, let us define $\Omega' \coloneqq \{x_1, \dots, x_{D+1} \}$,
and note that $\Omega' \subset \R^d$ is compact.
Set $M \coloneqq D+1$, and define
\[
  \Lambda :
  C(\Omega' , \R^{N_L}) \to \R^{M},
  f \mapsto \big( [f(x_k)]_{\ell_k} \big)_{k \in \underline{D+1}} .
\]
It is straightforward to verify that $\Lambda$ is Lipschitz continuous.
Therefore, Lemma~\ref{lem:LipschitzImagesOfNetworkClass} shows that the set
$\Lambda(\cRN_\varrho^{\Omega'} (S)) \subset \R^M$ is a null-set.
However,
\[
  \Lambda \big( \cRN_\varrho^{\Omega'} (S) \big)
    = \Lambda \big(
                \{
                  f|_{\Omega'}
                  \colon
                  f \in \cRN_\varrho^\Omega (S)
                \}
              \big) 
    \supset \Lambda \big( \{ f|_{\Omega'} \colon f \in V \} \big)
    =       \Lambda_0 (V)
    =       \R^{D+1} = \R^M .
\]
This yields the desired contradiction.
\end{proof}

%
%

Now, the announced estimate for the number of linearly independent centers of the set
of all network realizations of a fixed size is a direct consequence.

\begin{proposition}\label{prop:Centres}
Let $S = (d, N_1, \dots, N_L)$ be a neural network architecture, let $\Omega \subset \R^d$,
and let $\varrho: \R \to \R$ be locally Lipschitz continuous.
Then, $\cRN^\Omega_\varrho(S)$ contains at most $\sum_{\ell = 1}^L (N_{\ell-1} + 1) N_{\ell}$
linearly independent centers, where $N_0 = d$.
That is, the number of linearly independent centers is bounded by the total number of parameters
of the underlying neural networks.
\end{proposition}

\begin{proof}
%

Let us set $D \coloneqq \sum_{\ell = 1}^L (N_{\ell - 1} + 1) N_\ell$,
and assume towards a contradiction that $\cRN^\Omega_\varrho(S)$ contains
$M \coloneqq D+1$ linearly independent centers
$\Realization_\varrho^\Omega(\Phi_1), \dots, \Realization_\varrho^\Omega(\Phi_M)$.
Since $\cRN_\varrho^\Omega(S)$ is closed under multiplication with scalars, this implies
\[
  V
  \coloneqq \spn \left\{
            \Realization_\varrho^\Omega(\Phi_1), \dots, \Realization_\varrho^\Omega(\Phi_M)
          \right\}
  \subset \cRN_\varrho^\Omega (S) .
\]
Indeed, this follows by induction on $M$, using the following observation:
If $V$ is a vector space contained in a set $A$, if $A$ is closed under multiplication with scalars,
and if $x_0 \in A$ is a center for $A$, then $V + \spn \{x_0\} \subset A$.
To see this, let $\mu \in \R$ and $v \in V$.
There is some $\eps \in \{1,-1\}$ such that $\eps \mu = |\mu|$.
Now set $x \coloneqq \eps v \in V \subset A$ and $\lambda \coloneqq | \mu | / (1+| \mu |) \in [0,1]$.
Then,
\[
  v + \mu \, x_0
  = \eps \cdot (\eps v + | \mu | x_0)
  = \eps
    \cdot (1 + |\mu|)
    \cdot \left( \frac{1}{1+| \mu |} x + \frac{| \mu |}{1+| \mu |} x_0 \right)
  = \eps \cdot (1 + |\mu|) \cdot \big( \lambda x_0 + (1 - \lambda) x \big)
  \in A .
\]
Since the family $\big( \Realization_\varrho^\Omega (\Phi_k) \big)_{k \in \underline{M}}$
is linearly independent, we see $\dim V = M > D = \sum_{\ell = 1}^L (N_{\ell - 1} + 1) N_\ell$.
In view of Corollary~\ref{cor:NeuralNetworkSetOnlyHasSmallSubspaces},
this yields the desired contradiction.
\end{proof}

Next, we analyze the convexity of $\cRN^\Omega_\varrho(S)$.
As a direct consequence of Proposition~\ref{prop:Centres}, we see that $\cRN^\Omega_\varrho(S)$
is never convex if $\cRN^\Omega_\varrho(S)$ contains more than a certain number
of linearly independent functions.

\begin{corollary}\label{cor:NoConvexity}
  Let $S = (d, N_1, \dots, N_L)$ be a neural network architecture and let $N_0 \coloneqq d$.
  Let $\Omega \subset \R^d$, and let $\varrho: \R \to \R$ be locally Lipschitz continuous.

  If $\cRN^\Omega_\varrho(S)$ contains more than $\sum_{\ell = 1}^L (N_{\ell-1}+1)N_{\ell}$
  linearly independent functions, then $\cRN^\Omega_\varrho(S)$ is not convex.
\end{corollary}

\begin{proof}
  Every element of a convex set is a center.
  Thus the result follows directly from Proposition~\ref{prop:Centres}.
\end{proof}

Corollary~\ref{cor:NoConvexity} claims that if a set of realizations of neural
networks with fixed size contains more than a fixed number of linearly
independent functions, then it cannot be convex.
Since $\cRN^{\R^d}_\varrho(S)$ is translation invariant, it is very likely
that $\cRN^{\R^d}_\varrho(S)$ (and hence also $\cRN^{\Omega}_\varrho(S)$)
contains \emph{infinitely} many linearly independent functions.
In fact, our next result shows under minor regularity assumptions on $\varrho$
that if the set $\cRN^\Omega_\varrho(S)$
\emph{does not} contain infinitely many linearly independent functions,
then $\varrho$ is necessarily a polynomial.

\begin{proposition}\label{prop:NetworksUsuallyProduceManyLinearlyIndependentFunctions}
  Let $S = (d, N_1, \dots, N_L)$ be a neural network architecture with  $L \in \N_{\geq 2}$.
  Moreover, let $\varrho:\R\to \R$ be continuous.
  Assume that there exists $x_0 \in \R$ such that $\varrho$ is
  differentiable at $x_0$ with
  $\varrho'(x_0) \neq 0$.

  Further assume that $\Omega \subset \R^d$ has nonempty interior, and that
  $\cRN_\varrho^{\Omega}(S)$
  does \emph{not} contain infinitely many linearly independent functions.
  Then, $\varrho$ is a polynomial.
\end{proposition}

\begin{proof}
\textbf{Step 1:} Set $S' \coloneqq (d, N_1, \dots, N_{L-1}, 1)$.
We first show that $\cRN_\varrho^\Omega (S')$
does not contain infinitely many linearly independent functions.
To see this, first note that the map
\[
  \Theta :
  \cRN_\varrho^\Omega (S) \to \cRN_\varrho^\Omega (S'),
  f \mapsto f_1 ,
\]
which maps an $\R^{N_L}$-valued function to its first component,
is linear, well-defined, and \emph{surjective}.

Hence, if there were infinitely many linearly independent functions $(f_n)_{n \in \N}$
in $\cRN_{\varrho}^{\Omega} (S')$, then we could find
$(g_n)_{n \in \N}$ in $\cRN_\varrho^\Omega (S)$ such that $f_n = \Theta \, g_n$.
But then the $(g_n)_{n \in \N}$ are necessarily linearly independent,
contradicting the hypothesis of the theorem.

\medskip{}

\textbf{Step 2:}  We show that $\CalG := \cRN_\varrho^{\R^d} (S')$ does not
contain infinitely many linearly independent functions.

To see this, first note that since $\CalF := \cRN_\varrho^\Omega (S')$ does not
contain infinitely many linearly independent functions (Step 1),
elementary linear algebra shows that there is a finite-dimensional subspace $V \subset C(\Omega; \R)$
satisfying $\CalF \subset V$.
Let $D := \dim V$, and assume towards a contradiction that there are $D+1$ linearly independent
functions $f_1, \dots, f_{D+1} \in \CalG$,
and set $W := \mathrm{span} \{f_1, \dots, f_{D+1} \} \subset C(\R^d; \R)$.
The space ${\Gamma := \mathrm{span} \{\delta_x |_W \colon x \in \R^d \} \subset W^\ast}$
spanned by the point evaluation functionals $\delta_x : C(\R^d; \R) \to \R, f \mapsto f(x)$
is finite-dimensional with $\dim \Gamma \leq \dim W^\ast = \dim W = D+1$.
Hence, there are $x_1, \dots, x_{D+1} \in \R^d$ such that
$\Gamma = \mathrm{span} \{\delta_{x_1}|_W , \dots, \delta_{x_{D+1}} |_W \}$.

We claim that the map
\[
  \Theta : W \to \R^{D+1}, f \mapsto \big( f(x_\ell) \big)_{\ell \in \underline{D+1}}
\]
is surjective.
Since $\dim W = D+1$, it suffices to show that $\Theta$ is injective.
If this was not true, there would be some $f \in W \subset C(\R^d; \R)$, $f \not\equiv 0$
such that $\Theta f = 0$.
But since $f \not\equiv 0$, there is some $x_0 \in \R^d$ satisfying $f(x_0) \neq 0$.
Because of $\delta_{x_0}|_W \in \Gamma$,
we have $\delta_{x_0}|_W = \sum_{\ell = 1}^{D+1} a_\ell \, \delta_{x_\ell}|_W$
for certain $a_1, \dots, a_{D+1} \in \R$.
Hence,
\(
  0
  \neq f(x_0)
  = \delta_{x_0}|_W (f)
  = \sum_{\ell=1}^{D+1} a_\ell \, \delta_{x_\ell}|_W (f)
  = 0
  ,
\)
since $f(x_\ell) = \big( \Theta(f) \big)_\ell = 0$ for all $\ell \in \underline{D+1}$.
This contradiction shows that $\Theta$ is injective, and hence surjective.

Now, since $\Omega$ has nonempty interior, there is some $b \in \Omega$ and some $r > 0$ such that
$y_\ell := b + r \, x_\ell \in \Omega$ for all $\ell \in \underline{D+1}$.
Define
\[
  g_\ell : \R^d \to \R, y \mapsto f_\ell \left( \frac{y}{r} - \frac{b}{r} \right)
  \quad \text{for} \quad \ell \in \underline{D + 1} .
\]
It is not hard to see $g_\ell \in \CalG$, and hence $g_\ell |_\Omega \in \CalF \subset V$
for all $\ell \in \underline{D+1}$.
Now, define the linear operator
$\Lambda : V \to \R^{D+1}, f \mapsto \big( f(y_\ell) \big)_{\ell \in \underline{D+1}}$,
and note that
\(
  \Lambda(g_\ell)
  = \big( g_\ell (y_k) \big)_{k \in \underline{D+1}}
  = \big( f_\ell (x_k) \big)_{k \in \underline{D+1}}
  = \Theta (f_\ell)
  ,
\)
because of ${y_\ell}/{r} - {b}/{r} = x_\ell$.
Since the functions $f_1, \dots, f_{D+1}$ span the space $W$,
this implies $\Lambda(V) \supset \Theta(W) = \R^{D+1}$,
in contradiction to $\Lambda$ being linear and $\dim V = D < D+1$.
This contradiction shows that $\CalG$ does not contain
infinitely many linearly independent functions.

\medskip{}

\textbf{Step 3:} From the previous step, we know that
$\CalG = \cRN_\varrho^{\R^d}(S')$ does not
contain infinitely many linearly independent functions.
In this step, we show that this implies that the activation function $\varrho$ is a polynomial.

To this end, define
\[
  \cRN_{S', \varrho}^*
  \coloneqq \left\lbrace
       f:\R \to \R \,
       \middle|
       \begin{varwidth}{\linewidth}
       \text{ there is some}
       $g \in \cRN^{\R^d}_\varrho(S')$ \\ \vspace{.2cm}
       \text{with}
       $f(x) = g(x, 0, \dots, 0)$
       \text{ for all } $x \in \R$
       \end{varwidth}
     \right\rbrace .
\]
Clearly, $\cRN_{S',\varrho}^*$ is dilation- and translation invariant; that is, if
$f \in \cRN_{S', \varrho}^*$, then also ${f(a \, \cdot) \in \cRN_{S', \varrho}^*}$
and ${f(\cdot-x) \in \cRN_{S', \varrho}^*}$ for arbitrary $a > 0$ and $x \in \R$.
Furthermore, by Step~2, we see that $\cRN_{S', \varrho}^*$
does not contain infinitely many linearly independent functions.
Therefore, $V \coloneqq \spn \cRN_{S', \varrho}^*$ is a finite-dimensional
translation- and dilation invariant subspace of $C(\R)$.
Thanks to the translation invariance, it follows from \cite{MR0169048} that
there exists some $r \in \N$, and certain $\lambda_j \in \CC$,
$k_j \in \N_0$ for $j = 1, \dots, r$ such that
\begin{equation}
  \cRN_{S', \varrho}^*
  \subset V
  \subset \spn_{\mathbb{C}} \left\{
                 x \mapsto x^{k_j} e^{\lambda_j x}: j = 1, \dots, r
               \right\} ,
  \label{eq:ConvexityPolynomialExponentialInclusion}
\end{equation}
where $\spn_{\mathbb{C}}$ denotes the linear span, with $\CC$ as the underlying field.
Clearly, we can assume ${(k_j, \lambda_j) \neq (k_\ell, \lambda_\ell)}$ for $j \neq \ell$.

\medskip{}

\textbf{Step 4:} Let $N := \max_{j \in \FirstN{r}} \, k_j$.
We claim that $V$ is contained in the space $\CC_{\deg \leq N} [X]$ of (complex) polynomials
of degree at most $N$.

Indeed, suppose towards a contradiction that there is some ${f \in V \setminus \CC_{\deg \leq N} [X]}$.
Thanks to \eqref{eq:ConvexityPolynomialExponentialInclusion}, we can write
$f = \sum_{j = 1}^r a_j \, x^{k_j} \, e^{\lambda_j x}$ with $a_1, \dots, a_r \in \CC$.
Because of $f \notin \CC_{\deg \leq N} [X]$, there is some $\ell \in \FirstN{r}$ such that
$a_\ell \neq 0$ and $\lambda_\ell \neq 0$.
Now, choose $\beta > 0$ such that $|\beta \lambda_\ell| > |\lambda_j|$ for all $j \in \FirstN{r}$,
and note that $f(\beta \, \cdot) \in V$,
so that Equation~\eqref{eq:ConvexityPolynomialExponentialInclusion}
yields coefficients $b_1, \dots, b_r \in \CC$ such that
$f(\beta \, x) = \sum_{j=1}^r b_j \, x^{k_j} \, e^{\lambda_j x}$.
By subtracting the two different representations for $f(\beta \, x)$, we thus see
\[
  0 \equiv f(\beta x) - f(\beta x)
    =      \sum_{j=1}^r a_j \, \beta^{k_j} \, x^{k_j} \, e^{\beta \lambda_j x}
           - \sum_{j=1}^r b_j \, x^{k_j} \, e^{\lambda_j x} ,
\]
and hence
\begin{equation}
  x^{k_\ell} e^{\beta \lambda_\ell x}
  = \frac{1}{a_\ell \, \beta^{k_\ell}} \cdot
    \Big(
      \sum_{j=1}^r
        b_j \, x^{k_j} \, e^{\lambda_j x}
      - \sum_{j \in \FirstN{r} \setminus \{ \ell \}}
          a_j \, \beta^{k_j} \, x^{k_j} \, e^{\beta \lambda_j x}
    \Big) .
  \label{eq:ConvexityPolynomialExponentialContradiction}
\end{equation}
Note, however, that $|\beta \lambda_\ell| > |\lambda_j|$ and hence
$(k_\ell, \beta \lambda_\ell) \neq (k_j, \lambda_j)$ for all $j \in \FirstN{r}$,
and furthermore that $(k_\ell, \beta \lambda_\ell) \neq (k_j, \beta \lambda_j)$
for $j \in \FirstN{r} \setminus \{ \ell \}$.
Thus, Lemma~\ref{lem:PolynomialExponentialLinearIndependence} below shows
that Equation~\eqref{eq:ConvexityPolynomialExponentialContradiction} cannot be true.
This is the desired contradiction.

\medskip{}

\textbf{Step 5:} In this step, we complete the proof, by first showing for arbitrary $B > 0$
that $\varrho|_{[-B,B]}$ is a polynomial of degree at most $N$.

Let $\eps, B > 0$ be arbitrary.
Since $\varrho$ is continuous, it is uniformly
continuous on $[-B-1,B+1]$, that is, there is some $\delta \in (0,1)$ such that
$| \varrho(x) - \varrho(y) | \leq \eps$ for all $x,y \in [-B-1,B+1]$
with $| x -y | \leq \delta$.
Since $\varrho'(x_0) \neq 0$ and $L \geq 2$,
Proposition~\ref{prop:Identity} and Lemma~\ref{lem:enlarge} imply existence of a neural network
$\widetilde{\Phi}_{\eps,B} \in \cN((d, N_1, \dots, N_{L-1}))$ such that
\[
  \left|
    \big[ \Realization^{[-B,B]^d}_{\varrho}(\widetilde{\Phi}_{\eps,B})(x) \big]_1 - x_1
  \right|
  \leq \delta,
  \text{ for all } x \in [-B,B]^d .
\]
In particular, this implies because of $\delta \leq 1$ that
$\big[ \Realization^{[-B,B]^d}_\varrho(\widetilde{\Phi}_{\eps,B})(x) \big]_1 \in [-B-1,B+1]$
for all $x \in [-B,B]^d$.
We conclude that
\begin{align}\label{eq:ThePlusOneNetwork}
  \left|
    \big[
      \varrho \big( \Realization^{[-B,B]^d}_{\varrho}(\widetilde{\Phi}_{\eps,B})(x) \big)
    \big]_1
    - \varrho(x_1)
  \right|
  \leq \eps,
  \quad \text{for all} \quad x \in [-B,B]^d ,
\end{align}
with $\varrho$ acting componentwise.
By \eqref{eq:ThePlusOneNetwork}, it follows that there is a neural network
${\Phi_{\eps,B} \in \cN(S')}$ satisfying
\begin{align} \label{eq:approxNetwork}
  \big|
    \Realization^{[-B,B]^d}_{\varrho}(\Phi_{\eps,B})(x_1, 0, \dots, 0)
    - \varrho(x_1)
  \big|
 \leq \eps
 \quad \text{for all } x_1 \in [-B,B] .
\end{align}
From \eqref{eq:approxNetwork} and Step~4, we thus see
\[
  \varrho|_{[-B,B]}
  \in \overline{\left\{ f|_{[-B,B]}: f\in \cRN_{S', \varrho}^*\right\}}
  \subset \{ p|_{[-B,B]} \colon p \in \CC_{\deg \leq N}[X] \} ,
\]
where the closure is taken with respect to the sup norm, and where we
implicitly used that the space on the right-hand side is a closed
subspace of $C([-B,B])$, since it is a finite dimensional subspace.

Since $\varrho|_{[-B,B]}$ is a polynomial of degree at most $N$;
we see that the $N+1$-th derivative of $\varrho$ satisfies
$\varrho^{(N+1)} \equiv 0$ on $(-B, B)$, for arbitrary $B > 0$.
Thus, $\varrho^{(N+1)} \equiv 0$, meaning that $\varrho$ is a polynomial.
\end{proof}

In the above proof, we used the following elementary lemma,
whose proof we provide for completeness.

\begin{lemma}\label{lem:PolynomialExponentialLinearIndependence}
  For $k \in \N_0$ and $\lambda \in \CC$, define
  $f_{k,\lambda} : \R \to \CC, x \mapsto x^{k} \, e^{\lambda x}$.

  Let $N \in \N$, and let $(k_1, \lambda_1), \dots, (k_N, \lambda_N) \in \N_0 \times \CC$
  satisfy $(k_\ell, \lambda_\ell) \neq (k_j, \lambda_j)$ for $\ell \neq j$.
  Then, the family $(f_{k_\ell, \lambda_\ell})_{\ell=1,\dots,N}$ is linearly independent
  over $\CC$.
\end{lemma}

\begin{proof}
  Let us assume towards a contradiction that
  \begin{equation}
    0 \equiv \sum_{\ell=1}^N a_\ell \, f_{k_\ell, \lambda_\ell} (x)
      =      \sum_{\ell=1}^N a_\ell \, x^{k_\ell} \, e^{\lambda_\ell x}
    \label{eq:PolynomialExponentialLinearIndependence}
  \end{equation}
  for some coefficient vector $(a_1,\dots,a_N) \in \CC^N \setminus \{ 0 \}$.
  By dropping those terms for which $a_\ell = 0$, we can assume that $a_\ell \neq 0$ for all
  $\ell \in \FirstN{N}$.

  Let $\Lambda := \{ \lambda_i \colon i \in \FirstN{N} \}$.
  In the case where $|\Lambda| = 1$, it follows that $k_j \neq k_\ell$ for $j \neq \ell$.
  Furthermore, multiplying Equation~\eqref{eq:PolynomialExponentialLinearIndependence} by
  $e^{-\lambda_1 x}$, we see that
  \(
    0 \equiv \sum_{\ell=1}^N a_\ell \, x^{k_\ell} ,
  \)
  which is impossible since the monomials $(x^k)_{k \in \N_0}$ are linearly independent.
  Thus, we only need to consider the case that $|\Lambda| > 1$.

  Define ${M \coloneqq \max \{ k_\ell \colon \ell \in \FirstN{N} \}}$ and
  \[
    I := \big\{ \ell \in \FirstN{N} \colon \lambda_\ell = \lambda_1 \big\},
    \qquad \text{and choose $j \in I$ satisfying} \qquad
    k_j = \max_{\ell \in I} \, k_\ell
    .
  \]
  Note that this implies $k_\ell < k_j$ for all $\ell \in I \setminus \{ j \}$,
  since $(k_\ell, \lambda_\ell) \neq (k_j, \lambda_j)$ and hence $k_\ell \neq k_j$
  for $\ell \in I \setminus \{ j \}$.

  Consider the differential operator
  \[
    T
    := \prod_{\lambda \in \Lambda \setminus \{ \lambda_1 \}}
         \Big(
           \frac{d}{d x} - \lambda \, \identity
         \Big)^{M+1}
  \]
  Note that
  \(
    \big( \frac{d}{dx} - \lambda \, \identity \big) (x^k \, e^{\mu x})
    = (\mu - \lambda) x^k \, e^{\mu x} + k \, x^{k-1} \, e^{\mu x} .
  \)
  Using this identity, it is easy to see that if $\lambda \in \Lambda \setminus \{ \lambda_1 \}$
  and $k \in \N_0$ satisfies $k \leq M$, then $T (x^k \, e^{\lambda x}) \equiv 0$.
  Furthermore, for each $k \in \N_0$ with $k \leq M$,
  there exist a constant $c_k \in \CC \setminus \{ 0 \}$ and a polynomial $p_k \in \CC[X]$ with
  $\deg p_k < k$ satisfying $T(x^k \, e^{\lambda_1 x}) = e^{\lambda_1 x} \cdot (c_k x^k + p_k (x))$.
  Overall, Equation~\eqref{eq:PolynomialExponentialLinearIndependence} implies that
  \[
    0 \equiv e^{-\lambda_1 x}
             \cdot T \Big( \sum_{\ell=1}^N a_\ell \, x^{k_\ell} \, e^{\lambda_\ell x} \Big)
      = a_{j} c_{k_j} x^{k_j}
        + a_j p_{k_j} (x)
        + \sum_{\ell \in I \setminus \{ j \}}
            \big[ a_\ell \cdot (c_{k_\ell} x^{k_\ell} + p_{k_\ell} (x)) \big]
      =: a_{j} c_{k_j} x^{k_j} + q(x) ,
  \]
  where $a_j c_{k_j} \neq 0$ and where $\deg q < k_j$,
  since $k_j > k_\ell$ for all $\ell \in I \setminus \{ j \}$.
  This is the desired contradiction.
\end{proof}

As our final ingredient for the proof of Theorem~\ref{thm:NoConvexityEver},
we show that every non-constant locally Lipschitz function $\varrho$ satisfies
the technical assumptions of Proposition~\ref{prop:NetworksUsuallyProduceManyLinearlyIndependentFunctions}.

\begin{lemma}\label{lem:LocallyLipschitzHasNonzeroDerivativeSomewhere}
  Let $\varrho : \R \to \R$ be locally Lipschitz continuous and not constant.
  Then, there exists some $x_0 \in \R$ such that $\varrho$ is differentiable at $x_0$
  with $\varrho'(x_0) \neq 0$.
\end{lemma}

\begin{proof}
  Since $\varrho$ is not constant, there is some $B > 0$ such that $\varrho|_{[-B,B]}$
  is not constant.
  By assumption, $\varrho$ is Lipschitz continuous on $[-B,B]$.
  Thus, $\varrho|_{[-B,B]}$ is absolutely continuous; see for instance \cite[Definition~7.17]{Rudin}.
  Thanks to the fundamental theorem of calculus for the Lebesgue integral
  (see \cite[Theorem~7.20]{Rudin}), this implies that $\varrho|_{[-B,B]}$ is differentiable
  almost everywhere on $(-B,B)$ and satisfies $\varrho (y) - \varrho(x) = \int_x^y \varrho'(t) \, dt$
  for $-B \leq x < y \leq B$, where $\varrho'(t) := 0$ if $\varrho$ is not differentiable at $t$.

  Since $\varrho|_{[-B,B]}$ is not constant, the preceding formula shows that there has to be some
  $x_0 \in (-B, B)$ such that $\varrho' (x_0) \neq 0$; in particular, this means that $\varrho$
  is differentiable at $x_0$.
\end{proof}

Now, a combination of Corollary~\ref{cor:NoConvexity},
Proposition~\ref{prop:NetworksUsuallyProduceManyLinearlyIndependentFunctions},
and Lemma~\ref{lem:LocallyLipschitzHasNonzeroDerivativeSomewhere}
proves Theorem~\ref{thm:NoConvexityEver}.
For the application of Lemma~\ref{lem:LocallyLipschitzHasNonzeroDerivativeSomewhere},
note that if $\varrho$ is constant, then $\varrho$ is a polynomial, so that the conclusion
of Theorem~\ref{thm:NoConvexityEver} also holds in this case.

\subsection{Proof of Theorem~\ref{thm:epsconv}}
\label{app:epsconv}

We first show in the following lemma that if $\overline{\cRN_\varrho^{\Omega}(S)}$ is convex,
then $\cRN_\varrho^{\Omega}(S)$ is dense in $C(\Omega)$.
The proof of Theorem~\ref{thm:epsconv} is given thereafter.


\begin{lemma}\label{lem:NoApproximativeConvexity}
Let $S = (d, N_1, \dots, N_{L-1}, 1)$ be a neural network architecture with $L \geq 2$.
Let $\Omega \subset \R^d$ be compact and let $\varrho: \R \to \R$ be continuous but not a polynomial.
Finally, assume that there is some $x_0 \in \R$ such that $\varrho$ is differentiable at $x_0$
with $\varrho'(x_0) \neq 0$.

If $\overline{\cRN_\varrho^{\Omega}(S)}$ is convex, then $\cRN_\varrho^{\Omega}(S)$
is dense in $C(\Omega)$.
\end{lemma}


\begin{proof}
Since $\overline{\cRN_{\varrho}^{\Omega}(S)}$ is convex and closed under scalar multiplication,
$\overline{\cRN_{\varrho}^{\Omega}(S)}$ forms a closed linear subspace of $C(\Omega)$.
Below, we will show that
\(
  \big(
    \Omega \to \R,
    x \mapsto \varrho(\langle a, x \rangle + b)
  \big)
  \in \overline{\cRN_{\varrho}^{\Omega}(S)}
\)
for arbitrary $a \in \R^d$ and $b \in \R$.
Once we prove this, it follows that
\(
  \big( \Omega \to \R, x \mapsto \sum_{i=1}^N c_i \, \varrho(b_i + \langle a_i, x \rangle) \big)
  \in \overline{\cRN_{\varrho}^{\Omega}(S)}
\)
for arbitrary $N \in \N$, $a_i \in \R^d$, and $b_i, c_i \in \R$.
As shown in \cite{PinkusUniversalApproximation}, this then entails
that $\overline{\cRN_{\varrho}^{\Omega}(S)}$ (and hence also $\cRN_{\varrho}^{\Omega}(S)$)
is dense in $C(\Omega)$, since $\varrho$ is not a polynomial.

Thus, let $a \in \R^d$, $b \in \R$, and $\eps > 0$ be arbitrary, and define
$g : \Omega \to \R, x \mapsto \varrho(b + \langle a, x \rangle)$
and $\Psi \coloneqq \big( (a^T, b), (1, 1) \big) \in \cN( (d, 1, 1) )$,
noting that $\Realization_\varrho^{\Omega} (\Psi) = g$.
Since $g$ is continuous on the compact set $\Omega$, we have $|g(x)| \leq B$ for all $x \in \Omega$
and some $B > 0$.
By Proposition~\ref{prop:Identity} and since $\varrho'(x_0) \neq 0$ and $L \geq 2$,
there exists a neural network $\Phi_\eps \in \cN( (1,\dots,1) )$ (with $L-1$ layers)
such that $|\Realization_\varrho^{[-B,B]} (\Phi_\eps) - x| \leq \eps$
for all $x \in [-B,B]\vphantom{\sum_j}$.
This easily shows $\| \Realization_\varrho^{\Omega} (\Phi_\eps \conc \Psi) - g \|_{\sup} \leq \eps$,
while
\(
  \Realization_\varrho^\Omega(\Phi_\eps \conc \Psi)
  \in \cRN_{\varrho}^{\Omega} ( (d,1,\dots,1) )
  \subset \cRN_{\varrho}^{\Omega}(S)
\)
by Lemma~\ref{lem:enlarge}.
Therefore, $g \in \overline{\cRN_{\varrho}^{\Omega}(S)}$, which completes the proof.
\end{proof}

Now we are ready to prove Theorem~\ref{thm:epsconv}.
By assumption, $\cRN_{\varrho}^{\Omega}(S)$ is \emph{not} dense in $C(\Omega)$.
We start by proving that there exists at least one $\eps > 0$ such that the set
$\overline{\cRN_{\varrho}^{\Omega}(S)}$ is not $\eps$-convex.
Suppose towards a contradiction that this is not true, so that
$\overline{\cRN_{\varrho}^{\Omega}(S)}$ is $\eps$-convex for \emph{all} $\eps > 0$.
This implies
\begin{equation}\label{eq:DerAnfangVomEnde}
  \text{co}\left(\, \overline{\cRN_{\varrho}^{\Omega}(S)} \,\right)
  \subset \bigcap_{\eps > 0}
          \left(\,
            \overline{\cRN_{\varrho}^{\Omega}(S)} + B_\eps(0)
          \right)
  = \overline{\cRN_{\varrho}^{\Omega}(S)},
\end{equation}
where the last identity holds true, since if
$\widetilde{f} \not \in \overline{\cRN_{\varrho}^{\Omega}(S)}$,
there exists $\eps'>0$ such that $\|\widetilde{f}-f\|_{\sup}>\eps'$ for all
$f \in \overline{\cRN_{\varrho}^{\Omega}(S)}$.
Equation~\eqref{eq:DerAnfangVomEnde} shows that $\overline{\cRN_{\varrho}^{\Omega}(S)}$ is convex,
which by Lemma~\ref{lem:NoApproximativeConvexity} implies that
$\cRN_{\varrho}^{\Omega}(S) \subset C(\Omega)$ is dense, in contradiction to the assumptions
of Theorem~\ref{thm:epsconv}.
This is the desired contradiction, showing that $\overline{\cRN_{\varrho}^{\Omega}(S)}$
is $\eps$-convex for \emph{some} $\eps > 0$.

Thus, there exists $g \in \text{co} \big( \overline{\cRN_{\varrho}^{\Omega}(S)} \big)$ such that
\(
  \|g - f\|_{\sup} \geq \eps_0
\)
for all $f \in \overline{\cRN_{\varrho}^{\Omega}(S)}$.
Now, let $\eps > 0$ be arbitrary.
Then, $\frac{\eps}{\eps_0} g \in \text{co}(\overline{\cRN_{\varrho}^{\Omega}(S)})$,
since $\cRN_{\varrho}^{\Omega}(S)$ is closed under scalar multiplication.
Moreover,
\[
  \left\|\frac{\eps}{\eps_0} g - f\right\|_{\sup}\geq \eps
  \quad \text{for all} \quad f \in \overline{\cRN_{\varrho}^{\Omega}(S)},
\]
again due to the closedness under scalar multiplication of $\cRN_{\varrho}^{\Omega}(S)$.
This shows that $\overline{\cRN_{\varrho}^{\Omega}(S)}$ is not $\eps$-convex for any $\eps > 0$.
\hfill$\square$

\subsection{Non-dense network sets}
\label{app:nonDenseNetworkSets}

In this section, 
we review criteria on $\varrho$ which ensure that
$\overline{\cRN_\varrho^{\Omega}(S)} \neq C(\Omega)$.
Precisely, we will show that this is true if $\varrho : \R \to \R$ is
\textbf{computable by elementary operations}, which means that there is some $N \in \N$
and an algorithm that takes $x \in \R$ as input and returns $\varrho(x)$ after no more than $N$ of
the following operations:
\begin{itemize}
  \item applying the exponential function $\exp : \R \to \R$;

  \item applying one of the arithmetic operations $+, -, \times$, and $/$ on real numbers;

  \item jumps conditioned on comparisons of real numbers using the following
        operators: $<, >, \leq, \geq , = , \neq$.
\end{itemize}
Then, a combination of \cite[Theorem~14.1]{AnthonyBartlett} with
\cite[Theorem~8.14]{AnthonyBartlett} shows that if $\varrho$ is
computable by elementary operations, then the \emph{pseudo-dimension} of each
of the function classes $\cRN_\varrho^{\R^d}(S)$ is finite.
Here, the pseudo-dimension $\mathrm{Pdim}(\mathcal{F})$ of a function class
$\mathcal{F} \subset \R^X$ is defined as follows (see \cite[Section~11.2]{AnthonyBartlett}):
\[
  \mathrm{Pdim}(\mathcal{F})
  \coloneqq \sup
            \big\{
              |K|
              \colon
              K \subset X \text{ finite and pseudo-shattered by } \mathcal{F}
            \big\}
  \in \N \cup \{\infty\} \, .
\]
Here, a finite set $K = \{ x_1,\dots,x_m \} \subset X$ (with pairwise distinct $x_i$)
is \emph{pseudo-shattered} by $\mathcal{F}$
if there are $r_1,\dots,r_m \in \R$ such that for each $b \in \{0, 1\}^m$ there is a
function $f_b \in \mathcal{F}$ with $\Indicator_{[0,\infty)} \big( f_b (x_i) - r_i \big) = b_i$
for all $i \in \{1,\dots,m\}$.

Using this result, we can now show that the realization sets of networks with
activation functions that are computable by elementary operations are never dense
in $L^p (\Omega)$ or $C(\Omega)$.

\begin{proposition}\label{prop:MostActivationFunctionsDoNotYieldDenseRealisations}
  Let $\varrho : \R \to \R$ be continuous and computable by elementary operations.
  Moreover, let $S=(d, N_1, \dots, N_{L-1}, 1)$ be a neural network architecture.
  Let $\Omega \subset \R^d$ be any measurable set with nonempty interior,
  and let $\mathcal{Y}$ denote either $L^p(\Omega)$ (for some $p \in [1,\infty)$), or $C(\Omega)$.
  In case of $\mathcal{Y} = C(\Omega)$, assume additionally that $\Omega$ is compact.

  Then. we have $\overline{\mathcal{Y} \cap \cRN_\varrho^{\Omega}(S)} \subsetneq \mathcal{Y}$.
\end{proposition}

\begin{proof}
The considerations from before the statement of the proposition
show that
\[
  \mathrm{Pdim} \big( \mathcal{Y} \cap \cRN_\varrho^{\Omega}(S) \big)
  \leq \mathrm{Pdim} \big( \cRN_\varrho^{\R^d}(S) \big)
  <    \infty.
\]
Therefore, all we need to show is that if $\mathcal{F} \subset C(\Omega)$ is a function class
for which $\mathcal{F} \cap \mathcal{Y}$ is dense in $\mathcal{Y}$,
then $\mathrm{Pdim}(\mathcal{F}) = \infty$.

For $\mathcal{Y} = C(\Omega)$, this is easy:
Let $m \in \N$ be arbitrary, choose distinct points $x_1, \dots, x_m \in \Omega$,
and note that for each $b \in \{0,1\}^m$, there is $g_b \in C(\Omega)$ satisfying $g_b (x_j) = b_j$
for all $j \in \underline{m}$.
By density, for each $b \in \{0,1\}^m$, there is $f_b \in \mathcal{F}$ such that
$\|f_b - g_b\|_{\sup} < \frac{1}{2}$.
In particular, $f_b (x_j) > \frac{1}{2}$ if $b_j = 1$ and $f_b (x_j) < \frac{1}{2}$ if $b_j = 0$.
Thus, if we set $r_1 \coloneqq \dots \coloneqq r_m \coloneqq \frac{1}{2}$, then
$\Indicator_{[0,\infty)} (f_b(x_j) - r_j) = b_j$ for all $j \in \underline{m}$.
Hence, $S = \{x_1,\dots,x_m\}$ is pseudo-shattered by $\mathcal{F}$,
so that $\mathrm{Pdim}(\mathcal{F}) \geq m$.
Since $m \in \N$ was arbitrary, $\mathrm{Pdim}(\mathcal{F}) = \infty$.

For $\mathcal{Y} = L^p (\Omega)$, one can modify this argument as follows:
Since $\Omega$ has nonempty interior, there are ${x_0 \in \Omega}$ and $r > 0$
such that $x_0 + r [0,1]^d \subset \Omega$.
Let $m \in \N$ be arbitrary, and for $j \in \underline{m}$ define
${M_j := x_0 + r \big[ (\frac{j-1}{m}, \frac{j}{m}) \times [0,1]^{d-1} \big]}$.
Furthermore, for $b \in \{0,1\}^m$,
let $g_b := \sum_{j \in \underline{m} \text{ with } b_j = 1} \Indicator_{M_j}$,
and note $g_b \in L^p (\Omega)$.

Since $L^p(\Omega) \cap \mathcal{F} \subset L^p(\Omega)$ is dense,
there is for each $b \in \{0,1\}^m$ some $f_b \in \mathcal{F} \cap L^p(\Omega)$
such that $\|f_b - g_b \|_{L^p}^p \leq r^d / (2^{1+p} \cdot m \cdot 2^m)$.
If we set $\Omega_b := \{x \in \Omega \colon |f_b (x) - g_b (x) | \geq 1/2 \}$, then
$\Indicator_{\Omega_b} \leq 2^p \cdot |f_b - g_b|^p$, and hence
\[
  \lambda(\Omega_b)
  \leq 2^p \, \| f_b - g_b \|_{L^p}^p
  \leq \frac{r^d}{2 \cdot m \cdot 2^m} ,
\]
and thus $\lambda (\bigcup_{b \in \{0,1\}^m} \Omega_b) \leq \frac{r^d}{2m}$,
where $\lambda$ is the Lebesgue measure.
Hence, $\lambda (M_j \setminus \bigcup_{b \in \{0,1\}^m} \Omega_b) \geq \frac{r^d}{2m} > 0$,
so that we can choose for each $j \in \underline{m}$ some
$x_j \in M_j \setminus \bigcup_{b \in \{0,1\}^m} \Omega _b$.
We then have
\[
  |f_b (x_j) - \delta_{b_j, 1}| = |f_b (x_j) - g_b (x_j)| < 1/2 ,
\]
and hence $f_b (x_j) > 1/2$ if $b_j = 1$ and $f_b (x_j) < 1/2$ otherwise.
Thus, if we set $r_1 \coloneqq \dots \coloneqq r_m \coloneqq \frac{1}{2}$,
then we have as above that $\Indicator_{[0,\infty)} \big( f_b (x_j) - r_j \big) = b_j$
for all $j \in \underline{m}$ and $b \in \{ 0, 1 \}^m$.
The remainder of the proof is as for $\mathcal{Y} = C(\Omega)$.
\end{proof}

Note that the following activation functions are computable by elementary operations:
any piecewise polynomial function (in particular, the ReLU and the parametric ReLU),
the exponential linear unit, the softsign
(since the absolute value can be computed using a case distinction), the sigmoid, and the $\tanh$.
Thus, the preceding proposition applies to each of these activation functions.

\section{Proofs of the results in Section~\ref{sec:Closed}}

\subsection{Proof of Theorem~\ref{thm:nonclosed}}
\label{app:nonclosed}

The proof of Theorem~\ref{thm:nonclosed} is crucially based on the following lemma:

\begin{lemma}\label{lem:DiscontinuousFunctionWithoutContinuousRepresentative}
  Let $\mu$ be a finite Borel measure on $[-B, B]^d$ with uncountable support $\supp \mu$.
  For $x^\ast ,v \in \R^d$ with $v \neq 0$, define
  \[
    H_\pm (x^\ast, v) := x^\ast + H_{\pm} (v)
    \quad \text{where} \quad
    H_+ (v) := \{ x \in \R^d \colon \langle x, v \rangle > 0\}
    \quad \text{and} \quad
    H_- (v) := \{ x \in \R^d \colon \langle x, v \rangle < 0\}.
  \]
  Then, there are $x^\ast \in [-B, B]^d$ and $v \in S^{d-1}$ such that if $f : [-B,B]^d \to \R$ satisfies
  \begin{equation}
    f(x) = c \quad \text{for } x \in H_+ (x^\ast, v)
    \qquad \text{and} \qquad
    f(x) = c' \quad \text{for } x \in H_- (x^\ast, v)
    \qquad \text{with } c \neq c' ,
    \label{eq:DiscontinuousFunctionCondition}
  \end{equation}
  then there is no continuous $g : [-B,B]^d \to \R$ satisfying
  $f = g$ $\mu$-almost everywhere.
\end{lemma}

\begin{proof}
  \textbf{Step 1:} Let $K := \supp \mu \subset [-B,B]^d$.
  In this step, we show that there is some $x^\ast \in K$ and some $v \in S^{d-1}$
  such that $x^\ast \in \overline{K \cap H_+ (x^\ast, v)} \cap \overline{K \cap H_- (x^\ast, v)}$.
  This follows from a result in \cite{WardStructureOfNonEnumerableSetsOfPoints}, where the following
  is shown:
  For $x^\ast \in \R^d$ and $v \in S^{d-1}$, as well as $\delta, \eta > 0$, write
  \[
    C(v;\delta,\eta)
    \coloneqq \big\{
                r \, \xi
                \colon
                0 < r < \delta \text{ and } \xi \in S^{d-1} \text{ with } |\xi - v| < \eta
              \big\}
    \qquad \text{and} \qquad
    C(x^\ast, v; \delta, \eta) := x^\ast + C(v;\delta,\eta) .
  \]
  Then, for each uncountable set $E \subset \R^d$ and for all but countably many $x^\ast \in E$,
  there is some $v \in S^{d-1}$ such that $E \cap C(x^\ast, v; \delta, \eta)$
  and $E \cap C(x^\ast, -v; \delta, \eta)$ are both uncountable for all $\delta, \eta > 0$.

  Now, if $\eta < 1$, then any $r \xi \in C(v;\delta,\eta)$ with $|\xi - v| < \eta$
  satisfies
  \(
    \langle v, \xi \rangle
    = \langle v,v \rangle + \langle v, \xi - v \rangle
    \geq 1 - |\xi - v|
    > 0 ,
  \)
  so that $C(x^\ast, v; \delta, \eta) \subset B_\delta(x^\ast) \cap H_+ (x^\ast, v)$
  and $C(x^\ast, -v; \delta, \eta) \subset B_\delta(x^\ast) \cap H_- (x^\ast, v)$.
  From this it is easy to see that if $x^\ast, v$ are as provided by the result in
  \cite{WardStructureOfNonEnumerableSetsOfPoints} (for $E = K$), then indeed
  $x^\ast \in \overline{K \cap H_+ (x^\ast, v)} \cap \overline{K \cap H_- (x^\ast, v)}$.

  We remark that strictly speaking, the proof in \cite{WardStructureOfNonEnumerableSetsOfPoints}
  is only provided for $E \subset \R^3$, but the proof extends almost verbatim to $\R^d$.
  A direct proof of the existence of $x^\ast, v$ can be found in \cite{StackexchangeTwoSidedLimitPoint}.

  \medskip{}

  \textbf{Step 2:} We show that if $x^\ast, v$ are as in Step~1 and if $f : [-B,B]^d \to \R$
  satisfies \eqref{eq:DiscontinuousFunctionCondition}, then there is no continuous
  $g : [-B,B]^d \to \R$ satisfying $f = g$ $\mu$-almost everywhere.

  Assume towards a contradiction that such a continuous function $g$ exists.
  Recall (see for instance \mbox{\cite[Section~7.4]{CohnMeasureTheory}})
  that the support of $\mu$ is defined as
  \[
    \supp \mu
    = [-B,B]^d \setminus \bigcup \{ U \colon U \subset [-B,B]^d \text{ open and } \mu (U) = 0\}.
  \]
  In particular, if $U \subset [-B,B]^d$ is open with $U \cap \supp \mu \neq \emptyset$,
  then $\mu(U) > 0$.

  For each $n \in \N$, set $U_{n,+} := B_{1/n}(x^\ast) \cap [-B,B]^d \cap H_+ (x^\ast, v)$
  and $U_{n,-} := B_{1/n}(x^\ast) \cap [-B,B]^d \cap H_- (x^\ast, v)$,
  and note that $U_{n,\pm}$ are both open (as subsets of $[-B,B]^d$)
  with $K \cap U_{n, \pm} \neq \emptyset$, since
  $x^\ast \in \overline{K \cap H_+ (x^\ast, v)}$ and $x^\ast \in \overline{K \cap H_- (x^\ast, v)}$.
  Hence, $\mu(U_{n,\pm}) > 0$.
  Since $f = g$ $\mu$-almost everywhere, there exist $x_{n,\pm} \in U_{n,\pm}$
  with $f(x_{n,\pm}) = g(x_{n,\pm})$.
  This implies $g(x_{n,+}) = c$ and $g(x_{n,-}) = c'$.
  But since $x_{n,\pm} \in B_{1/n}(x^\ast)$, we have $x_{n,\pm} \to x^\ast$, so that the continuity
  of $g$ implies $g(x^\ast) = \lim_{n} g(x_{n,+}) = c$ and $g(x^\ast) = \lim_n g(x_{n,-}) = c'$,
  in contradiction to $c \neq c'$.
\end{proof}

We now prove Theorem~\ref{thm:nonclosed}.
Set $\Omega \coloneqq [-B,B]^d$ and define
$\widetilde{N}_i := 1$ for ${i = 1,\dots,L-2}$ and $\widetilde{N}_{L-1} := 2$ if $\varrho$ is unbounded,
while $\widetilde{N}_{L-1} := 1$ otherwise.
We show that for $\varrho$ as in the statement of the theorem there exists a sequence of functions in
\(
  \cRN_\varrho^{\Omega}( (d, \widetilde{N}_1, \dots, \widetilde{N}_{L-1}, 1) )
  \subset \cRN_\varrho^{\Omega}( S )
\)
such that the sequence converges (in $L^p(\mu)$) to a \emph{bounded, discontinuous limit}
$f \in L^\infty(\mu)$, meaning that $f$ does not have a continuous representative,
even after possibly changing it on a $\mu$-null-set.
Since 
\(
    \cRN_{\varrho}^\Omega(S) \subset C(\Omega) ,
\)
this will show that $f \in \overline{\cRN_\varrho^{\Omega}(S)} \setminus \cRN_\varrho^{\Omega}(S)$.

\medskip

For the construction of the sequence, let $x^\ast \in \supp \mu$ and $v \in S^{d-1}$ as
provided by Lemma~\ref{lem:DiscontinuousFunctionWithoutContinuousRepresentative}.
Extend $v$ to an orthonormal basis $(v, w_1, \dots, w_{d-1})$ of $\R^d$, and define
$A := O^T$ for $O := (v, w_1, \dots, w_{d-1}) \in \R^{d \times d}$.
Note that $x^\ast \in \Omega \subset \overline{B}_{dB} (0)$ and hence
$A (\Omega - x^\ast) \subset \overline{B}_{2dB}(0) \subset [-2dB, 2dB]^d =: \Omega'$.
Define $B' := 2 d B$.

Next, using Proposition~\ref{prop:Identity}, choose a neural network
$\Psi \in \cN((d, 1, \dots, 1))$ with $L-1$ layers such that
\begin{itemize}

  \item[(1)] $\Realization^{\Omega'}_{\varrho}(\Psi)(0) = 0$;

  \item[(2)] $\Realization_{\varrho}^{\Omega'}(\Psi)$ is differentiable at $0$ and
             $\frac{\partial\Realization_{\varrho}^{\Omega'}(\Psi)}{\partial x_1}(0)=1$;

  \item[(3)] $\Realization_{\varrho}^{\Omega'}(\Psi)$ is constant in all but the
             $x_1$-direction; and

  \item[(4)] $\Realization_\varrho^{\Omega'}(\Psi)$ is increasing with respect to each variable
             (with the remaining variables fixed).
\end{itemize}

Let $J_0 \coloneqq \Realization_{\varrho}^{\Omega'}(\Psi)$.
Since $J_0(0) = 0$ and $\frac{\partial J_0}{\partial x_1} (0) = 1$, we see
directly from the definition of the partial derivative that for each $\delta \in (0, B')$,
there are $x_\delta \in (-\delta, 0)$ and $y_\delta \in (0, \delta)$ such that
${J_0(x_\delta, 0, \dots, 0) < J_0(0) = 0}$ and $J_0(y_\delta, 0, \dots, 0) > J_0(0) = 0$.
Furthermore, Properties (3) and (4) from above show that $J_0(x)$ only depends on $x_1$
and that $t \mapsto J_0(t, 0, \dots, 0)$ is increasing.
In combination, these observations imply that
\begin{equation}\label{eq:IdentityBehaviorPreparation}
  J_0(x) < 0 \fa x \in \Omega' \text{ with } x_1 < 0 ,
  \quad \text{and} \quad
  J_0(x) > 0 \fa x \in \Omega' \text{ with } x_1 > 0 .
\end{equation}
Finally, with $\Psi = \big( (A_1, b_1),\dots,(A_{L-1}, b_{L-1}) \big)$, define
\(
  \Phi
  \coloneqq \big( (A_1 A, b_1 - A_1 A x^\ast), (A_2, b_2), \dots, (A_{L-1}, b_{L-1}) \big)
\),
and note that $\Phi \in \cN ( (d,1,\dots,1))$ with $L-1$ layers, and
$\Realization_\varrho^{\R^d} (\Phi)(x) = \Realization_\varrho^{\R^d} (\Psi)(A (x - x^\ast))$
for all $x \in \R^d$.
Combining this with the definition of $A$ and with Equation~\eqref{eq:IdentityBehaviorPreparation},
and noting that $A(x - x^\ast) \in \Omega'$ for $x \in \Omega$, we see that
$J := \Realization_\varrho^\Omega (\Phi)$ satisfies
\begin{equation}
  \begin{cases}
    J(x) < 0 , & \text{for } x \in \Omega \cap H_- (x^\ast, v), \\
    J(x) > 0 , & \text{for } x \in \Omega \cap H_+ (x^\ast, v), \\
    J(x) = 0 , & \text{for } x \in \Omega \cap H_0 (x^\ast, v),
  \end{cases}
  \label{eq:IdentityBehavior}
\end{equation}
where $H_0(x^\ast, v) := \R^d \setminus (H_- (x^\ast, v) \cup H_+ (x^\ast, v))$.

We now distinguish the cases given in Assumption (iv)(a) and (b) of Theorem~\ref{thm:nonclosed}.

\medskip{}

\textbf{Case 1:}
$\varrho$ is unbounded, so that necessarily Assumption (iv)(a) of Theorem~\ref{thm:nonclosed} holds,
and $\widetilde{N}_{L-1} = 2$.
For $n \in \N$ let $\Phi_n = \big( (A_1^n,b_1^n),(A_2^n,b_2^n) \big) \in \cN((1,2,1))$ be given by
\[
  A_1^n = \begin{pmatrix} n \\ n \end{pmatrix} \in \R^{2 \times 1},
  \quad
  b_1^n = \begin{pmatrix} 0 \\ -1 \end{pmatrix} \in \R^2,
  \quad
  A_2^n = \begin{pmatrix} 1 & -1 \end{pmatrix} \in \R^{1 \times 2},
  \quad
  b_2^n = 0 \in \R^1.
\]
Then, $\Phi_n \conc \Phi \in \cN((d,\widetilde{N}_1,\dots,\widetilde{N}_{L-1},1))$.
Now, let us define
\[
  h_n \coloneqq \Realization^{\Omega}_{\varrho}(\Phi_n \conc \Phi) ,
  \quad \text{and note} \quad
  h_n (x) = \varrho(nJ(x)) -  \varrho(nJ(x)-1)
  \quad \text{ for } x \in \Omega .
\]
Then, since $h_n$ is continuous and hence bounded on the compact set $\Omega$,
we see that $h_n \in L^p(\mu)$ for every $n \in \N$ and all $p \in (0, \infty]$.

We now show that $(h_n)_{n \in \N}$ converges to a discontinuous limit.
To see this, first consider $x \in \Omega \cap H_+ (x^\ast, v)$. 
Since $J(x) > 0$ by \eqref{eq:IdentityBehavior}, there exists some $N_x \in \N$
such that for all $n \geq N_x$, the estimate $nJ(x) - 1 > r$ holds,
where $r > 0$ is as in Assumption (iii) of Theorem~\ref{thm:nonclosed}.
Hence, by the mean value theorem, there exists some
$\xi_n^x\in [nJ(x)-1,nJ(x)]$ such that
\[
  \lim_{n \to \infty} h_n(x)
  = \lim_{n \to \infty} \varrho'(\xi_n^x)
  = \lambda ,
\]
since $\xi_n^x \to \infty$ as $n \to \infty, n \geq N_x$.
Analogously, it follows for $x \in \Omega \cap H_- (x^\ast, v)$
that $\lim_{n\to \infty} h_n(x) = \lambda'$.
Hence, setting $\gamma := \varrho(0) - \varrho(-1)$, we see for each $x \in \Omega$ that
\[
  \lim_{n \to \infty} h_n (x)
  = \left(
      \lambda \cdot \Indicator_{H_+ (x^\ast, v)}
      + \gamma \cdot \Indicator_{H_0 (x^\ast, v)}
      + \lambda' \cdot \Indicator_{H_- (x^\ast, v)}
    \right)(x)
  =:h(x) .
\]

We now claim that there is some $M > 0$ such that
$|\varrho(x) - \varrho(x - 1)| \leq M$ for all $x \in \R$.
To see this, note because of $\varrho'(x) \to \lambda$ as $x \to \infty$
and because of $\varrho'(x) \to \lambda'$ as $x \to -\infty$ that there are
$M_0 > 0$ and $R > r$ with $|\varrho'(x)| \leq M_0$ for all $x \in \R$ with $|x| \geq R$.
Hence, $\varrho$ is $M_0$-Lipschitz on $(-\infty,-R]$ and on $[R, \infty)$, so that
$| \varrho(x) - \varrho(x-1) | \leq M_0$ for all $x \in \R$ with $|x| \geq R+1$.
But by continuity and compactness, we also have
$| \varrho(x) - \varrho(x-1) | \leq M_1$ for all $|x| \leq R+1$ and some
constant $M_1 > 0$. Thus, we can simply choose $M \coloneqq \max \{M_0, M_1\}$.

By what was shown in the preceding paragraph,
we get $|h_n| \leq M$ and hence also $|h| \leq M$ for all $n \in \N$.
Hence, by the dominated convergence theorem, we see for any $p \in (0,\infty)$ that
\(
  \lim_{n\to \infty}
    \left\|
      h_n - h
    \right\|_{L^p(\mu)}
  = 0 .
\)
But since $\lambda \neq \lambda'$, Lemma~\ref{lem:DiscontinuousFunctionWithoutContinuousRepresentative}
shows that $h$ doesn't have a continuous representative, even after changing it on a $\mu$-null-set.
This yields the required non-continuity of a limit point as discussed at the
beginning of the proof.


\medskip{}

\textbf{Case 2:}
$\varrho$ is bounded, so that $\widetilde{N}_{L-1} = 1$.
Since $\varrho$ is monotonically increasing, there exist $c,c'\in \R$ such that
\begin{align*}
    \lim_{x\to \infty} \varrho(x) = c
    \quad \text{ and } \quad
    \lim_{x\to -\infty} \varrho(x) = c' .
\end{align*}
By the monotonicity and since $\varrho$ is not constant (because of $\varrho' (x_0) \neq 0$),
we have $c > c'\vphantom{\sum_j}$.

For each $n\in \N$, we now consider the neural network
\(
  \widetilde{\Phi}_n
  = \big( (\widetilde{A}_1^n,\widetilde{b}_1^n),(\widetilde{A}_2^n,\widetilde{b}_2^n) \big)
  \in \cN ((1, 1, 1))
\)
given by
\[
  \widetilde{A}_1^n = n \in \R^{1 \times 1},
  \quad
  \widetilde{b}_1^n = 0 \in \R^1,
  \quad
  \widetilde{A}_2^n = 1 \in \R^{1 \times 1},
  \quad
  \widetilde{b}_2^n = 0 \in \R^1 .
\]
Then, $\widetilde{\Phi}_n\conc \Phi\in \cN( (d, \widetilde{N}_1, \dots, \widetilde{N}_{L-1}, 1))$.
Now, let us define
\[
  \widetilde{h}_n
  \coloneqq \Realization^{\Omega}_{\varrho}(\widetilde{\Phi}_n\conc \Phi)
  \quad \text{and note} \quad
  \widetilde{h_n} (x) = \varrho(nJ(x))
  \quad \text{for all } x \in \Omega .
\]

Since each of the $\widetilde{h}_n$ is continuous and $\Omega$ is compact,
we have $\widetilde{h}_n\in L^p(\mu)$ for all $p\in (0,\infty]$.
Equation~\eqref{eq:IdentityBehavior} implies that $J(x) > 0$
for all $x \in \Omega \cap H_+ (x^\ast, v)$.
This in turn yields that
\begin{align}\label{eq:limit1}
     \lim_{n \to \infty} \widetilde{h}_n(x) = c
     \quad \fa x \in \Omega \cap H_+ (x^\ast, v) . 
\end{align}
Similarly, the fact that $J(x) < 0$ for all $x \in \Omega \cap H_- (x^\ast, v)$ yields
\begin{align}\label{eq:limit2}
     \lim_{n \to \infty} \widetilde{h}_n(x) = c'
     \quad \fa x \in \Omega \cap H_- (x^\ast, v). 
\end{align}
Combining \eqref{eq:limit1} with \eqref{eq:limit2} yields
for all $x \in \Omega$ that
\[
  \lim_{n\to \infty}\widetilde{h}_n (x)
  = \left(
       c \cdot \Indicator_{H_+ (x_\ast, v)}
      + \varrho(0) \cdot \Indicator_{H_0 (x^\ast, v)}
      +c' \cdot \Indicator_{H_- (x^\ast, v)}
    \right)(x)
  =:\widetilde{h}(x) .
\]
%
By the boundedness of $\varrho$, we get $|\widetilde{h}_n (x)| \leq C$
for all $n \in \N$ and $x \in \Omega$ and a suitable $C > 0$,
so that also $\widetilde{h}$ is bounded.
Together with the dominated convergence theorem, this implies for any $p \in (0,\infty)$ that
\(
  \lim_{n \to \infty}
    \big\|
      \widetilde{h}_n - \widetilde{h}
    \big\|_{L^p(\mu)}
  =0.
\)
Since $c \neq c'$, Lemma~\ref{lem:DiscontinuousFunctionWithoutContinuousRepresentative}
shows that $\widetilde{h}$ does not have a continuous representative
(with respect to equality $\mu$-almost everywhere).
This yields the required non-continuity of a limit point as discussed at the
beginning of the proof.
\hfill$\square$

\subsection{Proof of Corollary~\ref{cor:StuffNeverClosedInLp}}
\label{app:StuffNeverClosedInLp}

It is not hard to verify that all functions listed in Table~\ref{tab:ActFunctions} are continuous
and increasing.
Furthermore, each activation function $\varrho$ listed in Table~\ref{tab:ActFunctions} is not constant and satisfies
$\varrho|_{\R\setminus \{0\}} \in C^\infty (\R \setminus \{0\})$.
This shows that $\varrho|_{(-\infty,-r) \cup (r,\infty)}$ is differentiable for any $r > 0$,
and that there is some $x_0 = x_0 (\varrho) \in \R$  such that $\varrho'(x_0) \neq 0$.

Next, the softsign, the inverse square root unit, the sigmoid, the $\tanh$, and the $\arctan$
function are all bounded, and thus satisfy condition (iv)(b) of Theorem~\ref{thm:nonclosed}.
Thus, all that remains is to verify condition (iv)(a) of Theorem~\ref{thm:nonclosed}
for the remaining activation functions:

\begin{enumerate}
  \item For the ReLU $\varrho(x) = \max\{0, x\}$, condition (iv)(a) is satisfied
        with $\lambda = 1$ and $\lambda' = 0 \neq \lambda$.

  \item For the parametric ReLU $\varrho(x) = \max \{a x, x\}$ (with $a \geq 0$, $a \neq 1$),
        Condition~(iv)(a) is satisfied with $\lambda = \max \{ 1, a \}$
        and $\lambda' = \min \{ 1, a \}$, where $\lambda \neq \lambda'$ since $a \neq 1$.

  \item For the exponential linear unit
        $\varrho(x) = x \Indicator_{[0,\infty)}(x) + (e^x - 1) \Indicator_{(-\infty,0)}(x)$,
        Condition~(iv)(a) is satisfied for $\lambda = 1$
        and $\lambda' = \lim_{x \to -\infty} e^x = 0 \neq \lambda$.

  \item For the inverse square root linear unit
        $\varrho(x)
         = x \Indicator_{[0,\infty)}(x)
           + \frac{x}{\sqrt{1 + a x^2}} \Indicator_{(-\infty,0)}(x)$,
        the quotient rule shows that for $x<0$ we have
        \begin{equation}
          \varrho'(x)
          = \frac{\sqrt{1 + a x^2} - x \cdot \frac{1}{2} (1 + a x^2)^{-1/2} 2 a x}{1 + a x^2}
          = \frac{(1 + a x^2) - a x^2}{(1 + a x^2)^{3/2}}
          = (1 + a x^2)^{-3/2} .
          \label{eq:ISQRLUDerivative}
        \end{equation}
        Therefore, Condition~(iv)(a) is satisfied for $\lambda = 1$ and
        $\lambda' = \lim_{x \to -\infty} \varrho'(x) = 0 \neq \lambda$.

  \item For the softplus function $\varrho(x) = \ln(1 + e^x)$, Condition~(iv)(a) is satisfied for
        \begin{flalign*}
          &&
          \lambda = \lim_{x \to \infty} \frac{e^x}{1 + e^x} = 1
          \quad \text{and} \quad
          \lambda' = \lim_{x \to -\infty} \frac{e^x}{1 + e^x} = 0 \neq \lambda .
          && \square
        \end{flalign*}
\end{enumerate}

\subsection{Proof of Theorem~\ref{thm:GeneralNonClosednessC}}
\label{app:GeneralNonClosednessC}

\subsubsection{Proof of Theorem~\ref{thm:GeneralNonClosednessC} under Condition (i)}
\label{app:SmoothnessNoClosedness}

Let $\Omega \coloneqq [-B,B]^d$.
Let $m \in \N$ be maximal with $\varrho \in C^m (\R)$;
this is possible since $\varrho \in C^1 (\R) \setminus C^\infty (\R)$.
Note that $\varrho \in C^{m}(\R) \setminus C^{m+1}(\R)$.
This easily implies $\cRN_{\varrho}^{\Omega}(S) \subset C^{m}({\Omega})$.

We now show for the architecture
$S' := (d, \widetilde{N}_1, \dots, \widetilde{N}_{L-2}, 2, 1)$,
where $\widetilde{N}_i := 1$ for all $i = 1, \dots, L-2$,
that the set $\cRN_\varrho^{\Omega}(S')$ is not closed in $C(\Omega)$.
If we had $\varrho \in C^{m+1}([-C,C])$ for all $C > 0$, this would imply $\varrho \in C^{m+1}(\R)$;
thus, there is $C > 0$ such that $\varrho \notin C^{m+1} ([-C,C])$.
Now, choose $\lambda > C / B$, so that $\lambda [-B,B] \supset [-C,C]$.
This entails that $\varrho(\lambda \cdot)\in C^{m}([-B,B]) \setminus C^{m+1}([-B,B])$.
Next, since the continuous derivative
$\frac{d}{dx} \varrho (\lambda x) = \lambda \varrho' (\lambda x)$ is bounded
on the compact set $[-B,B]$, we see that $\varrho(\lambda \cdot)$ is Lipschitz continuous
on $[-B,B]$, and we set $M_1 \coloneqq \mathrm{Lip}(\varrho(\lambda \cdot))$.
Next, by the uniform continuity of $\lambda \cdot \varrho'(\lambda \cdot)$
on $[-(B+1),B+1]$, if we set
\[
  \eps_n
  \coloneqq \sup_{\substack{x,y \in [-(B+1),B+1]\\ \text{with } |x-y| \leq 1/n}}
       |\lambda \cdot \varrho'(\lambda x) - \lambda \cdot \varrho'(\lambda y)| ,
\]
then $\eps_n \to 0$ as $n \to \infty$.

For $n\in \N$, let
$\Phi_n^1 = \big((A_1^n,b_1^n),(A_2^n,b_2^n)\big)\in \cN((1,2,1))$ be given by
\[
  A_1^n = \begin{pmatrix} \lambda \\ \lambda \end{pmatrix} \in \R^{2 \times 1},
  \quad
  b_1^n = \begin{pmatrix} \lambda / n \\ 0 \end{pmatrix} \in \R^2,
  \quad
  A_2^n = \begin{pmatrix} n & -n \end{pmatrix} \in \R^{1 \times 2},
  \quad
  b_2^n = 0 \in \R^1 .
\]
Note that there is some $x^\ast \in \R$ such that $\varrho' (x^\ast) \neq 0$,
since otherwise $\varrho' \equiv 0$ and hence $\varrho \in C^\infty (\R)$.
Thus, for each $n \in \N$, Proposition~\ref{prop:Identity} yields a neural network
$\Phi^2_n \in \cN ((d,1,\dots,1))$ with $L-1$ layers such that
\begin{align}\label{eq:APropertyOfPhi2n}
  \left| \Realization_{\varrho}^{\Omega}(\Phi^2_n)(x) - x_1 \right|
  \leq \frac{1}{2 n^2}
  \text{ for all } x \in \Omega .
\end{align}
We set $\Phi_n \coloneqq \Phi^1_n \conc \Phi^2_n\in \cN(S')$
and $f_n \coloneqq \Realization_\varrho^\Omega (\Phi_n)$.
For $x \in \Omega$, we then have
\begin{align*}
  | f_n (x) - \lambda\varrho'(\lambda x_1) |
  = \left|
         n\cdot \left(
            \varrho\left(
                       \lambda \Realization_{\varrho}^{\Omega}(\Phi^2_n)(x)
                     + \lambda \cdot n^{-1} 
                   \right)
            - \varrho(\lambda \Realization_{\varrho}^{\Omega}(\Phi^2_n)(x))
          \right)
          - \lambda\varrho'(\lambda x_1)
       \right| .
\end{align*}
Now, by the Lipschitz continuity of $\varrho(\lambda \cdot)$ and
Equation~\eqref{eq:APropertyOfPhi2n}, we conclude that
\[
  \left|
    n \cdot \left(
        \varrho\left(
                  \lambda \Realization_{\varrho}^{\Omega}(\Phi^2_n)(x)
                  + \lambda \cdot n^{-1} 
                \right)
        - \varrho(\lambda \Realization_{\varrho}^{\Omega}(\Phi^2_n)(x))
     \right)
     - n \cdot \left(
           \varrho\left(
                    \lambda \left(x_1 + n^{-1}\right)
                  \right)
           - \varrho(\lambda x_1)
         \right)
  \right|
  \leq \frac{M_1 \, \lambda}{n} .
\]
This implies for every $x \in \Omega$ that
\begin{align*}
  | f_n (x) - \lambda \varrho'(\lambda x_1) |
  & \leq \left|
            n \left(
                \varrho\left(\lambda \left(x_1 + n^{-1} \right) \right)
                - \varrho(\lambda x_1)
              \right)
            - \lambda\varrho'(\lambda x_1)
         \right| + \frac{M_1 \, \lambda}{n} \\
  ({\scriptstyle{\text{by the mean value theorem, }
                 \xi_{n}^x \in (x_1, x_1 + n^{-1})}})
  & = \left|
         \lambda \cdot \varrho'(\lambda \cdot \xi_{n}^x)
         - \lambda\varrho'(\lambda x_1)
      \right| + \frac{M_1 \, \lambda}{n} \\
  & \leq \eps_n + \frac{M_1 \, \lambda}{n} .
\end{align*}
Here, the last step used that $| \xi_{n}^x - x_1 | \leq n^{-1} \leq 1$,
so that $x_1, \xi_{n}^x \in [-(B+1),B+1]$.
\smallskip{}

Overall, we established the existence of a sequence $(f_n)_{n \in \N}$ in
$\cRN_{\varrho}^{\Omega}(S')$ which converges uniformly to the function
$\Omega \to \R, ~x\mapsto \varrho_\lambda(x) \coloneqq \lambda \, \varrho'(\lambda x_1)$.
By our choice of $\lambda$, we have $\varrho_\lambda \not \in C^{m}({\Omega})$.
Because of $\cRN_{\varrho}^{\Omega}(S') \subset C^{m}({\Omega})$,
we thus see that $\varrho_\lambda \not \in \cRN_{\varrho}^{\Omega}(S')$,
so that $\cRN_{\varrho}^{\Omega}(S')$ is not closed in $C(\Omega)$.

Finally, note by Lemma~\ref{lem:enlarge} that
\[
  f_n
  \in \cRN_{\varrho}^{\Omega}(S')
  \subset \cRN_{\varrho}^{\Omega}(S)
  \quad \text{for all} \quad n \in \N .
\]
Since $f_n \to \varrho_\lambda$ uniformly, where $\varrho_\lambda \notin C^m (\Omega)$,
and hence $\varrho_\lambda \notin \cRN_{\varrho}^{\Omega}(S)$, we thus see that
$\cRN_{\varrho}^{\Omega}(S)$ is not closed in $C(\Omega)$.
\hfill$\square$

\subsubsection{Proof of Theorem~\ref{thm:GeneralNonClosednessC} under Condition (ii)}
\label{app:closedAnalytic}

Let $\Omega \coloneqq [-B,B]^d$.
We first show that if we set $S' := (d, \widetilde{N}_1, \dots, \widetilde{N}_{L-2},2,1)$,
where $\widetilde{N}_i \coloneqq 1$ for all ${i = 1, \dots, L-2}$,
then there exists a limit point of $\cRN_\varrho^{\Omega}(S')$
which is the restriction $f|_\Omega$ of an \emph{unbounded} analytic function $f : \R \to \R$.

Since $\varrho$ is not constant, there is some $x^\ast \in \R$ such that $\varrho'(x^\ast) \neq 0$.
For $n \in \N$, let us define
${\Phi_n^1 \coloneqq \big( (A_1^n,b_1^n),(A_2^n,b_2^n) \big) \in \cN((1,2,1))}$ by
\[
  A_1^n \coloneqq \begin{pmatrix} 1 \\ 1/n \end{pmatrix} \in \R^{2 \times 1},
  \quad
  b_1^n \coloneqq \begin{pmatrix} 0 \\ x^\ast  \end{pmatrix} \in \R^2,
  \quad
  A_2^n \coloneqq \begin{pmatrix} 1 & n \end{pmatrix} \in \R^{1 \times 2},
  \quad
  b_2^n \coloneqq - \varrho(x^\ast)n \in \R^1 .
\]
With this choice, we have
\[
  \Realization_\varrho^{\R}(\Phi_n^1)(x)
  = \varrho(x) + n \cdot \big( \varrho(x/n + x^\ast) - \varrho(x^\ast) \big)
  \quad \text{for all } x \in \R .
\]
For any $x \in \R$, the mean-value theorem yields $\widetilde{x}$
between $x^\ast$ and $x^\ast + \frac{x}{n}$ satisfying
$\varrho(x^\ast + \frac{x}{n}) - \varrho(x^\ast) = \frac{x}{n} \cdot \varrho' (\widetilde{x})$.
Therefore, if $B > 0$ and $x \in [-B,B]$, then
\begin{align*}
  \left|
    \Realization_\varrho^{\R}(\Phi_n^1)(x) - \big( \varrho(x) + \varrho'(x^\ast) x \big)
  \right|
  \leq B |\varrho'(\widetilde{x}) - \varrho'(x^\ast)|
  \quad \text{for some} \quad \widetilde{x} \in [x^\ast - B/n, x^\ast + B/n] .
\end{align*}
Since $\varrho'$ is continuous, we conclude that
\begin{align}\label{eq:JustAnotherEstimateWithoutAName}
  \sup_{x \in [-B,B]}
    \left|
      \Realization_\varrho^{\R}(\Phi_n^1)(x)
      - \big( \varrho(x) + \varrho'(x^\ast) x \big)
    \right|
  \xrightarrow[n\to\infty]{} 0 .
\end{align}
Moreover, note that
\(
  \frac{d}{dx}\Realization_\varrho^{\R}(\Phi_n^1)(x)
  = \varrho'(x) + \varrho'(x^\ast + n^{-1} \cdot x)
\)
is bounded on $[-(B+1),B+1]$, uniformly with respect to $n \in \N$.
Hence, $\Realization_\varrho^{\R}(\Phi_n^1)$ is Lipschitz continuous on
$[-(B+1), B+1]$, with Lipschitz constant $C' > 0$ independent of $n \in \N$.

Next, for each $n \in \N$, Proposition~\ref{prop:Identity} yields
a neural network $\Phi^2_n \in \cN((d,1,\dots,1))$ with $L-1$ layers such that
\begin{align}\label{eq:AnotherPropertyOfPhi2n}
  \left| \Realization_{\varrho}^{\Omega}(\Phi^2_n)(x) - x_1 \right|
  \leq \frac{1}{n},
  \text{ for all } x \in \Omega .
\end{align}
We set $\Phi_n \coloneqq \Phi_n^1 \conc \Phi_n^2 \in \cN(S')$ and note for all $x \in \Omega$ that
\begin{align*}
  \left|
    \Realization_\varrho^{\Omega}(\Phi_n)(x)
    - \big( \varrho(x_1) + \varrho'(x^\ast) x_1 \big)
  \right|
  = \left|
      \Realization_\varrho^{\R}(\Phi_n^1)
        \big( \Realization_\varrho^{\Omega}(\Phi_n^2)(x) \big)
      - \big( \varrho(x_1) + \varrho'(x^\ast) x_1 \big)
    \right|.
\end{align*}
By the Lipschitz continuity of $\Realization_\varrho^{\R}(\Phi_n^1)$
on $[-(B+1), B+1]$, and using \eqref{eq:AnotherPropertyOfPhi2n}, we conclude that
\[
  \left|
    \Realization_\varrho^{\Omega}(\Phi_n)(x)
    - \big( \varrho(x_1) + \varrho'(x^\ast) x_1 \big)
  \right|
  \leq \left|
         \Realization_\varrho^{\R}(\Phi_n^1) (x_1)
         - \big( \varrho(x_1) + \varrho'(x^\ast) x_1 \big)
       \right|
       + \frac{C'}{n} ,
\]
so that an application of \eqref{eq:JustAnotherEstimateWithoutAName} yields
\[
  \sup_{x \in \Omega}
    \left|
      \Realization_\varrho^{\Omega}(\Phi_n)(x)
      - \big( \varrho(x_1) + \varrho'(x^\ast) x_1 \big)
    \right|
  \xrightarrow[n\to\infty]{} 0 .
\]
Now, to show that $\cRN_\varrho^{\Omega}(S) \subset C(\Omega)$ is not closed,
since
\({
  \Realization_\varrho^\Omega (\Phi_n) \in \cRN_\varrho^{\Omega}(S') \subset \cRN_\varrho^{\Omega}(S),
}\)
it is sufficient to show that with
\[
  F: \R^d \to \R, \quad  x \mapsto \varrho(x_1) + \varrho'(x^\ast) x_1,
\]
$F|_{\Omega}$ is not an element of $\cRN_\varrho^{\Omega}(S)$.
This is accomplished, once we show that there do not exist any
$\widehat{N}_{1}, \dots, \widehat{N}_{L-1} \in \N$ such that $F|_{\Omega}$ is an element of
$\cRN_\varrho^{\Omega}( (d, \widehat{N}_1, \dots, \widehat{N}_{L-1}, 1) )$.

Towards a contradiction, we assume that there exist $\widehat{N}_1, \dots, \widehat{N}_{L-1} \in \N$
such that $F|_{\Omega} = \Realization_\varrho^{\Omega}(\Phi^3)$ for a network
$\Phi^3 \in \cN( (d, \widehat{N}_1, \dots, \widehat{N}_{L-1}, 1) )$.
Since $F$ and $ \Realization_\varrho^{\R^{d}}(\Phi^3)$ are both analytic functions
that coincide on $\Omega = [-B,B]^d$, they must be equal on all of $\R^d$.
However, $F$ is unbounded (since $\varrho$ is bounded, and since $\varrho'(x^\ast) \neq 0$),
while $\Realization_\varrho^{\R^{d}}(\Phi^3)$ is bounded as a consequence of $\varrho$ being bounded.
This produces the desired contradiction.
\hfill$\square$

\subsubsection{Proof of Theorem~\ref{thm:GeneralNonClosednessC} under Condition (iii)}
\label{app:closedHomogeneous}

Let $\varrho \in C^{\max\{r,q\}}(\R)$ be approximately homogeneous of order $(r, q)$ with $r \neq q$.
For simplicity, let us assume that $r > q$;
we will briefly comment on the case $q > r$ at the end of the proof.

Note that $r \geq 1$, since $r,q \in \N_0$ with $r > q$.
Let $(x)_+ \coloneqq \max\{x,0\}$ for $x \in \R$.
We start by showing that
\begin{align}\label{eq:converge}
  k^{-r} \varrho(k \cdot)
  \xrightarrow[k \to \infty]{\text{uniformly on } [-B,B]} (\cdot)_+^r .
\end{align}
To see this, let $s > 0$ such that $|\varrho(x) - x^r| \leq s$ for all $x> 0$ and
$|\varrho(x) - x^q| \leq s$ for all $x < 0$.
For any $k \in \N$ and $x \in [-B,0]$, we have
\[
  | k^{-r} \, \varrho(k x) - (x)_+^r |
  = |k^{-r} \, \varrho(k x)|
  \leq k^{-r} \cdot \big( |\varrho(k x) - (k x)^q| + |(k x)^q| \big)
  \leq k^{-r} \cdot (s + k^q B^q)
  \leq c_0 \cdot k^{-1}
\]
for a constant $c_0 = c_0(B, s, r, q) > 0$.
Moreover, for $x \in [0,B]$, we have
\[
  | k^{-r} \, \varrho(k x) - (x)_+^r |
  = k^{-r} |\varrho(k x) - (k x)^r|
  \leq s \cdot k^{-r} .
\]
Overall, we conclude that
\[
  \sup_{x\in [-B,B]}
    |k^{-r} \varrho(k x) - (x)_+^r|
  \leq \max\{c_0, s\} \cdot k^{-1} ,
\]
which implies \eqref{eq:converge}.

We observe that $\big( x \mapsto (x)_+^r \big) \not \in C^{r}([-B,B])$.
Additionally, since $\varrho \in C^{\max\{r,q\}}(\R) = C^r (\R)$, we have
$\cRN_\varrho^{[-B,B]^d}(S) \subset C^r([-B,B]^d)$.
Hence, the proof is complete if we can construct a sequence $(\Phi_n)_{n \in \N}$ of neural networks
in $\cN((d,1,\dots,1))$ (with $L$ layers) such that the $\varrho$-realizations
$\Realization_\varrho^\Omega (\Phi_n)$ converge uniformly
to the function $[-B,B]^d \to \R, x \mapsto (x_1)_+^r$.
By the preceding considerations, this is clearly possible, as can be seen
by the same arguments used in the proofs of the previous results.
For invoking these arguments, note that $\max\{r,q\} \geq 1$, so that $\varrho \in C^1 (\R)$.
Also, since $\varrho$ is approximately homogeneous of order $(r,q)$ with $r \neq q$,
$\varrho$ cannot be constant, and hence $\varrho'(x_0) \neq 0$ for some $x_0 \in \R$.

\medskip{}

For completeness, let us briefly consider the case where $q > r$
that was omitted at the beginning of the proof.
In this case, $(-k)^{-q} \, \varrho(-k \cdot) \to (\cdot)_+^q$
with uniform convergence on $[-B,B]$.
Indeed, for $x \in [0,B]$, we have
$|(-k)^{-q} \, \varrho(-k x) - (x)_+^q|
 = k^{-q} |\varrho(-k x) - (-k x)^q|
 \leq k^{-q} \cdot s \leq s \cdot k^{-1}$.
Similarly, for $x \in [-B, 0]$, we get
$| (-k)^{-q} \, \varrho(-k x) - (x)_+^q |
 \leq k^{-q} \big( |\varrho(-k x) - (-k x)^r| + |(-k x)^r| \big)
 \leq k^{-q} (s + B^r k^r)
 \leq c_1 \cdot k^{-1}$
for some constant $c_1 = c_1(B,s,r,q) > 0$.
Here, we used that $q - r \geq 1$, since $r,q \in \N_0$ with $q > r$.
Now, the proof proceeds as before, noting that $\big( x \mapsto (x)_+^q \big) \notin C^q ([-B,B])$,
while $\varrho \in C^{\max\{r,q\}} (\R) \subset C^q (\R)$.
\hfill$\square$

\subsection{Proof of Corollary~\ref{cor:GeneralNonClosednessC}}
\label{app:ProofCorollaryNonClosed}

\subsubsection{Proof of Corollary~\ref{cor:GeneralNonClosednessC}.(1)}
\label{app:LimitedSmoothnessNonClosedness}

\textbf{Powers of ReLUs:}
For $k \in \N$, let $\mathrm{ReLU}_k : \R \to \R, x \mapsto \max \{0, x\}^k$,
and note that this is a continuous function.
On $\R \setminus \{0\}$, $\mathrm{ReLU}_k$ is differentiable with
$\mathrm{ReLU}_k ' = k \cdot \mathrm{ReLU}_{k - 1}$.
Furthermore, if $k \geq 2$, then
$| h^{-1} (\mathrm{ReLU}_k (h) - \mathrm{ReLU}_k (0))| \leq |h|^{k-1} \to 0$ as $h \to 0$.
Thus, if $k \geq 2$, then $\mathrm{ReLU}_k$ is continuously differentiable with derivative
$\mathrm{ReLU}_k ' = k \cdot \mathrm{ReLU}_{k-1}$.
Finally, $\mathrm{ReLU}_1$ is \emph{not} differentiable at $x = 0$.
Overall, this shows $\mathrm{ReLU}_k \in C^1 (\R) \setminus C^\infty (\R)$ for all $k \geq 2$.

\medskip{}

\noindent
\textbf{The exponential linear unit:}
We have $\frac{d^k}{d x^k} (e^x - 1) = e^x$ for all $k \in \N$.
Therefore, the exponential linear unit
$\varrho : \R \to \R, x \mapsto x \Indicator_{[0,\infty)}(x) + (e^x - 1) \Indicator_{(-\infty,0)}(x)$
satisfies for $k \in \N_0$ that
\[
  \lim_{x \downarrow 0} \varrho^{(k)} (x) = \delta_{k,1}
  \qquad \text{and} \qquad
  \lim_{x \uparrow 0} \varrho^{(k)} (x)
  = \begin{cases}
      \lim_{x \uparrow 0} (e^x - 1) = 0, & \text{if } k = 0 , \\
      \lim_{x \uparrow 0} e^x = 1,       & \text{if } k \neq 0 .
    \end{cases}
\]
By standard results in real analysis
(see for instance \cite[Problem~2 in Chapter~VIII.6]{DieudonneFoundations}),
this implies that $\varrho \in C^1 (\R) \setminus C^2 (\R)$.

\medskip{}

\noindent
\textbf{The softsign function:}
On $(-1,\infty)$, we have $\frac{d}{dx} \frac{x}{1 + x} = (1 + x)^{-2}$ and
$\frac{d^2}{d x^2} \frac{x}{1 + x} = -2 (1 + x)^{-3}$.
Furthermore, if $x < 0$, then
$\mathrm{softsign}(x) = \frac{x}{1 + |x|} = - \frac{-x}{1 + (-x)} = - \mathrm{softsign}(-x)$.
Therefore, the softsign function is $C^\infty$ on $\R \setminus \{0\}$,
and satisfies
\[
  \lim_{x \downarrow 0} \mathrm{softsign}'(x)
  = \lim_{x \downarrow 0} (1 + x)^{-2}
  = 1
  = \lim_{x \uparrow 0} (1 - x)^{-2}
  = \lim_{x \uparrow 0} \mathrm{softsign} ' (x) .
\]
By standard results in real analysis
(see for instance \cite[Problem~2 in Chapter~VIII.6]{DieudonneFoundations}),
this implies that $\mathrm{softsign} \in C^1 (\R)$.
However, since
\[
  \lim_{x \downarrow 0} \mathrm{softsign}''(x)
  = \lim_{x \downarrow 0} -2 (1 + x)^{-3}
  = -2
  \quad \text{and} \quad
  \lim_{x \uparrow 0} \mathrm{softsign}''(x)
  = \lim_{x \uparrow 0} 2 (1 - x)^{-3}
  = 2 ,
\]
we have $\mathrm{softsign} \notin C^2 (\R)$.

\medskip{}

\noindent
\textbf{The inverse square root linear unit:}
Let
\(
  \varrho :
  \R \to \R,
  x \mapsto x \Indicator_{[0,\infty)}(x) + \frac{x}{(1 + a x^2)^{1/2}} \Indicator_{(-\infty,0)}(x)
\)
denote the inverse square root linear unit with parameter $a > 0$,
and note $\varrho|_{\R\setminus \{0\}} \in C^\infty (\R \setminus \{0\})$.
As we saw in Equation~\eqref{eq:ISQRLUDerivative}, we have
$\frac{d}{dx} \frac{x}{(1 + a x^2)^{1/2}} = (1 + a x^2)^{-3/2}$,
and thus $\frac{d^2}{d x^2} \frac{x}{(1 + a x^2)^{1/2}} = - 3 a x \cdot (1 + a x^2)^{-5/2}$,
and finally
\(
  \frac{d^3}{d x^3} \frac{x}{(1 + a x^2)^{1/2}}
  = -3 a (1 + a x^2)^{-5/2} + 15 a^2 x^2 (1 + a x^2)^{-7/2}
  .
\)
These calculations imply
\[
  \lim_{x \uparrow 0} \varrho'(x)
  = \lim_{x \uparrow 0} (1 + a x^2)^{-3/2}
  = 1
  = \lim_{x \downarrow 0} \varrho'(x)
  \quad \text{and} \quad
  \lim_{x \uparrow 0} \varrho''(x)
  = \lim_{x \uparrow 0} -3 a x \cdot (1 + a x^2)^{-5/2}
  = 0
  = \lim_{x \downarrow 0} \varrho''(x) ,
\]
but also
\[
  \lim_{x \uparrow 0} \varrho'''(x)
  = \lim_{x \uparrow 0}
    \big[ -3 a (1 + a x^2)^{-5/2} + 15 a^2 x^2 (1 + a x^2)^{-7/2} \big]
  = - 3 a
  \neq 0
  = \lim_{x \downarrow 0} \varrho ''' (x) .
\]
By standard results in real analysis
(see for instance \cite[Problem~2 in Chapter~VIII.6]{DieudonneFoundations}),
this implies that $\varrho \in C^2 (\R) \setminus C^3 (\R)$.
\hfill$\square$

\subsubsection{Proof of Corollary~\ref{cor:GeneralNonClosednessC}.(3)}
\label{app:Softplus}

\textbf{The softplus function:}
Clearly, $\mathrm{softplus} \in C^\infty(\R) \subset C^{\max\{1,0\}}(\R)$.
Furthermore, the softplus function is approximately homogeneous of order $(1,0)$.
Indeed, for $x \geq 0$, we have
\[
  | \ln(1 + e^x) - x |
  = \Big|
      \ln \Big( \frac{1 + e^x}{e^x} \Big)
    \Big|
  = \ln(1 + e^{-x})
  \leq \ln (2) ,
\]
and for $x \leq 0$, we have $|\ln(1 + e^x) - x^0| \leq 1 + \ln(2)$.
\hfill$\square$

\subsection{Proof of Proposition~\ref{prop:bdWeights}}
\label{app:bdWeights}

The set $\Theta_C$
is closed and bounded in the normed space
\(
  \big(
    \cN(S),
    \| \cdot \|_{\cN(S)}
  \big) .
\)
Thus, the Heine-Borel Theorem implies the compactness of $\Theta_C$.
By Proposition~\ref{prop:RealizationContinuity} (which will be proved independently), the map
\[
  \Realization_{\varrho}^{\Omega}
  : \big(
      \cN(S),
      \|\cdot\|_{\cN(S)}
    \big)
    \to \big(C(\Omega),\|\cdot\|_{\sup}\big)
  \]
is continuous.
As a consequence, the set $\Realization_{\varrho}^\Omega(\Theta_C)$
is compact in $C(\Omega)$.
%
%
Because of the compactness of $\Omega$, $C(\Omega)$ is continuously embedded into $L^p(\mu)$
for every $p \in (0,\infty)$ and any finite Borel measure $\mu$ on $\Omega$.
This implies that the set $\Realization_\varrho^\Omega(\Theta_C)$
is compact in $L^p(\mu)$ as well.
\hfill$\square$

\subsection{Proof of Proposition~\ref{prop:ExplodingWeights}}
\label{app:ExplodingWeightsProof}

With $(\Phi_N)_{N \in \N}$ as in the statement of Proposition~\ref{prop:ExplodingWeights},
we want to show that $\| \Phi_N \|_{\mathrm{total}} \to \infty$ in probability.
By definition, this means that for each fixed $C > 0$, and letting $\Omega_N$ denote the event where
$\| \Phi_N \| \geq C$, we want to show that $\mathbb{P}(\Omega_N) \to 1$ as $N \to \infty$.
For brevity, let us write $\mathcal{R}^Z \coloneqq \mathrm{R}_{\varrho}^{\Omega}(\Theta_{Z})$
for $\Theta_{Z}$, $Z > 0$ as in Proposition \ref{prop:bdWeights}.

\medskip{}

By compactness of $\mathcal{R}^C$, we can choose $g \in \mathcal{R}^C$ satisfying
\(
  \| f_\prob - g \|_{L^2(\prob_\Omega)}^2
  = \inf_{h \in \mathcal{R}^C} \| f_\prob - h \|_{L^2(\prob_\Omega)}^2 .
\)
Define $M := \inf_{h \in \cRN_\varrho^\Omega (S)} \| f_\prob - h \|_{L^2(\prob_\Omega)}^2$.
Since by assumption the infimum defining $M$ is not attained, we have
$\| f_\prob - g \|_{L^2(\prob_\Omega)}^2 > M$, so that there are $h \in \cRN_\varrho^\Omega (S)$
and $\delta > 0$ with
$\| f_\prob - g \|_{L^2(\prob_\Omega)}^2 \geq 2 \delta + \| f_\prob - h \|_{L^2(\prob_\Omega)}^2$.
Let $C' > 0$ with $h \in \mathcal{R}^{C'}$.

For $N \in \N$ and $\eps > 0$, let us denote by $\Omega_{N,\eps}^{(1)}$ the event where
\(
  \sup_{f \in \mathcal{R}^{C'}}
    |\mathcal{E}_\prob (f) - E_N (f)|
  > \eps .
\)
Since $\mathcal{R}^{C'}$ is compact, \cite[Theorem~B]{Cucker02onthe} shows for arbitrary
$\eps > 0$ that $\mathbb{P}(\Omega_{N,\eps}^{(1)}) \xrightarrow[N \to \infty]{} 0$ for each fixed
$\eps > 0$.
Similarly, denoting by $\Omega_{N,\eps}^{(2)}$ the event where
$E_N (\Realization_\varrho^\Omega (\Phi_N)) - \inf_{f \in \cRN_\varrho^\Omega(S)} E_N (f) > \eps$,
we have by assumption \eqref{eq:ConvergenceOfMinimizers} that
$\mathbb{P}(\Omega_{N,\eps}^{(2)}) \xrightarrow[N \to \infty]{} 0$,
for each fixed $\eps > 0$.

\medskip{}

We now claim that $\Omega_N^c \subset \Omega_{N, \delta/3}^{(1)} \cup \Omega_{N, \delta/3}^{(2)}$.
Once we prove this, we get
\[
  0 \leq \mathbb{P} (\Omega_N^c)
    \leq \mathbb{P} \big( \Omega_{N,\delta/3}^{(1)} \big)
         + \mathbb{P} \big( \Omega_{N,\delta/3}^{(2)} \big)
    \xrightarrow[N \to \infty]{} 0,
\]
and hence $\mathbb{P}(\Omega_N) \to 1$, as desired.

To prove $\Omega_N^c \subset \Omega_{N, \delta/3}^{(1)} \cup \Omega_{N, \delta/3}^{(2)}$,
assume towards a contradiction that there exists a training sample
\({
  \omega
  \coloneqq \big( (x_i ,y_i) \big)_{i \in \N}
  \in \Omega_N^c \setminus \big( \Omega_{N, \delta/3}^{(1)} \cup \Omega_{N, \delta/3}^{(2)} \big)
  .
}\)
Thus, $\| \Phi_N \|_{\mathrm{total}} < C$, meaning
${f_N := \Realization_\varrho^\Omega (\Phi_N) \in \mathcal{R}^C \subset \mathcal{R}^{C'}}$.
Using the decomposition of the expected loss from Equation~\eqref{eq:expectedRiskDecomposition},
we thus see
\begin{align*}
  & \| g - f_\prob \|_{L^2(\prob_\Omega)}^2 + \mathcal{E}_\prob (f_\prob) \\
  ({\scriptstyle{\text{by choice of $g$ and since } f_N \in \mathcal{R}^C}})
  & \leq \| f_N - f_\prob \|_{L^2(\prob_\Omega)}^2 + \mathcal{E}_\prob (f_\prob) 
    =    \mathcal{E}_\prob \big( f_N \big) \\
  ({\scriptstyle{\text{since } \omega \notin \Omega_{N,\delta/3}^{(1)}
                 \text{ and } f_N \in \mathcal{R}^{C'}}})
  & \leq E_N (f_N) + \frac{\delta}{3} \\
  ({\scriptstyle{\text{since } \omega \notin \Omega_{N,\delta/3}^{(2)}}})
  & \leq \frac{2}{3} \delta
         + \inf_{f \in \cRN_\varrho^\Omega (S)}
             E_N (f)
    \leq \frac{2}{3} \delta + E_N (h) \\
  ({\scriptstyle{\text{since } h \in \mathcal{R}^{C'}
                 \text{ and } \omega \notin \Omega_{N,\delta/3}^{(1)}}})
  & \leq \mathcal{E}_{\prob} (h) + \delta
    =    \mathcal{E}_{\prob} (f_\prob) + \| h - f_\prob \|_{L^2(\prob_\Omega)}^2 + \delta .
\end{align*}
By rearranging and recalling the choice of $h$ and $\delta$, we finally see
\[
  \| f_\prob - h \|_{L^2(\sigma_\Omega)}^2 + 2 \delta
  \leq \| f_\sigma - g \|_{L^2(\sigma_\Omega)}
  \leq \| f_\prob - h \|_{L^2(\sigma_\Omega)}^2 + \delta
  ,
\]
which is the desired contradiction.
\hfill$\square$

\subsection{Proof of Proposition~\ref{prop:BoundedScalingWeightsClosedness}}
\label{app:BoundedScalingWeightsClosedness}

The main ingredient of the proof will be to show that one can replace
a given sequence of networks with $C$-bounded scaling weights
by another sequence with $C$-bounded scaling weights that also has
bounded biases. Then one can apply Proposition~\ref{prop:bdWeights}.

\begin{lemma}\label{lem:WeightShiftLemma}
  Let $S = (d, N_1, \dots, N_L)$ be a neural network architecture,
  let $C > 0$ and let $\Omega \subset \R^d$ be measurable and bounded.
  Let $\mu$ be a finite Borel measure on $\Omega$ with $\mu(\Omega) > 0$.
  Finally, let ${\varrho : \R \to \R, ~x \mapsto \max\{0,x\}}$ denote the ReLU activation function.

  Let $(\Phi_n)_{n \in \N}$ be a sequence of networks in
  $\cN(S)$ with $C$-bounded scaling weights and such that there exists some $M > 0$ with
  ${\| \Realization_\varrho^{\Omega} (\Phi_n) \|_{L^1 (\mu)} \leq M}$
  for all $n \in \N$.

  Then, there is an infinite set $I \subset \N$ and a family of networks
  $(\Psi_n)_{n \in I} \subset \cN(S)$ with $C$-bounded scaling weights which satisfies
  $\Realization_\varrho^\Omega (\Phi_n) = \Realization_\varrho^\Omega (\Psi_n)$
  for $n \in I$ and such that $\|\Psi_n\|_{\mathrm{total}} \leq C'$
  for all $n \in I$ and a suitable constant $C' > 0$.
\end{lemma}

\begin{proof}
  Set $N_0 := d$.
  Since $\Omega$ is bounded, there is some $R > 0$ with
  $\|x\|_{\ell^\infty} \leq R$ for all $x \in \Omega$.
  In the following, we will use without further comment the estimate
  $\|A x\|_{\ell^\infty} \leq k \cdot \|A\|_{\max} \cdot \|x\|_{\ell^\infty}$
  which is valid for $A \in \R^{n \times k}$ and $x \in \R^k$.


  Below, we will show by induction on $m\in\left\{ 0,\dots,L-1\right\}$
  that for each $m\in\left\{ 0,\dots,L-1\right\} $, there is an infinite
  subset $I_{m} \subset \N$, and a family of networks
  $\big( \Psi^{(m)}_n \big)_{n \in I_m} \subset \cN(S)$ of the form
  \begin{equation}
     \Psi_n^{(m)}
    = \left(
        (B_1^{(n,m)}, c_1^{(n,m)}),
        \dots,
        (B_L^{(n,m)}, c_L^{(n,m)})
      \right)
    \label{eq:BoundedScalingClosednessNetworkNamingScheme}
  \end{equation}
  with the following properties:
  \begin{enumerate}
    \item[(A)] We have
               \(
                 \Realization_\varrho^\Omega \big( \Psi_n^{(m)} \big)
                 = \Realization_\varrho^{\Omega} (\Phi_n)
               \)
               for all $n \in I_m$;

    \item[(B)] each network $\Psi_n^{(m)}$, $n \in I_m$, has $C$-bounded scaling weights;

    \item[(C)] there is a constant $C_m > 0$ with
               $\big\| c_\ell^{(n,m)} \big\|_{\ell^\infty} \leq C_m$
               for all $n \in I_m$ and all $\ell \in \FirstN{m}$.
  \end{enumerate}
  Once this is shown, we set $I \coloneqq I_{L-1}$ and $\Psi_n \coloneqq \Psi_n^{(L-1)}$
  for $n \in I$.
  Clearly, $\Psi_n$ has $C$-bounded scaling weights and satisfies
  $\Realization_\varrho^\Omega (\Psi_n) = \Realization_\varrho^\Omega(\Phi_n)$,
  so that it remains to show $\| \Psi_n \|_{\mathrm{total}} \leq C'$, for which
  it suffices to show $\| c_L^{(n,L-1)} \|_{\ell^\infty} \leq C''$ for some $C'' > 0$ and all
  $n \in I$, since we have $\| c_\ell^{(n,L-1)} \|_{\ell^\infty} \leq C_{L-1}$ for all
  $\ell \in \FirstN{L-1}$.

  Now, note for $\ell \in \FirstN{L-1}$ and $x \in \R^{N_{\ell-1}}$ that
  $T_\ell^{(n,L-1)}(x) \coloneqq B_\ell^{(n,L-1)}x + c_\ell^{(n,L-1)}$ satisfies
  \[
    \big\| T_\ell^{(n,L-1)} (x) \big\|_{\ell^\infty}
    \leq N_{\ell-1} \cdot C \cdot \|x\|_{\ell^\infty} + C_{L-1} .
  \]
  Since $\Omega$ is bounded, and since $| \varrho(x) | \leq |x|$
  for all $x \in \R$, there is thus a constant $C_{L-1} ' > 0$ such that if we set
  \[
    \beta^{(n)} (x)
    \coloneqq \big(\varrho \circ T_{L-1}^{(n,L-1)}
        \circ \cdots \circ
        \varrho \circ T_1^{(n,L-1)} \, \big) (x)
    \quad \text{for} \quad x \in \Omega ,
  \]
  then $\|\beta^{(n)} (x)\|_{\ell^\infty} \leq C_{L-1} '$
  for all $x \in \Omega$ and all $n \in I$.

  For arbitrary $i \in \FirstN{N_L}$ and $x \in \Omega$, this implies
  \begin{align*}
      \big| [\Realization_\varrho^\Omega (\Phi_n) (x)]_i \big|
    = \big| [\Realization_\varrho^\Omega (\Psi_n^{(L-1)}) (x)]_i \big|
    & = \big|
          \big\langle \big(B_L^{(n,L-1)}\big)_i \,,\, \beta^{(n)}(x) \big\rangle
          + (c_L^{(n,L-1)})_{i}
        \big| \\
    & \geq \big| (c_L^{(n,L-1)})_{i} \big|
           - \big|
               \big\langle
                 \big(B_L^{(n,L-1)}\big)_i
                 \,,\,
                 \beta^{(n)}(x)
               \big\rangle
             \big| \\
    & \geq \big| (c_L^{(n,L-1)})_{i} \big|
           - N_{L-1} \cdot C \cdot \|\beta^{(n)}(x)\|_{\ell^\infty} \\
    & \geq \big| (c_L^{(n,L-1)})_{i} \big|
           - N_{L-1} \cdot C \cdot C_{L-1}' .
  \end{align*}
  Since by assumption $\|\Realization^\Omega_\varrho (\Phi_n) \|_{L^1(\mu)} \leq M$
  and $\mu(\Omega) > 0$, we see that $\big( c_L^{(n,L-1)} \big)_{n \in I}$ must be a bounded sequence.

  Thus, it remains to construct the networks $\Psi_n^{(m)}$ for $n \in I_m$
  (and the sets $I_m$) for $m \in \{0,\dots,L-1\}$ with the properties (A)--(C)
  from above.

  \medskip{}

  For the start of the induction ($m=0$), we can simply take
  $I_0 \coloneqq \N$, $\Psi_n^{(0)} \coloneqq \Phi_n$, and $C_0 > 0$ arbitrary,
  since condition (C) is void in this case.

  Now, assume that a family of networks $(\Psi_n^{(m)})_{n \in I_m}$
  as in Equation~\eqref{eq:BoundedScalingClosednessNetworkNamingScheme}
  with an infinite subset $I_{m} \subset \N$ and satisfying conditions (A)--(C)
  has been constructed for some $m \in \left\{ 0,\dots,L-2\right \}$.
  In particular, $L \geq 2$.

  For brevity, set
  $T_{\ell}^{(n)} : \R^{N_{\ell-1}} \to \R^{N_{\ell}}, ~x\mapsto B_{\ell}^{(n,m)}x + c^{(n,m)}_\ell$
  for $\ell \in \FirstN{L}$, and $\varrho_{L} \coloneqq \mathrm{id}_{\R^{N_{L}}}$,
  and let $\varrho_{\ell} \coloneqq \varrho \times \cdots \times \varrho$ denote the $N_\ell$-fold
  cartesian product of $\varrho$ for $\ell \in \FirstN{L-1}$.
  Furthermore, let us define
  $\beta_n \coloneqq \varrho_{m}\circ T_{m}^{(n)} \circ \cdots \circ
   \varrho_{1} \circ T_{1}^{(n)} : \R^{d} \to \R^{N_{m}}$.
  Note $\| \varrho_{\ell} (x) \|_{\ell^\infty} \leq \| x \|_{\ell^\infty}$
  for all $x\in\R^{N_{\ell}}$.
  Additionally, observe for $n \in I_{m}$, $\ell \in \FirstN{m}$ and
  $x \in \R^{N_{\ell-1}}$ that
  \[
    \left\| T_{\ell}^{(n)} (x) \right\|_{\ell^\infty}
    =    \left\| B_{\ell}^{(n,m)} x + c^{(n,m)}_{\ell} \right\|_{\ell^\infty}
    \leq N_{\ell-1} \cdot C \cdot \| x \|_{\ell^\infty} + C_{m} .
  \]
  Combining these observations, and recalling that $\Omega$ is bounded,
  we easily see that there is some $R' > 0$
  with $\| \beta_n (x) \|_{\ell^\infty} \leq R'$ for all $x \in \Omega$ and $n \in I_{m}$.

  Next, since $\left( c_{m+1}^{(n,m)} \right)_{n \in I_{m}}$
  is an infinite family in $\R^{N_{m+1}} \subset [-\infty, \infty]^{N_{m+1}}$,
  we can find (by compactness) an infinite subset $I_{m}^{(0)} \subset I_{m}$
  such that $c_{m+1}^{(n,m)}\to c_{m+1} \in [-\infty, \infty]^{N_{m+1}}$
  as $n \to \infty$ in the set $I_{m}^{(0)}$.

  \medskip{}

  Our goal is to construct vectors $d^{(n)},e^{(n)} \in \R^{N_{m+1}}$,
  matrices $C^{(n)} \in \R^{N_{m+1} \times N_{m}}$, and an
  infinite subset $I_{m+1} \subset I_{m}^{(0)}$ such that
  $\|C^{(n)}\|_{\max} \leq C$ for all $n \in I_{m+1}$, such that
  $\left(d^{(n)}\right)_{n\in I_{m+1}}$ is a bounded family,
  and such that we have
  \begin{equation}
      \varrho_{m+1} \big( T_{m+1}^{(n)}( x ) \big)
    = \varrho_{m+1} \big(C^{(n)} \, x + d^{(n)} \big) + e^{(n)}
    \quad \text{for all } x \in \R^{N_{m}} \text{ with } \| x \|_{\ell^\infty} \leq R' ,
    \label{eq:MainTarget}
  \end{equation}
  for all $n \in I_{m+1}$.

  Once $d^{(n)},e^{(n)},C^{(n)}$ are
  constructed, we can simply choose $\Psi_n^{(m+1)}$
  as in Equation~\eqref{eq:BoundedScalingClosednessNetworkNamingScheme},
  where we define $B_{\ell}^{(n,m+1)} \coloneqq B_{\ell}^{(n,m)}$
  and $c_{\ell}^{(n,m+1)} \coloneqq c_{\ell}^{(n,m)}$
  for $\ell \in \FirstN{L} \setminus \left \{ m+1,m+2 \right\}$, and finally
  \[
    B_{m+1}^{(n,m+1)} \coloneqq C^{(n)},
    \quad
    B_{m+2}^{(n,m+1)} \coloneqq B_{m+2}^{(n,m)},
    \quad
    c_{m+1}^{(n,m+1)} \coloneqq d^{(n)},
    \quad \text{and} \quad
    c_{m+2}^{(n,m+1)} \coloneqq c_{m+2}^{(n,m)} + B_{m+2}^{(n,m+1)} e^{(n)}
  \]
  for $n\in I_{m+1}$.
  Indeed, these choices clearly ensure $\big\|B_\ell^{(n,m+1)} \big\|_{\max} \leq C$
  for all $\ell \in \FirstN{L}$, as well as $\big\|c_\ell^{(n,m+1)}\big\|_{\ell^\infty} \leq C_{m+1}$
  for all $\ell \in \FirstN{m+1}$ and $n \in I_{m+1}$, for a suitable constant $C_{m+1} > 0$.

  Finally, since $\| \beta_{n}(x) \|_{\ell^\infty} \leq R'$
  for all $x \in \Omega$ and $n \in I_m$, Equation~\eqref{eq:MainTarget} implies
  \begin{align*}
    T_{m+2}^{(n)}
    \left(
      \varrho_{m+1} \left( T_{m+1}^{(n)} \left( \beta_n (x) \right) \right)
    \right)
    & = T_{m+2}^{(n)}
        \left(
          \varrho_{m+1} \left( C^{(n)} \beta_n (x) + d^{(n)} \right) + e^{(n)}
        \right)\\
    & = B_{m+2}^{(n,m)}
        \left(
          \varrho_{m+1}
          \left(
            B_{m+1}^{(n,m+1)} \beta_n (x) + c_{m+1}^{(n,m+1)}
          \right)
          + e^{(n)}
        \right)
        + c_{m+2}^{(n,m)}\\
    & = B_{m+2}^{(n,m+1)}
        \left(
          \varrho_{m+1}
          \left(
            B_{m+1}^{(n,m+1)} \beta_n (x) + c_{m+1}^{(n,m+1)}
          \right)
        \right)
        + c_{m+2}^{(n,m+1)}
  \end{align*}
  for all $x\in \Omega$ and $n\in I_{m+1}$. By recalling the definition
  of $\beta_n$, and by noting that $B_{\ell}^{(n,m+1)},c_{\ell}^{(n,m+1)}$
  are identical to $B_{\ell}^{(n,m)},c_{\ell}^{(n,m)}$
  for $\ell\in\FirstN{L}\setminus\left\{ m+1,m+2\right\} $, this
  easily yields
  \[
    \Realization_\varrho^\Omega \big( \Psi_n^{(m+1)} \big)
    = \Realization_\varrho^\Omega \big( \Psi_n^{(m)} \big)
    = \Realization_\varrho^\Omega (\Phi_n)
    \qquad \text{ for all } n \in I_{m+1} .
  \]

  \medskip{}

  Thus, it remains to construct $d^{(n)}, e^{(n)}, C^{(n)}$
  for $n \in I_{m+1}$ (and the set $I_{m+1}$ itself)
  as described around Equation~\eqref{eq:MainTarget}.
  To this end, for $n \in I_{m}^{(0)}$ and $k \in \FirstN{N_{m+1}}$, define
  \[
    d_{k}^{(n)}
    \coloneqq \begin{cases}
         R'\cdot N_m C,
         & \text{if } \left( c_{m+1} \right)_{k} = \infty,\\
         0,
         & \text{if } \left( c_{m+1} \right)_{k} = -\infty,\\
         \left( c_{m+1}^{(n,m)} \right)_{k},
         & \text{if } \left(c_{m+1}\right)_{k} \in \R,
       \end{cases}
    \qquad\text{and}\qquad
    e_{k}^{(n)}
    \coloneqq \begin{cases}
         \left( c_{m+1}^{(n,m)} \right)_{k} - R' \cdot N_m C,
         & \text{if } \left(c_{m+1}\right)_{k} = \infty,\\
         0,
         & \text{if } \left(c_{m+1}\right)_{k} = -\infty,\\
         0,
         & \text{if } \left(c_{m+1}\right)_{k} \in \R,
       \end{cases}
  \]
  as well as
  \[
  C_{k,-}^{(n)}
  \coloneqq \begin{cases}
       \left( B_{m+1}^{(n,m)} \right)_{k,-} \quad,
       & \text{if }\left( c_{m+1} \right)_{k} = \infty,\\
       0 \in \R^{N_{m}},
       & \text{if } \left( c_{m+1} \right)_{k} = -\infty,\\
       \left( B_{m+1}^{(n,m)} \right)_{k,-}\quad,
       & \text{if } \left( c_{m+1} \right)_{k} \in \R.
     \end{cases}
  \]
  To see that these choices indeed fulfil the conditions outlined around
  Equation~\eqref{eq:MainTarget} for a suitable choice of
  $I_{m+1}\subset I_{m}^{\left(0\right)}$,
  first note that $\left(d^{(n)}\right)_{n\in I_{m}^{\left(0\right)}}$
  is indeed a bounded family.
  Furthermore, $\big| C_{k,i}^{(n)} \big| \leq \big| (B_{m+1}^{(n,m)})_{k,i} \big|$ for all
  $k \in \FirstN{N_{m+1}}$ and $i \in \FirstN{N_m}$, which easily implies
  $\| C^{(n)} \|_{\max} \leq \| B_{m+1}^{(n,m)} \|_{\max} \leq C$ for all $n \in I_m^{(0)}$.
  Thus, it remains to verify equation~\eqref{eq:MainTarget} itself.
  But the estimate $\| B_{m+1}^{(n,m)} \|_{\max} \leq C$ also implies
  \begin{equation}
    \left| \left( B_{m+1}^{(n,m)} \, x \right)_k \right|
    \leq N_m \, C \cdot \| x \|_{\ell^\infty}
    \leq N_m \, C \cdot R'
    \quad \text{for all } \, k \in \underline{N_{m+1}} \text{ and all }
                      x \in \R^{N_{m}} \text{ with } \| x \|_{\ell^\infty} \leq R' .
    \label{eq:EasyEstimate}
  \end{equation}
  As a final preparation, note that $\varrho_{m+1} = \varrho \times \cdots \times \varrho$
  is a cartesian product of ReLU functions, since $m \leq L - 2$.
  Now, for $k \in \FirstN{N_{m+1}}$ there are three cases:

  \textbf{Case 1:} We have $(c_{m+1})_{k} = \infty$.
  Thus, there is some $n_{k} \in \N$ such that
  $\big( c_{m+1}^{(n,m)} \big)_{k} \geq R' \cdot N_m C$
  for all $n \in I_{m}^{(0)}$ with $n \geq n_{k}$.
  In view of Equation~\eqref{eq:EasyEstimate}, this implies
  \(
    \big( T_{m+1}^{(n)} (x) \big)_{k}
   = \big( B_{m+1}^{(n,m)} x + c_{m+1}^{(n,m)} \big)_{k} \geq 0
   ,
  \)
  and hence
  \[
    \Big[ \varrho_{m+1} \big(\, T_{m+1}^{(n)} (x) \, \big) \Big]_{k}
    = \Big( B_{m+1}^{(n,m)} \, x + c_{m+1}^{(n,m)} \Big)_{k}
    = \Big(
        \varrho_{m+1} \big( C^{(n)} \, x + d^{(n)} \big) + e^{(n)}
      \Big)_{k}\;,
  \]
  where the last step used our choice of $d^{(n)},e^{(n)},C^{(n)}$,
  and the fact that $\left(C^{(n)}x+d^{(n)}\right)_{k}\geq0$
  by Equation~\eqref{eq:EasyEstimate}.

  \medskip{}

  \textbf{Case 2:} We have $(c_{m+1})_{k} = -\infty$.
  This implies that there is some $n_{k} \in \N$ with
  $\big( c_{m+1}^{(n,m)} \big)_{k} \leq -R' \cdot N_m C$
  for all $n \in I_{m}^{(0)}$ with $n \geq n_{k}$.
  Because of Equation~\eqref{eq:EasyEstimate}, this yields
  \(
    \big( T_{m+1}^{(n)} (x) \big)_{k}
    = \big( B_{m+1}^{(n,m)} x + c_{m+1}^{(n,m)} \big)_{k}
    \leq 0
  \),
  and hence
  \[
    \Big[
      \varrho_{m+1} \big( T_{m+1}^{(n)} (x) \big)
    \Big]_{k}
    = 0
    = \Big[
        \varrho_{m+1} \big( C^{(n)} x + d^{(n)} \big) + e^{(n)}
      \Big]_{k}
    ,
  \]
  where the last step used our choice of $d^{(n)}, e^{(n)}, C^{(n)}$.

  \medskip{}

  \textbf{Case 3:} We have $(c_{m+1})_{k} \in \R$. In this
  case, set $n_{k} \coloneqq 1$, and note by our choice of $d^{(n)}, e^{(n)}, C^{(n)}$
  for $n \in I_{m}^{(0)}$ with $n \geq 1 = n_{k}$ that
  \[
    \Big[
      \varrho_{m+1} \big( C^{(n)}x + d^{(n)} \big) + e^{(n)}
    \Big]_{k}
    = \Big[
        \varrho_{m+1} \big( B_{m+1}^{(n,m)}x + c_{m+1}^{(n,m)} \big)
      \Big]_{k}
    = \Big[
        \varrho_{m+1} \big( T_{m+1}^{(n)} (x) \big)
      \Big]_{k} \;.
  \]

  Overall, we have thus shown that Equation~\eqref{eq:MainTarget} is
  satisfied for all $n \in I_{m+1}$, where
  \[
    I_{m+1}
    \coloneqq
    \big\{
      n \in I_{m}^{\left(0\right)}
      \,:\,
      n \geq \max \left\{ n_{k} \,:\, k \in \FirstN{N_{m+1}} \right\}
    \big\}
  \]
  is clearly an infinite set, since $I_{m}^{\left(0\right)}$ is.
\end{proof}

Using Lemma \ref{lem:WeightShiftLemma}, we can now easily show that the set
$\cRN_{\varrho}^{\Omega,C}(S)$ is closed in $L^p (\mu; \R^{N_L})$ and in $C(\Omega; \R^{N_L})$:
Let $\mathcal{Y}$ denote either $L^p (\mu;\R^{N_L})$ for some $p \in [1, \infty]$ and some finite
Borel measure $\mu$ on $\Omega$,
or $C(\Omega;\R^{N_L})$, where we assume in the latter case that $\Omega$ is compact
and set $\mu = \delta_{x_0}$ for a fixed $x_0 \in \Omega$.
Note that we can assume $\mu(\Omega) > 0$, since otherwise the claim is trivial.
Let $(f_n)_{n \in \N}$ be a sequence in $\cRN_{\varrho}^{\Omega,C}(S)\vphantom{\sum_{j_j}}$
which satisfies $f_n \to f$ for some $f \in \mathcal{Y}$, with convergence in $\mathcal{Y}$.
Thus, $f_n = \Realization_\varrho^\Omega (\Phi_n)$ for a
suitable sequence $(\Phi_n)_{n \in \N}$ in $\cN(S)$ with $C$-bounded scaling weights.

Since $(f_n)_{n \in \N} = \big( \Realization_\varrho^\Omega (\Phi_n) \big)_{n \in \N}$ is
convergent in $\mathcal{Y}$, it is also bounded in $\mathcal{Y}$.
But since $\Omega$ is bounded and $\mu$ is a finite measure, it is not hard to see
$\mathcal{Y} \hookrightarrow L^1 (\mu)$, so that we get
$\| \Realization_\varrho^\Omega (\Phi_n) \|_{L^1(\mu)} \leq M$ for all $n \in \N$ and
a suitable constant $M > 0$.

Therefore, Lemma~\ref{lem:WeightShiftLemma} yields an infinite set
$I \subset \N$ and networks $(\Psi_n)_{n \in I} \subset \cN(S)$ with
$C$-bounded scaling weights such that $f_n = \Realization_\varrho^\Omega (\Psi_n)$
and $\| \Psi_n \|_{\mathrm{total}} \leq C'$ for all $n \in I$ and a suitable $C' > 0$.

Hence, $(\Psi_n)_{n \in I}$ is a bounded, infinite family in the
finite dimensional vector space $\cN(S)$.
Thus, there is a further infinite set $I_1 \subset I$ such that
$\Psi_n \to \Psi \in \cN(S)$ as $n \to \infty$ in $I_1$.

But since $\Omega$ is bounded, say $\Omega \subset [-R,R]^d$, the realization map
\[
  \Realization^{[-R,R]^d}_\varrho
  : \cN(S) \to     C([-R,R]^d; \R^{N_L}),
    \Phi           ~    \mapsto \Realization^{[-R,R]^d}_\varrho (\Phi)
\]
is continuous (even locally Lipschitz continuous); see Proposition~\ref{prop:RealizationContinuity},
which will be proved independently.
Hence, $\Realization_\varrho^{[-R,R]^d} (\Psi_n) \to \Realization_\varrho^{[-R,R]^d}(\Psi)$
as $n \to \infty$ in $I_1$, with uniform convergence.
This easily implies
$f_n = \Realization_\varrho^{\Omega}(\Psi_n) \to \Realization_\varrho^{\Omega}(\Psi)$,
with convergence in $\mathcal{Y}$ as $n \to \infty$ in $I_1$.
Hence, $f = \Realization^\Omega_\varrho(\Psi) \in \cRN_{\varrho}^{C,\Omega}(S)$.
\hfill$\square$

\subsection{Proof of Theorem~\ref{thm:closedReLU}}
\label{app:closedReLU}

For the proof of Theorem~\ref{thm:closedReLU}, we will use a careful analysis
of the \textbf{singularity hyperplanes} of functions of the form
$x \mapsto \varrho_a (\langle \alpha, x \rangle + \beta)$, that is, the
hyperplane on which this function is not differentiable.
To simplify this analysis, we first introduce a convenient terminology
and collect quite a few auxiliary results.

\begin{definition}\label{def:SingularityHyperplanes}
  For $\alpha, \widetilde{\alpha} \in S^{d-1}$ and
  $\beta, \widetilde{\beta} \in \R$, we write
  $(\alpha,\beta) \sim (\widetilde{\alpha}, \widetilde{\beta})$
  if there is a $\eps \in \{\pm 1\}$ such that
  $(\alpha, \beta) = \eps \cdot (\widetilde{\alpha}, \widetilde{\beta})$.

  Furthermore, for $a \geq 0$ and with
  $\varrho_a : \R \to \R, x \mapsto \max \{x, ax \}$ denoting the
  parametric ReLU, we set
  \[
    S_{\alpha, \beta}
    \coloneqq \big\{ x \in \R^d \,:\, \langle \alpha, x \rangle + \beta = 0 \big\}
    \qquad \text{and} \qquad
    h_{\alpha,\beta}^{(a)} :
    \R^d \to \R,
    x \mapsto \varrho_a (\langle \alpha, x \rangle + \beta) .
  \]
  Moreover, we define
  \[
    W_{\alpha,\beta}^{+}
    \coloneqq \{ x \in \R^d \,:\, \langle \alpha, x \rangle + \beta > 0 \}
    \quad \text{and} \quad
    W_{\alpha,\beta}^{-}
    \coloneqq \{ x \in \R^d \,:\, \langle \alpha, x \rangle + \beta < 0 \} ,
  \]
  and finally
  \[
    U_{\alpha,\beta}^{(\eps)}
    \coloneqq \big\{
         x \in \R^d
         \,:\,
         |\langle \alpha, x \rangle + \beta| \geq \eps
       \big\} ,
    \quad
    U_{\alpha,\beta}^{(\eps,+)}
    \coloneqq U_{\alpha,\beta}^{(\eps)} \cap W_{\alpha,\beta}^{+}
    \quad \text{and} \quad
    U_{\alpha,\beta}^{(\eps,-)}
    \coloneqq U_{\alpha,\beta}^{(\eps)} \cap W_{\alpha,\beta}^{-}
    \quad \text{for } \eps > 0 .
  \]
\end{definition}

\begin{lemma}\label{lem:HyperplanesTechnicalLemma}
  Let $(\alpha, \beta) \in S^{d-1} \times \R$ and $x_0 \in S_{\alpha,\beta}$.
  Furthermore, let
  $(\alpha_1, \beta_1), \dots, (\alpha_N, \beta_N) \in S^{d-1} \times \R$ with
  $(\alpha_\ell, \beta_\ell) \not\sim (\alpha, \beta)$ for all
  $\ell \in \underline{N}$.
  Then, there exists $z \in \R^d$ satisfying
  \[
    \langle z, \alpha \rangle = 0
    \qquad \text{and} \qquad
    \langle z, \alpha_j \rangle \neq 0
    \quad \forall \, j \in \underline{N}
                  \text{ with } x_0 \in S_{\alpha_j, \beta_j} .
  \]
\end{lemma}

\begin{proof}
  By discarding those $(\alpha_j, \beta_j)$ for which
  $x_0 \notin S_{\alpha_j, \beta_j}$, we can assume that
  $x_0 \in S_{\alpha_j, \beta_j}$ for all $j \in \underline{N}$.

  Assume towards a contradiction that the claim of the lemma is false; that is,
  \begin{equation}
    \alpha^{\perp}
    = \bigcup_{j=1}^N
        \left\{ z \in \alpha^{\perp} \,:\, \langle z, \alpha_j \rangle = 0 \right\} ,
    \label{eq:HyperplanesTechnicalBaireInclusion}
  \end{equation}
  where $\alpha^\perp := \{z \in \R^d \,:\, \langle z,\alpha \rangle = 0\}$.
  Since $\alpha^\perp$ is a closed subset of $\R^d$ and thus a complete metric
  space, and since the right-hand side of
  \eqref{eq:HyperplanesTechnicalBaireInclusion} is a countable (in fact, finite)
  union of closed sets, the Baire category theorem
  (see \cite[Theorem~5.9]{FollandRA}) shows that there are
  $j \in \underline{N}$ and $\eps > 0$ such that
  \[
    V \coloneqq \left\{z \in \alpha^\perp \,:\, \langle z, \alpha_j \rangle = 0 \right\}
   \supset B_{\eps} (x) \cap \alpha^\perp
   \qquad \text{for some} \quad x \in V .
  \]
  But since $V$ is a vector space, this easily implies $V = \alpha^\perp$,
  that is, $\langle z, \alpha_j \rangle = 0$ for all $z \in \alpha^\perp$.
  In other words, $\alpha^\perp \subset \alpha_j^\perp$, and then
  $\alpha^\perp = \alpha_j^\perp$ by a dimension argument, since
  $\alpha, \alpha_j \neq 0$.

  Hence,
  $\spn \alpha = (\alpha^\perp)^\perp = (\alpha_j^\perp)^\perp = \spn \alpha_j$.
  Because of $|\alpha| = |\alpha_j| = 1$,
  we thus see $\alpha = \eps \, \alpha_j$
  for some $\eps \in \{\pm 1\}$.
  Finally, since $x_0 \in S_{\alpha,\beta} \cap S_{\alpha_j, \beta_j}$, we see
  \(
    \beta
    = - \langle \alpha, x_0 \rangle
    = - \eps \langle \alpha_j, x_0 \rangle
    = \eps \beta_j ,
  \)
  and thus $(\alpha, \beta) = \eps (\alpha_j, \beta_j)$, in contradiction
  to $(\alpha, \beta) \not\sim (\alpha_j, \beta_j)$.
\end{proof}

\begin{lemma}\label{lem:SingularityHyperplanesIntersection}
  Let $(\alpha,\beta) \in S^{d-1} \times \R$ and
  $(\alpha_1, \beta_1), \dots, (\alpha_N, \beta_N) \in S^{d-1} \times \R$
  with $(\alpha_i, \beta_i) \not\sim (\alpha, \beta)$ for all
  $i \in \underline{N}$.
  Furthermore, let $U \subset \R^d$ be open with
  $S_{\alpha, \beta} \cap U \neq \emptyset$.

  Then, there is $\eps > 0$ satisfying
  \[
    U
    \cap S_{\alpha, \beta}
    \cap \bigcap_{j=1}^N U_{\alpha_j, \beta_j}^{(\eps)}
    \neq \emptyset .
  \]
\end{lemma}

\begin{proof}
  By assumption, there exists $x_0 \in U \cap S_{\alpha, \beta}$.
  Next, Lemma~\ref{lem:HyperplanesTechnicalLemma} yields $z \in \R^d$
  such that $\langle z, \alpha \rangle = 0$ and
  $\langle z, \alpha_j \rangle \neq 0$ for all $j \in \underline{N}$ with
  $x_0 \in S_{\alpha_j, \beta_j}$.
  Note that this implies
  $\langle \alpha, x_0 + tz \rangle + \beta
   = \langle \alpha, x_0 \rangle + \beta = 0$
  and hence $x_0 + tz \in S_{\alpha, \beta}$ for all $t \in \R$.

  Next, let
  $J \coloneqq \{j \in \underline{N} \,:\, x_0 \notin S_{\alpha_j, \beta_j} \}$,
  so that $\langle \alpha_j, x_0 \rangle + \beta_j \neq 0$ for all $j \in J$.
  Thus, there are $\eps_1, \delta > 0$ with
  $|\langle \alpha_j, x_0 + t z\rangle + \beta_j| \geq \eps_1$
  (that is, $x_0 + t z \in U_{\alpha_j, \beta_j}^{(\eps_1)}$)
  for all $t \in \R$ with $|t| \leq \delta$ and all $j \in J$.
  Since $U$ is open with $x_0 \in U$, we can shrink $\delta$ so that
  $x_0 + t z \in U$ for all $|t| \leq \delta$. Let $t \coloneqq \delta$.

  We claim that there is some $\eps > 0$ such that
  $x \coloneqq x_0 + t z \in U \cap S_{\alpha, \beta}
                        \cap \bigcap_{j=1}^N U_{\alpha_j, \beta_j}^{(\eps)}$.
  To see this, note for $j \in \underline{N} \setminus J$ that
  $x_0 \in S_{\alpha_j, \beta_j}$, and hence
  \[
    |\langle x_0 + t z, \alpha_j \rangle + \beta_j|
    = |t| \cdot |\langle z, \alpha_j \rangle|
    \geq \delta \cdot \min_{\ell \in \underline{N} \setminus J}
                        |\langle z, \alpha_\ell \rangle |
    =: \eps_2 > 0 ,
  \]
  since $\langle z, \alpha_j \rangle \neq 0$ for all
  $j \in \underline{N} \setminus J$, by choice of $z$.
  By combining all our observations, we see that
  \({
    x_0 + t z
    \in U \cap S_{\alpha, \beta}
          \cap \bigcap_{j=1}^N U_{\alpha_j, \beta_j}^{(\eps)}
  }\)
  for $\eps \coloneqq \min \{ \eps_1, \eps_2 \} > 0$.
\end{proof}

\begin{lemma}\label{lem:ParametricReLUNotDifferentiable}
  If $0 \leq a < 1$ and $(\alpha, \beta) \in S^{d-1} \times \R$, then
  $h_{\alpha,\beta}^{(a)}$ is \emph{not} differentiable at any
  $x_0 \in S_{\alpha,\beta}$.
\end{lemma}

\begin{proof}
  Assume towards a contradiction that $h_{\alpha,\beta}^{(a)}$ is differentiable
  at some $x_0 \in S_{\alpha,\beta}$.
  Then, the function
  \(
    f : \R \to \R, t \mapsto h_{\alpha,\beta}^{(a)} (x_0 + t \alpha)
  \)
  is differentiable at $t = 0$.
  But since $x_0 \in S_{\alpha,\beta}$ and $\| \alpha \|_{\ell^2} = 1$, we have
  \[
    f(t)
    = \varrho_a \big( \langle \alpha, x_0 + t \alpha \rangle + \beta \big)
    = \varrho_a (t)
    = \begin{cases}
        t ,      & \text{if } t \geq 0 , \\
        a \, t , & \text{if } t < 0 ,
      \end{cases}
  \]
  for all $t \in \R$. This easily shows that $f$ is not differentiable at
  $t = 0$, since the right-sided derivative is $1$, while the left-sided
  derivative is $a \neq 1$.
  This is the desired contradiction.
\end{proof}

\begin{lemma}\label{lem:ParametricReLULinearIndependence}
  Let $0 \leq a < 1$, and let
  $(\alpha_1, \beta_1), \dots, (\alpha_N, \beta_N) \in S^{d-1} \times \R$
  with $(\alpha_i, \beta_i) \not\sim (\alpha_j, \beta_j)$ for $j \neq i$.
  Furthermore, let $U \subset \R^d$ be open with
  $U \cap S_{\alpha_i, \beta_i} \neq \emptyset$ for all $i \in \underline{N}$.
  Finally, set $h_i \coloneqq h_{\alpha_i, \beta_i}^{(a)}|_{U}$ for
  $i \in \underline{N}$ with $h_{\alpha_i, \beta_i}^{(a)}$ as in
  Definition~\ref{def:SingularityHyperplanes}, and let
  $h_{N+1} : U \to \R, x \mapsto 1$.

  \smallskip{}

  Then, the family $(h_i)_{i = 1,\dots,N+1}$ is linearly independent.
\end{lemma}

\begin{proof}
  Assume towards a contradiction that $0 = \sum_{i=1}^{N+1} \gamma_i \, h_i$
  for certain $\gamma_1, \dots, \gamma_{N+1} \in \R$ with $\gamma_\ell \neq 0$
  for some $\ell \in \underline{N+1}$.
  Note that if we had $\gamma_i = 0$ for all $i \in \underline{N}$, we would
  get $0 = \gamma_{N+1} \, h_{N+1} \equiv \gamma_{N+1}$, and thus
  $\gamma_i = 0$ for all $i \in \underline{N+1}$, a contradiction.
  Hence, there is some $j \in \underline{N}$ with $\gamma_j \neq 0$.

  By Lemma~\ref{lem:SingularityHyperplanesIntersection} there is some $\eps > 0$
  such that there exists
  \(
    x_0 \in U \cap S_{\alpha_j, \beta_j}
             \cap \bigcap_{i \in \underline{N} \setminus \{j\}}
                    U_{\alpha_i, \beta_i}^{(\eps)}
    .
  \)
  Therefore, $x_0 \in U \cap S_{\alpha_j, \beta_j} \cap V$ for the open set
  \(
    V \coloneqq \bigcap_{i \in \underline{N} \setminus \{j\}}
          \big( \R^d \setminus S_{\alpha_i, \beta_i} \big)
    .
  \)

  Because of $x_0 \in U \cap S_{\alpha_j, \beta_j}$,
  Lemma~\ref{lem:ParametricReLUNotDifferentiable} shows that
  $h_{\alpha_j, \beta_j}^{(a)} |_U$ is \emph{not} differentiable at $x_0$.
  On the other hand, we have
  \[
    h_{\alpha_j, \beta_j}^{(a)} |_{U}
    = h_j
    = - \gamma_j^{-1}
      \cdot \Big(
              \gamma_{N+1} \, h_{N+1}
              + \sum_{i \in \underline{N} \setminus \{j\}}
                  \gamma_i \,\, h_{\alpha_i, \beta_i}^{(a)}|_{U}
            \Big) ,
  \]
  where the right-hand side is differentiable at $x_0$, since
  each summand is easily seen to be differentiable on the open set $V$,
  with $x_0 \in V \cap U$.
\end{proof}

\begin{lemma}\label{lem:UEpsilonProperties}
  Let $(\alpha, \beta) \in S^{d-1} \times \R$.
  If $\Omega \subset \R^d$ is compact with $\Omega \cap S_{\alpha,\beta} = \emptyset$,
  then there is some $\eps > 0$ such that $\Omega \subset U_{\alpha,\beta}^{(\eps)}$.

\end{lemma}

\begin{proof}
  The continuous function
  $\Omega \to (0,\infty), x \mapsto |\langle \alpha, x \rangle + \beta|$, which is
  well-defined by assumption, attains a minimum
  $\eps = \min_{x \in \Omega} |\langle \alpha, x \rangle + \beta| > 0$.
  %
\end{proof}

\begin{lemma}\label{lem:PiecewiseLinearReLURepresentation}
  Let $0 \leq a < 1$, let $(\alpha, \beta) \in S^{d-1} \times \R$, and let
  $U \subset \R^d$ be open with $U \cap S_{\alpha, \beta} \neq \emptyset$.
  Finally, let $f : U \to \R$ be continuous, and assume that $f$ is
  affine-linear on $U \cap W_{\alpha,\beta}^{+}$
  and on $U \cap W_{\alpha,\beta}^{-}$.

  Then, there are $c, \kappa \in \R$ and $\zeta \in \R^d$ such that
  \[
    f(x) = c \cdot \varrho_a (\langle \alpha, x \rangle + \beta)
           + \langle \zeta, x \rangle + \kappa
    \qquad \text{ for all }\, x \in U .
  \]
\end{lemma}

\begin{proof}
  By assumption, there are $\xi_1, \xi_2 \in \R^d$ and
  $\omega_1, \omega_2 \in \R$ satisfying
  \[
    f(x) = \langle \xi_1, x \rangle + \omega_1
    \quad \text{for } x \in U \cap W_{\alpha,\beta}^{+}
    \qquad \text{and} \qquad
    f(x) = \langle \xi_2, x \rangle + \omega_2
    \quad \text{for } x \in U \cap W_{\alpha,\beta}^{-} .
  \]

  \noindent
  \textbf{Step~1:} We claim that
  $U \cap S_{\alpha,\beta} \subset \overline{U \cap W_{\alpha,\beta}^{\pm}}$.
  Indeed, for arbitrary $x \in U \cap S_{\alpha,\beta}$, we have
  $x + t \alpha \in U$ for $t \in (-\eps,\eps)$ for a suitable $\eps > 0$, since $U$ is open.
  But since $x \in S_{\alpha,\beta}$ and $\| \alpha \|_{\ell^2} = 1$, we have
  $\langle x + t \alpha, \alpha \rangle + \beta = t$.
  Hence, $x + t \alpha \in U \cap W_{\alpha,\beta}^{+}$ for $t \in (0,\eps)$
  and $x + t \alpha \in U \cap W_{\alpha,\beta}^{-}$ for $t \in (-\eps,0)$.
  This easily implies the claim of this step.

  \medskip{}

  \noindent
  \textbf{Step~2:} We claim that $\xi_1 - \xi_2 \in \spn \alpha$.
  To see this, consider the modified function
  \[
    \widetilde{f} : U \to \R,
    x \mapsto f(x) - (\langle \xi_2, x \rangle + \omega_2) ,
  \]
  which is continuous and satisfies $\widetilde{f} \equiv 0$ on
  $U \cap W_{\alpha,\beta}^{-}$ and
  $\widetilde{f}(x) = \langle \theta, x \rangle + \omega$ on
  $U \cap W_{\alpha,\beta}^{+}$, where we defined $\theta \coloneqq \xi_1 - \xi_2$
  and $\omega \coloneqq \omega_1 - \omega_2$.

  Since we saw in Step~1 that
  $U \cap S_{\alpha,\beta} \subset \overline{U \cap W_{\alpha,\beta}^{\pm}}$,
  we thus get by continuity of $\widetilde{f}$ that
  \[
    0 = \widetilde{f} (x) = \langle \theta, x \rangle + \omega
    \qquad \forall \, x \in U \cap S_{\alpha,\beta} .
  \]
  But by assumption on $U$, there is some $x_0 \in U \cap S_{\alpha,\beta}$.
  For arbitrary $v \in \alpha^{\perp}$, we then have
  $x_0 + t v \in U \cap S_{\alpha,\beta}$ for all $t \in (-\eps,\eps)$ and a
  suitable $\eps = \eps(v) > 0$, since $U$ is open.
  Hence, $0 = \langle \theta, x_0 + t v \rangle + \omega
            = t \cdot \langle \theta, v \rangle$ for all $t \in (-\eps,\eps)$,
  and thus $v \in \theta^{\perp}$.
  In other words, $\alpha^{\perp} \subset \theta^{\perp}$, and thus
  $
    \spn \alpha
    = (\alpha^{\perp})^\perp
    \supset (\theta^{\perp})^{\perp}
    \ni \theta = \xi_1 - \xi_2 ,
  $
  as claimed in this step.

  \medskip{}

  \noindent
  \textbf{Step~3:} In this step, we complete the proof.
  As seen in the previous step, there is some $c \in \R$ satisfying
  $c \alpha = (\xi_1 - \xi_2) / (1 - a)$.
  Now, set $\zeta \coloneqq (\xi_2 - a \xi_1) / (1 - a)$ and
  $\kappa \coloneqq f(x_0) - \langle \zeta, x_0 \rangle$, where
  $x_0 \in U \cap S_{\alpha,\beta}$ is arbitrary.
  Finally, define
  \[
    g : \R^d \to \R,
        x \mapsto c \cdot \varrho_a (\langle \alpha, x \rangle +  \beta)
                  + \langle \zeta, x \rangle + \kappa .
  \]
  Because of $x_0 \in S_{\alpha,\beta}$, we then have
  $g(x_0) = \langle \zeta, x_0 \rangle + \kappa = f(x_0)$.
  Furthermore, since $\varrho_a (x) = x$ for $x \geq 0$, we see for all
  $x \in U \cap W_{\alpha,\beta}^{+}$ that
  \begin{equation}
    \begin{split}
      g(x) - f(x_0)
      = g(x) - g(x_0)
      & = c \cdot (\langle \alpha, x \rangle + \beta)
          + \langle \zeta, x - x_0 \rangle \\
      ({\scriptstyle{\text{since } x_0 \in S_{\alpha,\beta},
                     \text{ i.e., } \langle \alpha, x_0 \rangle + \beta = 0}})
      & = c \cdot \langle \alpha, x - x_0 \rangle
          + \langle \zeta, x - x_0 \rangle
        = \Big\langle
            \frac{\xi_1 -  \xi_2}{1 - a} + \frac{\xi_2 - a \, \xi_1}{1 - a}
            \,,\,
            x - x_0
          \Big\rangle \\
      & = \langle \xi_1, x - x_0 \rangle
        = f(x) - f(x_0) .
    \end{split}
    \label{eq:PiecewiseLinearReLURepresentationPositivePartOK}
  \end{equation}
  Here, the last step used that $f(x) = \langle \xi_1, x \rangle + \omega_1$
  for $x \in U \cap W_{\alpha,\beta}^{+}$, and that
  $x_0\in U\cap S_{\alpha,\beta} \subset \overline{U \cap W_{\alpha,\beta}^{+}}$
  by Step~1, so that we get $f(x_0) = \langle \xi_1, x_0 \rangle + \omega_1$ as well.

  Likewise, since $\varrho_a (t) = a \, t$ for $t < 0$, we see
  for $x \in U \cap W_{\alpha,\beta}^{-}$ that
  \begin{equation}
    \begin{split}
      g(x) - f(x_0)
      = g(x) - g(x_0)
      & = a c \cdot (\langle \alpha, x \rangle + \beta)
          + \langle \zeta, x - x_0 \rangle \\
      ({\scriptstyle{\text{since } x_0 \in S_{\alpha,\beta},
                     \text{ i.e., } \langle \alpha, x_0 \rangle + \beta = 0}})
      & = ac \langle \alpha, x - x_0 \rangle + \langle \zeta, x - x_0 \rangle
        = \Big\langle
            a \frac{\xi_1 - \xi_2}{1 - a} + \frac{\xi_2 - a \xi_1}{1 - a}
            \,,\,
            x - x_0
          \Big\rangle \\
      & = \langle \xi_2, x - x_0 \rangle
        = f(x) - f(x_0) .
    \end{split}
    \label{eq:PiecewiseLinearReLURepresentationNegativePartOK}
  \end{equation}

  In combination,
  Equations~\eqref{eq:PiecewiseLinearReLURepresentationPositivePartOK}
  and \eqref{eq:PiecewiseLinearReLURepresentationNegativePartOK} show
  $f(x) = g(x)$ for all
  $x \in U \cap (W_{\alpha,\beta}^{+} \cup W_{\alpha,\beta}^{-})$.
  Since this set is dense in $U$ by Step 1,  we are done.
\end{proof}

With all of these preparations, we can finally prove Theorem~\ref{thm:closedReLU}.

\begin{proof}[Proof of Theorem~\ref{thm:closedReLU}]
Since $\varrho_1 = \mathrm{id}_\R$, the result is trivial for $a = 1$,
since $\cRN_{\varrho_1}^{[-B,B]^d}((d,N_0,1))$ is just the set of all
affine-linear maps $[-B,B]^d \to \R$.
Furthermore, if $a > 1$, then
${\varrho_a (x) = \max \{x, a x\} = a \, \varrho_{a^{-1}} (x)}$,
and hence
$\cRN_{\varrho_a}^{[-B,B]^d}((d,N_0,1)) = \cRN_{\varrho_{a^{-1}}}^{[-B,B]^d}((d,N_0,1))$.
Therefore, we can assume $a < 1$ in the sequel.
For brevity, let $\Omega\coloneqq[-B,B]^d$.
Then, each $\Phi \in \cN((d,N_0,1))$ is of the form
$\Phi = \big( (A_1, b_1), (A_2, b_2) \big)$ with $A_1 \in \R^{N_0 \times d}$,
$A_2 \in \R^{1 \times N_0}$, and $b_1 \in \R^{N_0}$, $b_2 \in \R^1$.

Let $(\Phi_n)_{n\in\N} \subset \cN((d,N_0,1))$ with
\(
  \Phi_n
  = \big(
      (\widetilde{A}_1^n, \widetilde{b}_1^n),
      (\widetilde{A}_2^n, \widetilde{b}_2^n)
    \big)
\)
be such that $f_n \coloneqq \Realization_{\varrho_a}^{\Omega}(\Phi_n)$ converges uniformly to some
$f \in C(\Omega)$.
Our goal is to prove $f \in \cRN_{\varrho_a}^{\Omega}((d,N_0,1))$.
The proof of this is divided into seven steps.

\medskip{}

\noindent
\textbf{Step 1 (Normalizing the rows of the first layer):}
Our first goal is to normalize the rows of the matrices $\widetilde{A}_1^n$;
that is, we want to change the parametrization of the network such that
$\| (\widetilde{A}_1^n)_{i,-} \|_{\ell^2} = 1$ for all $i \in \underline{N_0}$.
To see that this is possible, consider arbitrary
$A \in \R^{M_1 \times M_2} \neq 0$ and $b\in \R^{M_1}$;
then we obtain by the positive homogeneity of ${\varrho_a}$ for all $C >0$ that
\[
  {\varrho_a}(A x + b)
  = C \cdot \varrho_a \left(\frac{A}{C} \, x  + \frac{b}{C}\right)
  \text{ for all } x \in \R^{M_2}.
\]
This identity shows that for each $n \in \N$, we can find a network
\[
  \widetilde{\Phi}_n
  = \big(
      \left( A_1^n, b_1^n \right),
      \left( A_2^n, b_2^n \right)
    \big) \in \cN((d,N_0,1)) ,
  \]
such that the rows of $A_1^n$ are normalized, that is,
$\| (A_1^n)_{i,-} \|_{\ell^2} = 1$ for all $i \in \underline{N_0}$, and such that
\[
  \Realization_{\varrho_a}^{\Omega} \big( \widetilde{\Phi}_n \big)
  = \Realization_{\varrho_a}^{\Omega}(\Phi_n) = f_n
  \quad \text{ for all } n \in \N.
\]

\noindent
\textbf{Step 2 (Extracting a partially convergent subsequence):}
By the Theorem of Bolzano-Weierstra\ss{}, there is a common subsequence of
$(A_1^{n})_{n\in \N}$ and $(b_1^{n})_{n\in \N}$, denoted by
$(A^{n_k}_1)_{k \in \N}$ and $(b_1^{n_k})_{k \in \N}$, converging to
$A_1 \in \R^{N_0 \times d}$ and $b_1 \in [-\infty, \infty]^{N_0}$, respectively.

For $j \in \underline{N_0}$, let $a_{k,j} \in \R^{d}$ denote the $j$-th row of
$A_1^{n_k}$, and let $a_{j} \in \R^d$ denote the $j$-th row of $A_1$.
Note that $\| a_{k,j} \|_{\ell^2} = \| a_j \|_{\ell^2} = 1$
for all $j \in \underline{N_0}$ and $k \in\N$.
Next, let
\[
  J
  \coloneqq \left\{
       j \in \underline{N_0}
       \,:\,
       (b_1)_j \in \{\pm \infty\}
       \text{ or }
       \big[
         (b_1)_j \in \R \text{ and }
         S_{a_j, (b_1)_j} \cap \Omega^\circ = \emptyset
       \big]
     \right\} ,
\]
where $\Omega^\circ = (-B, B)^d$ denotes the interior of $\Omega$.
Additionally, let $J^c \coloneqq \underline{N_0} \setminus J$, and for $j,\ell \in J^c$
write $j \simeq \ell$ iff $(a_j, (b_1)_j) \sim (a_\ell, (b_1)_\ell)$,
with the relation $\sim$ introduced in Definition~\ref{def:SingularityHyperplanes}.
Note that this makes sense, since $(b_1)_j \in \R$ if $j \in J^c$.
Clearly, the relation $\simeq$ is an equivalence relation on $J^c$.
Let $(J_i)_{i = 1,\dots,r}$ denote the equivalence classes of the relation
$\simeq$.
For each $i \in \underline{r}$, choose $\alpha^{(i)} \in S^{d-1}$ and
$\beta^{(i)} \in \R$ such that for each $j \in J_i$ there is a (unique) $\sigma_j \in \{\pm 1\}$
with $(a_j, (b_1)_j) = \sigma_j \cdot (\alpha^{(i)}, \beta^{(i)})$.

\medskip{}

\noindent
\textbf{Step 3 (Handling the case of distinct singularity hyperplanes):}
Note that $r \leq |J^c| \leq N_0$.
Before we continue with the general case, let us consider the special case
where equality occurs, that is, where $r = N_0$.
This means that $J = \emptyset$ (and hence $(b_1)_j \in \R$
and $\Omega^\circ \cap S_{a_j, (b_1)_j} \neq \emptyset$ for all $j \in \underline{N_0}$),
and that each equivalence class $J_i$ has precisely one element;
that is, $(a_j, (b_1)_j) \not\sim (a_\ell, (b_1)_\ell)$
for $j, \ell \in \underline{N_0}$ with $j \neq \ell$.

Therefore, Lemma~\ref{lem:ParametricReLULinearIndependence} shows that the
functions $(h_j |_{\Omega^\circ})_{j = 1,\dots,N_0 + 1}$ with
$h_j \coloneqq h_{a_j, (b_1)_j}^{(a)}|_{\Omega}$ for $j \in \underline{N_0}$ and
$h_{N_0+1} : \Omega \to \R, x \mapsto 1$ are linearly independent.
In particular, these functions are linearly independent when considered on
all of $\Omega$.
Thus, we can define a norm $\|\cdot\|_{\ast}$ on $\R^{N_0 + 1}$ by virtue of
\[
  \|c\|_{\ast}
  \coloneqq \Big\|
       c_{N_0 + 1} + \sum_{j=1}^{N_0} c_j \, h_{a_j, (b_1)_j}^{(a)}
     \Big\|_{L^\infty (\Omega)}
  \qquad \text{for } c = (c_j)_{j = 1,\dots,N_0 + 1} \in \R^{N_0 + 1} .
\]
Since all norms on the finite dimensional vector space $\R^{N_0 + 1}$ are
equivalent, there is some $\tau > 0$ with
$\|c\|_{\ast} \geq \tau \cdot \|c\|_{\ell^1}$ for all $c \in \R^{N_0 + 1}$.

Now, recall that $a_{k,j} \to a_j$ and $b_1^{n_k} \to b_1 \in \R^{N_0}$ as $k \to \infty$.
Since $\Omega$ is bounded, this implies for arbitrary $j \in \underline{N_0}$ and
$h_j^{(k)} \coloneqq h_{a_{k,j}, (b_1^{n_k})_j}^{(a)}$ that
$h_j^{(k)} \to h_{a_j, (b_1)_j}^{(a)}$ as $k \to \infty$, with uniform convergence on $\Omega$.
Thus, there is some $N_1 \in \N$ such that
$\big\| h_j^{(k)} - h_{a_j, (b_1)_j}^{(a)} \big\|_{L^\infty (\Omega)} \leq \tau / 2$
for all $k \geq N_1$ and $j \in \underline{N_0}$.
Therefore, if $k \geq N_1$, we have
\[
  \begin{split}
    \bigg\|
      c_{N_0 + 1} + \sum_{j=1}^{N_0} c_j \, h_j^{(k)}
    \bigg\|_{L^\infty (\Omega)}
    & \geq \bigg\|
            c_{N_0 + 1} + \sum_{j=1}^{N_0} c_j \, h_{a_j, (b_1)_j}^{(a)}
          \bigg\|_{L^\infty (\Omega)}
          - \bigg\|
              \sum_{j=1}^{N_0}
                c_j \, \big( h_{a_j, (b_1)_j}^{(a)} - h_j^{(k)} \big)
            \bigg\|_{L^\infty (\Omega)} \\
    & \geq \tau \cdot \|c\|_{\ell^1}
           - \sum_{j=1}^{N_0}
               |c_j|
               \cdot \big\|
                       h_{a_j, (b_1)_j}^{(a)} - h_j^{(k)}
                     \big\|_{L^\infty (\Omega)} \\
    & \geq \big(\, \tau - \frac{\tau}{2} \,\big) \cdot \|c\|_{\ell^1}
      =    \frac{\tau}{2} \cdot \|c\|_{\ell^1}
      \qquad \text{ for all } c = (c_j)_{j = 1,\dots,N_0 + 1} \in \R^{N_0 + 1} .
  \end{split}
\]
Since
\(
  f_{n_k}
  = \Realization_{\varrho_a}^\Omega \big( \widetilde{\Phi}_{n_k} \big)
  = b_2^{n_k} + \sum_{j=1}^{N_0} \big(A_2^{n_k}\big)_{1,j} \,\, h_j^{(k)}
  \vphantom{\sum_j}
\)
converges uniformly on $\Omega$, we thus see that the sequence consisting of
$(A_2^{n_k}, b_2^{n_k}) \in \R^{1 \times N_0} \times \R \cong \R^{N_0 + 1}$
is bounded.
Thus, there is a further subsequence $(n_{k_{\ell}})_{\ell \in \N}$ such that
$A_2^{n_{k_\ell}} \to A_2 \in \R^{1 \times N_0}$ and
$b_2^{n_{k_\ell}} \to b_2 \in \R$ as $\ell \to \infty$.
But this implies as desired that
\begin{align*}
  f
  = \lim_{\ell \to \infty} f_{n_{k_\ell}}
  & = \lim_{\ell \to \infty}
         \bigg[
           b_2^{n_{k_\ell}}
           + \sum_{j=1}^{N_0}
               \big( A_2^{n_{k_\ell}} \big)_{1,j} \,\, h_j^{(k_\ell)}\big|_{\Omega}
         \bigg] \\
  & = b_2 + \sum_{j=1}^{N_0} (A_2)_{1,j} \,\, h_{a_j, (b_1)_j}^{(a)}|_\Omega
    \in \cRN_{ \varrho_a}^{\Omega}((d,N_0,1)).
\end{align*}

\medskip{}

\noindent
\textbf{Step 4 (Showing that the $j$-th neuron is eventually affine-linear, for $j \in J$):}
Since Step~3 shows that the claim holds in case of $r = N_0$, we will from now on
consider only the case where $r < N_0$.

For $j \in J$, there are two cases: In case of $(b_1)_j \in [0,\infty]$, define
\[
  \phi_j^{(k)} : \R^d \to \R,
                 x \mapsto (A_2^{n_k})_{1,j}
                           \cdot \big[
                                  \langle a_{k,j} , x \rangle + (b_1^{n_k})_j
                                 \big]
  \qquad \text{for all } k \in \N .
\]
If otherwise $(b_1)_j \in [-\infty,0)$, define
\[
  \phi_j^{(k)} : \R^d \to \R,
                 x \mapsto a \cdot (A_2^{n_k})_{1,j}
                           \cdot \big[
                                  \langle a_{k,j} , x \rangle + (b_1^{n_k})_j
                                 \big]
  \qquad \text{for all } k \in \N .
\]

Next, for arbitrary $0 < \delta < B$, we define
$\Omega_{\delta} \coloneqq [-(B - \delta), B - \delta]^d$.
Note that since $S_{\alpha^{(i)}, \beta^{(i)}} \cap \Omega^\circ \neq \emptyset$
for all $i \in \underline{r}$, there is some $\delta_0 > 0$ such that
$S_{\alpha^{(i)}, \beta^{(i)}} \cap ( -(B - \delta), B - \delta )^d
 \neq \emptyset$
for all $i \in \underline{r}$ and all $0 < \delta \leq \delta_0$.
For the remainder of this step, we will consider a fixed
$\delta \in (0, \delta_0]$, and we claim that there is some
$N_2 = N_2 (\delta) > 0$ such that
\begin{equation}
  (b_1)_j \neq 0
  \quad \text{and} \quad
  \sgn \big( \langle a_{k,j}, x \rangle + (b_1^{n_k})_j \big)
  = \sgn \big( (b_1^{n_k})_j \big) \neq 0
  \quad \text{ for all } j \in J, \, k \geq N_2 , \, \text{ and } x \in \Omega_{\delta} ,
  \label{eq:ReLUClosednessLinearPartsSignum}
\end{equation}
where $\sgn x = 1$ if $x > 0$, $\sgn x = -1$ if $x < 0$, and $\sgn 0 = 0$.
Note that once this is shown, it is not hard to see that there is some
$N_3 = N_3 (\delta) \in \N$ such that
\[
  (A_2^{n_k})_{1,j} \cdot \varrho_a
                          \big( \langle a_{k,j} , x \rangle + (b_1^{n_k})_j \big)
  = \phi_j^{(k)} (x)
  \quad \text{ for all }\, j \in J, \, k \geq N_3 , \, \text{ and } x \in \Omega_{\delta}
  ,
\]
simply because $(b_1^{n_k})_j \to (b_1)_j$ and $\varrho_a (x) = x$ if $x \geq 0$,
and $\varrho_a (x) = ax$ if $x < 0$.
Therefore, the affine-linear function
\begin{equation}
  \begin{split}
    & g_{r+1}^{(k)} \coloneqq b_2^{n_k} + \sum_{j \in J} \phi_j^{(k)} : \R^d \to \R \\
    \text{satisfies} \quad
    & g_{r+1}^{(k)} (x)
      = b_2^{n_k}
        + \sum_{j \in J}
            (A_2^{n_k})_{1,j} \,
            \varrho_a \big( \langle a_{k,j}, x \rangle + (b_1^{n_k} )_j \big)
      \quad \text{ for all } k \geq N_3 (\delta) \text{ and } x \in \Omega_{\delta} .
  \end{split}
  \label{eq:ReLUClosednessLinearPartsOnKDelta}
\end{equation}

To prove Equation~\eqref{eq:ReLUClosednessLinearPartsSignum}, we distinguish
two cases for each $j \in J$; by definition of $J$, these are the only two possible cases:

\smallskip{}

\noindent
\textbf{Case 1:} We have $(b_1)_j \in \{ \pm \infty \}$.
In this case, the first part of Equation~\eqref{eq:ReLUClosednessLinearPartsSignum}
is trivially satisfied.
To prove the second part, note that because of $(b_1^{n_k})_j \to (b_1)_j \in \{-\infty,\infty\}$,
there is some $k_j \in \N$ with $|(b_1^{n_k})_j| \geq 2d \cdot B$ for all $k \geq k_j$.
Since we have $\| a_{k,j} \|_{\ell^2} = 1$ and $\| x \|_{\ell^2} \leq \sqrt{d} B \leq d B$
for $x \in \Omega$, this implies
\[
  | \langle a_{k,j}, x \rangle + (b_1^{n_k})_j|
  \geq |(b_1^{n_k})_j| - | \langle a_{k_j}, x \rangle |
  \geq 2d \cdot B - \| x \|_{\ell^2}
  \geq dB > 0
  \qquad \forall \, x \in \Omega = [-B, B]^d \,\text{ and }\, k \geq k_j .
\]
Now, since the function $x \mapsto \langle a_{k,j}, x \rangle + (b_1^{n_k})_j$
is continuous, since $\Omega$ is connected (in fact convex), and since $0 \in \Omega$,
this implies
$\sgn (\langle a_{k,j}, x \rangle + (b_1^{n_k})_j) = \sgn (b_1^{n_k})_j$
for all $x \in \Omega$ and $k \geq k_j$.

\smallskip{}

\noindent
\textbf{Case 2:} We have $(b_1)_j \in \R$, but
$S_{a_j, (b_1)_j} \cap \Omega^\circ = \emptyset$, and hence
$S_{a_j, (b_1)_j} \cap \Omega_{\delta} = \emptyset$.
In view of Lemma~\ref{lem:UEpsilonProperties}, there is thus some
$\eps_{j,\delta} > 0$ satisfying $\Omega_{\delta} \subset U_{a_j, (b_1)_j}^{(\eps_{j,\delta})}$;
that is, $|\langle a_j, x \rangle + (b_1)_j| \geq \eps_{j,\delta} > 0$
for all $x \in \Omega_{\delta}$.
In particular, since $0 \in \Omega_{\delta}$, this implies $|(b_1)_j| \geq \eps_{j,\delta} > 0$
and hence $(b_1)_j \neq 0$, as claimed in the first part
of Equation~\eqref{eq:ReLUClosednessLinearPartsSignum}.

To prove the second part, note that because of $a_{k,j} \to a_j$
and $(b_1^{n_k})_j \to (b_1)_j$ as $k \to \infty$,
there is some $k_j = k_j (\eps_{j,\delta}) = k_j (\delta) \in \N$ such that
$\| a_{k,j} - a_j \|_{\ell^2} \leq \eps_{j,\delta} / (4d B)$ and
$|(b_1^{n_k})_j - (b_1)_j| \leq \eps_{j,\delta} / 4$ for all $k \geq k_j$.
Therefore,
\begin{align*}
  | \langle a_{k,j}, x \rangle + (b_1^{n_k})_j|
  & \geq |\langle a_j , x \rangle + (b_1)_j|
         - |\langle a_j - a_{k,j} , x \rangle + (b_1)_j - (b_1^{n_k})_j| \\
  & \geq \eps_{j,\delta}
         - \| a_j - a_{k,j} \|_{\ell^2} \cdot \| x \|_{\ell^2}
         - |(b_1)_j - (b_1^{n_k})_j| \\
  & \geq \eps_{j,\delta}
         - \frac{\eps_{j,\delta}}{4 d B} \cdot dB
         - \frac{\eps_{j,\delta}}{4}
    =    \frac{\eps_{j,\delta}}{2} > 0
    \quad \text{ for all }\, x \in \Omega_{\delta} \text{ and } k \geq k_j .
\end{align*}
With the same argument as at the end of Case~1, we thus
see $\sgn (\langle a_{k,j}, x \rangle + (b_1^{n_k})_j) = \sgn (b_1^{n_k})_j$
for all $x \in \Omega_{\delta}$ and $k \geq k_j (\delta)$.

\medskip{}

Together, the two cases prove that
Equation~\eqref{eq:ReLUClosednessLinearPartsSignum} holds if we set
$N_2 (\delta) \coloneqq \max_{j \in J} k_j (\delta)$.

\medskip{}

\noindent
\textbf{Step 5 (Showing that the $j$-th neuron is affine-linear on
                $U_{\alpha^{(i)}, \beta^{(i)}}^{(\eps,+)}$
                and on $U_{\alpha^{(i)}, \beta^{(i)}}^{(\eps,-)}$ for $j \in J_i$):}
In the following, we write $U_{\alpha^{(i)}, \beta^{(i)}}^{(\eps, \pm)}$
for one of the two sets $U_{\alpha^{(i)}, \beta^{(i)}}^{(\eps,+)}$
or $U_{\alpha^{(i)}, \beta^{(i)}}^{(\eps,-)}$.
We claim that for each $\eps > 0$, there is some $N_4 (\eps) \in \N$ such that:
\[
  \text{If } i \in \underline{r} , \,\,\,
             j \in J_i \text{ and }
             k \geq N_4 (\eps) ,
  \text{ then }
  \nu_j^{(k)}
  \coloneqq \varrho_a \big(\langle a_{k,j}, \cdot \rangle + (b_1^{n_k})_j\big)
  \text{ is affine-linear on } \Omega \cap U_{\alpha^{(i)}, \beta^{(i)}}^{(\eps,\pm)}
  .
\]
To see this, let $\eps > 0$ be arbitrary, and recall
$J^c = \bigcup_{i=1}^r J_i$.
By definition of $J_i$, and by choice of $\alpha^{(i)}$ and $\beta^{(i)}$,
there is for each $i \in \underline{r}$ and $j \in J_i$ some $\sigma_j \in \{\pm 1\}$ satisfying
\[
  \big(a_{k,j}, (b_1^{n_k})_j \big)
  \xrightarrow[k\to\infty]{} \big( a_j , (b_1)_j \big)
  = \sigma_j \cdot \big(\alpha^{(i)}, \beta^{(i)} \big) .
\]
Thus, there is some $k^{(j)}(\eps) \in \N$ such that
$\| a_{k,j} - \sigma_j \, \alpha^{(i)} \|_{\ell^2} \leq \eps / (4 d B)$ and
$|(b_1^{n_k})_j - \sigma_j \, \beta^{(i)}| \leq \eps / 4$ for all
$k \geq k^{(j)}(\eps)$.

Define $N_4 (\eps) \coloneqq \max_{j \in J^c} k^{(j)}(\eps)$.
Then, for $k \geq N_4 (\eps)$, $i \in \underline{r}$, $j \in J_i$,
and arbitrary $x \in \Omega \cap U_{\alpha^{(i)},\beta^{(i)}}^{(\eps, \pm)}$,
we have on the one hand
$|\sigma_j \cdot (\langle \alpha^{(i)}, x \rangle + \beta^{(i)})| \geq \eps$,
and on the other hand
\[
  \big|
    \big(
      \langle a_{k,j}, x \rangle
      + (b_1^{n_k})_j
    \big)
    - \sigma_j \cdot \big( \langle \alpha^{(i)}, x \rangle + \beta^{(i)} \big)
  \big|
  \leq dB \cdot \| a_{k,j} - \sigma_j \, \alpha^{(i)} \|_{\ell^2}
       + |(b_1^{n_k})_j - \sigma_j \, \beta^{(i)}|
  \leq \eps / 2 ,
\]
since $\| x \|_{\ell^2} \leq \sqrt{d} \cdot B \leq d B$.
In combination, this shows
$|\langle a_{k,j}, x \rangle + (b_1^{n_k})_j| \geq \eps/2 > 0$ for all
$x \in \Omega \cap U_{\alpha^{(i)}, \beta^{(i)}}^{(\eps, \pm)}$.
But since $\Omega \cap U_{\alpha^{(i)}, \beta^{(i)}}^{(\eps, \pm)}$ is connected
(in fact, convex), and since the function
$x \mapsto \langle a_{k,j} , \, x \rangle + (b_1^{n_k})_j$ is continuous,
it must have a constant sign on
$\Omega \cap U_{\alpha^{(i)}, \beta^{(i)}}^{(\eps, \pm)}$.
This easily implies that $\nu_j^{(k)}
= \varrho_a \big(\langle a_{k,j}, \cdot \rangle + (b_1^{n_k})_j\big)$
is indeed affine-linear on $\Omega \cap U_{\alpha^{(i)}, \beta^{(i)}}^{(\eps, \pm)}$
for $k \geq N_4 (\eps)$.

\medskip{}

\noindent
\textbf{Step 6 (Proving the ``almost convergence'' of the sum
                of all $j$-th neurons for $j \in J_i$):}
For $i \in \underline{r}$ define
\[
  g_i^{(k)} :
  \R^d \to \R,
  x \mapsto \smash{\sum_{j \in J_i}}
              (A_2^{n_k})_{1,j} \,\,
              \varrho_a \big( \langle a_{k,j} , x \rangle + (b_1^{n_k})_j \big)
          = \smash{\sum_{j \in J_i}}
              (A_2^{n_k})_{1,j} \,\, \nu_j^{(k)}(x)
          \vphantom{\sum_i}
  .
\]
In combination with Equation~\eqref{eq:ReLUClosednessLinearPartsOnKDelta}, we see
\begin{equation}
  f_{n_k} (x)
  = \Realization_{\varrho_a}^\Omega (\widetilde{\Phi}_{n_k})(x)
  = \vphantom{\sum}\smash{\sum_{\ell=1}^{r+1}} \,\, g_{\ell}^{(k)} (x)
  \qquad \forall \, x \in \Omega_{\delta} \text{ and } k \geq N_3 (\delta) ,
  \label{eq:ReLUClosednessRealizationDecomposition}
\end{equation}
with $g_{r+1}^{(k)} : \R^d \to \R$ being affine-linear.

Recall from Step 4 that
$\Omega_{\delta_0}^\circ \cap S_{\alpha^{(i)}, \beta^{(i)}} \neq \emptyset$ for all
$i \in \underline{r}$, by choice of $\delta_0$.
Therefore, Lemma~\ref{lem:SingularityHyperplanesIntersection} shows
(because of
 $U_{\alpha,\beta}^{(\sigma)} \subset (U_{\alpha,\beta}^{(\eps)})^\circ$
 for $\eps < \sigma$)
for each
$i \in \underline{r}$ that
\[
  K_i \coloneqq \Omega_{\delta_0}^\circ
         \cap S_{\alpha^{(i)}, \beta^{(i)}}
         \cap \bigcap_{\ell \in \underline{r} \setminus \{i\}}
                \big( U_{\alpha^{(\ell)}, \beta^{(\ell)}}^{(\eps_i)} \big)^\circ
      \neq \emptyset
  \qquad \text{for a suitable} \quad \eps_i > 0 .
\]
Let us fix some $x_i \in K_i$ and some $r_i > 0$ such that
$\overline{B}_{r_i} (x_i) \subset
 \Omega_{\delta_0}^\circ
 \cap \bigcap_{\ell \in \underline{r} \setminus \{i\}}
        \big( U_{\alpha^{(\ell)}, \beta^{(\ell)}}^{(\eps_i)} \big)^\circ$;
this is possible, since the set on the right-hand side contains $x_i$ and is open.
Now, since $\overline{B}_{r_i}(x_i)$ is connected, we see for each
$\ell \in \underline{r} \setminus \{i\}$ that either
$\overline{B}_{r_i}(x_i)\subset U_{\alpha^{(\ell)},\beta^{(\ell)}}^{(\eps_i,+)}$ or
$\overline{B}_{r_i}(x_i)\subset U_{\alpha^{(\ell)},\beta^{(\ell)}}^{(\eps_i,-)}$.
Therefore, as a consequence of the preceding step, we see that there is some
$N_5^{(i)} \in \N$ such that $g_\ell^{(k)}$ is affine-linear on
$\overline{B}_{r_i} (x_i)$ for all $\ell \in \underline{r} \setminus \{i\}$
and all $k \geq N_5^{(i)}$.

Thus, setting
$N_5 \coloneqq \max \{ N_3 (\delta_0), \max_{i = 1,\dots,r} N_5^{(i)} \}$, we see
as a consequence of Equation~\eqref{eq:ReLUClosednessRealizationDecomposition}
and because of $\overline{B}_{r_i} (x_i) \subset \Omega_{\delta_0}^\circ$
that for each $i \in \underline{r}$ and any $k \geq N_5$, there is an
affine-linear map $q_i^{(k)} : \R^d \to \R$ satisfying
\begin{equation}
  f_{n_k}(x)
  = \sum_{\ell=1}^{k+1} \, g_\ell^{(k)}(x)
  = g_i^{(k)}(x) + q_i^{(k)}(x)
  \quad \text{ for all } x \in \overline{B}_{r_i}(x_i) \text{ and } k \geq N_5 .
  \label{eq:ReLUClosednessNetworkRewrittenOnSmallBall}
\end{equation}

Next, note that Step~5 implies for arbitrary $\eps > 0$ that
for all $k$ large enough (depending on $\eps$), $g_i^{(k)}$ is affine-linear on
$B_{r_i} (x_i) \cap U_{\alpha^{(i)}, \beta^{(i)}}^{(\eps, \pm)}$.
Since $f(x) = \lim_k f_{n_k}(x) = \lim_{k} g_i^{(k)}(x) + q_i^{(k)}(x)$,
we thus see that $f$ is affine-linear on
$B_{r_i} (x_i) \cap U_{\alpha^{(i)}, \beta^{(i)}}^{(\eps, \pm)}$ for
\emph{arbitrary} $\eps > 0$.
Therefore, $f$ is affine-linear on
$B_{r_i} (x_i) \cap W_{\alpha^{(i)}, \beta^{(i)}}^{\pm}$
and continuous on $\Omega \supset B_{r_i}(x_i)$, and we have
$x_i \in B_{r_i} (x_i) \cap S_{\alpha^{(i)}, \beta^{(i)}} \neq \emptyset$.
Thus, Lemma~\ref{lem:PiecewiseLinearReLURepresentation} shows that there
are $c_i \in \R$, $\zeta_i \in \R^d$, and $\kappa_i \in \R$ such that
\begin{equation}
  f(x) = G_i (x)
  \quad \forall \, x \in B_{r_i} (x_i) ,
  \qquad \text{with} \quad
  G_i : \R^d \to \R,
        x \mapsto c_i \cdot \varrho_a
                            \big(
                              \langle \alpha^{(i)}, x \rangle + \beta^{(i)}
                            \big)
                  + \langle \zeta_i , x \rangle + \kappa_i .
  \label{eq:ReLUClosednessNetworkLimitRewrittenOnSmallBall}
\end{equation}

\smallskip{}

We now intend to make use of the following elementary fact:
If $(\psi_k)_{k \in \N}$ is a sequence of maps $\psi_k : \R^d \to \R$,
if $\Theta \subset \R^d$ is such that each $\psi_k$ is affine-linear on
$\Theta$, and if $U \subset \Theta$ is a nonempty \emph{open} subset such that
$\psi (x) \coloneqq \lim_{k \to \infty} \psi_k (x) \in \R$ exists for all $x \in U$,
then $\psi$ can be uniquely extended to an affine-linear map $\psi : \R^d \to \R$,
and we have $\psi_k (x) \to \psi(x)$ for all $x \in \Theta$, even with locally
uniform convergence.
Essentially, what is used here is that the vector space of affine-linear maps
$\R^d \to \R$ is finite-dimensional, so that the (Hausdorff) topology of
pointwise convergence on $U$ coincides with that of locally uniform convergence
on $\Theta$; see \cite[Theorem~1.21]{RudinFunkana}.

To use this observation, note that
Equations~\eqref{eq:ReLUClosednessNetworkRewrittenOnSmallBall}
and \eqref{eq:ReLUClosednessNetworkLimitRewrittenOnSmallBall} show that
$g_i^{(k)} + q_i^{(k)}$ converges pointwise to $G_i$ on $B_{r_i}(x_i)$.
Furthermore, since $x_i \in S_{\alpha^{(i)}, \beta^{(i)}}$, it is not hard
to see that there is some $\eps_0 > 0$ with
$\big( U_{\alpha^{(i)}, \beta^{(i)}}^{(\eps, \pm)} \big)^{\circ} \cap B_{r_i} (x_i) \neq \emptyset$
for all $\eps \in (0, \eps_0)$; for the details, we refer to Step~1 in the
proof of Lemma~\ref{lem:PiecewiseLinearReLURepresentation}.
Finally, as a consequence of Step~5, we see for arbitrary $\eps \in (0, \eps_0)$
that $g_i^{(k)} + q_i^{(k)}$ and $G_i$ are both affine-linear on
$U_{\alpha^{(i)}, \beta^{(i)}}^{(\eps, \pm)}$, at least for $k$ large enough
(depending on $\eps$).
Thus, the observation from above (with $\Theta = U_{\alpha^{(i)}, \beta^{(i)}}^{(\eps, \pm)}$
and $U = \Theta^{\circ} \cap B_{r_i} (x_i)$) implies that $g_i^{(k)} + q_i^{(k)} \to G_i$
pointwise on $U_{\alpha^{(i)}, \beta^{(i)}}^{(\eps, \pm)}$, for \emph{arbitrary}
$\eps \in (0,\eps_0)$. 

Because of
\(
  \bigcup_{\sigma \in \{\pm\}}
    \bigcup_{0 < \eps < \eps_0}
      U_{\alpha^{(i)}, \beta^{(i)}}^{(\eps,\sigma)}
  = \R^d \setminus S_{\alpha^{(i)}, \beta^{(i)}},
\)
this implies
\begin{equation}
  g_i^{(k)} + q_i^{(k)} \xrightarrow[k \to \infty]{} G_i
  \quad \text{pointwise on} \quad
  \R^d \setminus S_{\alpha^{(i)}, \beta^{(i)}}
  \quad \text{for any} \quad i \in \underline{r} .
  \label{eq:AugmentedPointwiseConvergence}
\end{equation}

\medskip{}

\noindent
\textbf{Step 7 (Finishing the proof):} For arbitrary $\delta \in (0, \delta_0)$,
let us set
\[
  \Lambda_\delta
  \coloneqq \Omega_{\delta}^\circ \setminus \bigcup_{i=1}^r S_{\alpha^{(i)}, \beta^{(i)}} .
\]
Then, Equations~\eqref{eq:ReLUClosednessRealizationDecomposition} and
\eqref{eq:AugmentedPointwiseConvergence} imply for $k \geq N_3 (\delta)$ that
\[
  g_{r+1}^{(k)} - \sum_{i=1}^r q_i^{(k)}
  = \sum_{i=1}^{r+1} g_i^{(k)} - \Big( \sum_{i=1}^r g_i^{(k)} + q_i^{(k)} \Big)
  = f_{n_k} - \Big( \sum_{i=1}^r g_i^{(k)} + q_i^{(k)} \Big)
  \xrightarrow[k\to\infty]{\text{pointwise on } \Lambda_\delta}
  f - \sum_{i=1}^r G_i .
\]
But since $g_{r+1}^{(k)}$ and all $q_i^{(k)}$ are affine-linear,
and since $\Lambda_\delta$ is an open set of positive measure, this implies
that there is an affine-linear map
$\psi : \R^d \to \R, x \mapsto \langle \zeta , x \rangle + \kappa$ satisfying
$f - \sum_{i=1}^r G_i = \psi$ on $\Lambda_\delta$, for arbitrary $\delta \in (0, \delta_0)$.
Note that $\psi$ is independent of the choice of $\delta$, and thus
\[
  f = \psi + \sum_{i = 1}^r G_i
  \quad \text{on} \quad
  \bigcup_{0 < \delta < \delta_0} \Lambda_\delta
  = \Omega^\circ \setminus \bigcup_{i=1}^r S_{\alpha^{(i)}, \beta^{(i)}} .
\]
But the latter set is dense in $\Omega$ (since its complement is a null-set),
and $f$ and $\psi + \sum_{i = 1}^r G_i$ are continuous on $\Omega$.
Hence,
\[
  f(x) = \psi(x) + \sum_{i = 1}^r G_i(x)
    = \Big( \kappa + \sum_{i=1}^r \kappa_i \Big)
      + \Big\langle \zeta + \sum_{i=1}^r \zeta_i , x \Big\rangle
      + \sum_{i=1}^r
          c_i \cdot
              \varrho_a
              \big(
                \langle \alpha^{(i)}, x \rangle + \beta^{(i)}
              \big)
  \quad \text{ for all } x \in \Omega .
\]
Recalling from Steps 3 and 4 that $r < N_0 $, this implies
$f \in \cRN_{\varrho_a}^\Omega((d,r+1,1)) \subset \cRN_{\varrho_a}^\Omega((d,N_0,1))$,
as claimed.
Here, we implicitly used that
\[
  \langle \alpha, x \rangle + \beta
  = \varrho_a \big( \langle \alpha, x \rangle + d B \, \| \alpha \|_{\ell^2} \big)
    + \beta - d B \, \| \alpha \|_{\ell^2}
  \quad \text{ for all } x \in \Omega
                 \text{ and arbitrary } \alpha \in \R^d, \, \beta \in \R ,
\]
since $\langle \alpha, x \rangle + d B \, \| \alpha \|_{\ell^2} \geq 0$
for $x \in \Omega = [-B,B]^d$, so that
$\varrho_a (\langle \alpha, x \rangle + d B \| \alpha \|_{\ell^2})
 = \langle \alpha, x \rangle + d B \| \alpha \|_{\ell^2}$.
\end{proof}

\section{Proofs of the results in Section~\ref{sec:InverseStability}}\label{app:InvStab}

\subsection{Proof of Proposition~\ref{prop:RealizationContinuity}}
\label{app:RealCont}

\textbf{Step 1:} We first show that if $(f_n)_{n \in \N}$ and
$(g_n)_{n \in \N}$ are sequences of continuous functions $f_n : \R^d \to \R^N$
and $g_n : \R^N \to \R^D$ that satisfy $f_n \to f$ and $g_n \to g$ with
locally uniform convergence, then also $g_n \circ f_n \to g \circ f$
locally uniformly.

To see this, let $R, \eps > 0$ be arbitrary.
On $\overline{B_R}(0) \subset \R^d$, we then have $f_n \to f$ uniformly.
In particular, $C \coloneqq \sup_{n \in \N} \sup_{|x| \leq R} | f_n (x) | < \infty$;
here, we implicitly used that $f$ and all $f_n$ are continuous,
and hence bounded on $\overline{B_R}(0)$.
But on $\overline{B_C}(0) \subset \R^N$, we have $g_n \to g$ uniformly,
so that there is some $n_1 \in \N$ with
$| g_n (y) - g(y) | < \eps \vphantom{\sum_j}$
for all $n \geq n_1$ and all $y \in \R^N$ with $| y | \leq C$.
Furthermore, $g$ is uniformly continuous on $\overline{B_C} (0)$, so that
there is some $\delta > 0$ with $| g(y) - g(z) | < \eps$ for all
$y,z \in \overline{B_C} (0)$ with $| y-z | \leq \delta$.
Finally, by the uniform convergence of $f_n \to f$ on $\overline{B_R}(0)$,
we get some $n_2 \in \N$ with $| f_n (x) - f(x) | \leq \delta$ for all
$n \geq n_2$ and all $x \in \R^d$ with $| x | \leq R$.

Overall, these considerations show for $n \geq \max \{n_1, n_2\}$ and
$x \in \R^d$ with $| x | \leq R$ that
\[
  | g_n (f_n (x)) - g(f(x)) |
  \leq | g_n (f_n (x)) - g(f_n(x)) | + | g(f_n (x)) - g(f(x)) |
  \leq \eps + \eps .
\]

\medskip{}

\noindent
\textbf{Step 2:} We show that $\Realization^{\Omega}_{\varrho}$ is continuous.
Assume that a sequence $(\Phi_n)_{n \in \N} \subset \cN((d,N_1,\dots,N_L))$
given by ${\Phi_n = \big( (A_1^{(n)}, b_1^{(n)}), \dots, (A_L^{(n)}, b_L^{(n)}) \big)}$
satisfies ${\Phi_n \to \Phi = \big( (A_1,b_1), \dots, (A_L,b_L) \big)} \in \cN((d,N_1,\dots,N_L))$.
For $\ell \in \FirstN{L-1}$ set
\begin{align*}
  \alpha_\ell^{(n)} : \
  & \R^{N_{\ell - 1}} \to     \R^{N_\ell}, ~
  x                 \mapsto \varrho_\ell(A_\ell^{(n)} \, x + b_\ell^{(n)}), \\
  \alpha_\ell : \
  & \R^{N_{\ell - 1}} \to     \R^{N_\ell}, ~
  x                 \mapsto \varrho_\ell(A_\ell \, x + b_\ell),
\end{align*}
where $\varrho_\ell \coloneqq \varrho \times \cdots \times \varrho$ denotes the
$N_\ell$-fold cartesian product of $\varrho$.
Likewise, set
\[
  \alpha_L^{(n)} : \R^{N_{L-1}} \to     \R^{N_L} , ~
                   x            \mapsto A_L^{(n)} x + b_L^{(n)}
  \qquad \text{and} \qquad
  \alpha_L : \R^{N_{L-1}} \to     \R^{N_L} , ~
             x            \mapsto A_L \, x + b_L .
\]

By what was shown in Step~1, it is not hard to see
for every $\ell \in \FirstN{L}$ that
$\alpha_\ell^{(n)} \to \alpha_\ell$ locally uniformly as $n \to \infty$.
By another (inductive) application of Step~1, this shows
\[
  \Realization^{\Omega}_{\varrho} (\Phi_n)
  = \alpha_L^{(n)} \circ \cdots \circ \alpha_1^{(n)}
  \to \alpha_L \circ \cdots \circ \alpha_1
  =  \Realization^{\Omega}_{\varrho} (\Phi)
\]
with locally uniform convergence. Since $\Omega$ is compact, this implies
uniform convergence on $\Omega$, and thus completes the proof of the first
claim.

\medskip{}

\noindent
\textbf{Step~3:} Let $\varrho_\ell \coloneqq \varrho \times \cdots \times \varrho$
be the $N_\ell$-fold cartesian product of $\varrho$ in case of
$\ell \in \FirstN{L-1}$, and set $\varrho_L \coloneqq \identity_{\R^{N_L}}$.
For arbitrary $x \in \Omega$ and
$\Phi = \big( (A_1,b_1), \dots, (A_L,b_L) \big) \in \cN(S)$,
define inductively $\alpha_x^{(0)} (\Phi) \coloneqq x \in \R^d = \R^{N_0}$, and
\[
  \alpha_x^{(\ell+1)} (\Phi)
  \coloneqq \varrho_{\ell+1} \big( A_{\ell+1} \, \alpha_x^{(\ell)} (\Phi) + b_{\ell+1} \big)
  \in \R^{N_{\ell+1}}
  \quad \text{for} \quad \ell \in \{0,\dots,L-1\} .
\]
Let $R > 0$ be fixed, but arbitrary.
We will prove by induction on $\ell \in \{0, \dots, L\}$ that
\[
  \| \alpha_x^{(\ell)} (\Phi) \|_{\ell^\infty} \leq C_{\ell,R}
  \quad \text{and} \quad
  \| \alpha_x^{(\ell)} (\Phi) - \alpha_x^{(\ell)} (\Psi) \|_{\ell^\infty}
  \leq M_{\ell,R} \cdot \|\Phi - \Psi\|_{\mathrm{total}}
\]
for suitable $C_{\ell,R},M_{\ell,R} > 0$ and arbitrary $x \in \Omega$
and $\Phi,\Psi \in \cN(S)$ with
$\|\Phi\|_{\mathrm{total}}, \|\Psi\|_{\mathrm{total}} \leq R$.

This will imply that $\Realization^{\Omega}_{\varrho}$ is locally Lipschitz,
since clearly $\Realization_\varrho^\Omega (\Phi)(x) = \alpha_x^{(L)} (\Phi)$, and hence
\[
  \|
    \Realization^{\Omega}_{\varrho} (\Phi) - \Realization^{\Omega}_{\varrho} (\Psi)
  \|_{\sup}
  = \sup_{x \in \Omega} | \alpha_x^{(L)}(\Phi) - \alpha_x^{(L)}(\Psi) |
  \leq M_{L,R} \cdot \| \Phi - \Psi\|_{\mathrm{total}} .
\]

The case $\ell = 0$ is trivial:
On the one hand, $| \alpha_x^{(0)}(\Phi) - \alpha_x^{(0)}(\Psi) |
= 0 \leq \| \Phi - \Psi \|_{\mathrm{total}}$.
On the other hand, since $\Omega$ is bounded, we have $| \alpha_x^{(0)} (\Phi) | = | x | \leq C_0$
for a suitable constant $C_0 = C_0 (\Omega)$.

For the induction step, let us write $\Psi = \big( (B_1,c_1), \dots, (B_L,c_L)\big)$, and note that
\[
  \| A_{\ell+1} \, \alpha_x^{(\ell)} (\Phi) + b_{\ell+1} \|_{\ell^\infty}
  \leq N_\ell \| A_{\ell+1} \|_{\max}
       \cdot \|\alpha_x^{(\ell)}(\Phi)\|_{\ell^\infty}
       + \| b_{\ell+1} \|_{\ell^\infty}
  \leq (1 + N_\ell C_{\ell,R}) \cdot \|\Phi\|_{\mathrm{total}}
  \eqqcolon K_{\ell+1,R} .
\]
Clearly, the same estimate holds with $A_{\ell+1},b_{\ell+1}$ and $\Phi$
replaced by $B_{\ell+1}, c_{\ell+1}$ and $\Psi$, respectively.
Next, observe that with $\varrho$ also
$\varrho_{\ell+1}$ is locally Lipschitz.
Thus, there is $\Gamma_{\ell+1,R} > 0$ with
\[
  \| \varrho_{\ell+1} (x) - \varrho_{\ell+1}(y) \|_{\ell^\infty}
    \leq \Gamma_{\ell+1,R} \cdot \| x-y \|_{\ell^\infty}
  \qquad \text{ for all } x,y \in \R^{N_{\ell+1}}
                    \text{ with } \| x \|_{\ell^\infty}, \| y \|_{\ell^\infty} \leq K_{\ell+1,R} .
\]
On the one hand, this implies
\begin{align*}
  \|
    \alpha_x^{(\ell+1)} (\Phi)
  \|_{\ell^\infty}
  & \leq \big\|
           \varrho_{\ell+1} \big( A_{\ell+1} \, \alpha_x^{(\ell)} (\Phi) + b_{\ell+1} \big)
           - \varrho_{\ell+1} (0)
         \big\|_{\ell^\infty}
         + \| \varrho_{\ell+1} (0) \|_{\ell^\infty} \\
  & \leq \Gamma_{\ell+1,R} \, \|
                                A_{\ell+1} \, \alpha_x^{(\ell)} (\Phi) + b_{\ell+1}
                              \|_{\ell^\infty}
         + \| \varrho_{\ell+1} (0) \|_{\ell^\infty}
  \leq \Gamma_{\ell+1,R} \, K_{\ell+1, R} + \| \varrho_{\ell+1} (0) \|_{\ell^\infty}
  \eqqcolon C_{\ell+1, R} .
\end{align*}
On the other hand, we also get
\begin{align*}
  & \|
      \alpha_x^{(\ell+1)} (\Phi) - \alpha_x^{(\ell+1)}(\Psi)
    \|_{\ell^\infty} \\
  & = \|
        \varrho_{\ell+1} (A_{\ell+1} \alpha_x^{(\ell)}(\Phi) + b_{\ell+1})
        - \varrho_{\ell+1} (B_{\ell+1} \alpha_x^{(\ell)}(\Psi) + c_{\ell+1})
      \|_{\ell^\infty} \\
  & \leq \Gamma_{\ell+1,R} \cdot
         \|
            (A_{\ell+1} \alpha_x^{(\ell)}(\Phi) + b_{\ell+1})
            - ( B_{\ell+1} \alpha_x^{(\ell)}(\Psi) + c_{\ell+1})
         \|_{\ell^\infty} \\
  & \leq \Gamma_{\ell+1,R} \cdot
         \left(
           \|
             (A_{\ell+1} - B_{\ell+1}) \alpha_x^{(\ell)} (\Phi)
           \|_{\ell^\infty}
           + \|
               B_{\ell+1} (\alpha_x^{(\ell)}(\Phi) - \alpha_x^{(\ell)}(\Psi))
             \|_{\ell^\infty}
           + \| b_{\ell+1} - c_{\ell+1} \|_{\ell^\infty}
         \right) \\
  & \leq \Gamma_{\ell+1,R} \cdot
         \left(
           N_\ell \cdot \| \Phi - \Psi \|_{\mathrm{total}}
           \cdot \|\alpha_x^{(\ell)}(\Phi)\|_{\ell^\infty}
           + N_\ell \cdot \| \Psi \|_{\mathrm{total}} \cdot
             \|
               \alpha_x^{(\ell)}(\Phi) - \alpha_x^{(\ell)}(\Psi)
             \|_{\ell^\infty}
           + \| \Phi - \Psi \|_{\mathrm{total}}
         \right) \\
  & \leq \Gamma_{\ell+1,R} \cdot (N_\ell C_{\ell,R} + R N_\ell M_{\ell,R} + 1)
         \cdot \| \Phi - \Psi \|_{\mathrm{total}}
    \eqqcolon M_{\ell+1, R} \cdot \| \Phi - \Psi \|_{\mathrm{total}} .
\end{align*}

\medskip{}

\noindent
\textbf{Step~4:} Let $\varrho$ be Lipschitz with Lipschitz constant $M$,
where we assume without loss of generality that $M \geq 1$.
With the functions $\varrho_\ell$ from the preceding step,
it is not hard to see that each $\varrho_\ell$ is $M$-Lipschitz,
where we use the $\| \cdot \|_{\ell^\infty}$-norm on $\R^{N_\ell}$.

Let $\Phi = \! \big( (A_1,b_1), \dots, (A_L,b_L) \big) \! \in \cN(S)$, and
$\alpha_\ell : \R^{N_{\ell-1}} \! \to \R^{N_\ell}, x \mapsto \varrho_\ell (A_\ell \, x + b_\ell)$
for ${\ell \in \FirstN{L-1}}$.
Then, $\alpha_\ell$ is Lipschitz with 
$\Lip (\alpha_\ell) \leq M \cdot \|A_\ell\|_{\ell^\infty \to \ell^\infty}
\leq M \cdot N_{\ell-1} \cdot \| A \|_{\max}
\leq M N_{\ell-1} \cdot \|\Phi\|_{\mathrm{scaling}}$.
Thus, we finally see that
$\Realization^{\Omega}_{\varrho} (\Phi) = \alpha_L \circ \cdots \circ \alpha_1$ is
Lipschitz with Lipschitz constant
$M^L \cdot N_0 \cdots N_{L-1} \cdot \|\Phi\|_{\mathrm{scaling}}^L$.
This proves the final claim of the proposition
when choosing the $\ell^\infty$-norm on $\R^d$ and $\R^{N_L}$.
Of course, choosing another norm than the $\ell^\infty$-norm
can be done, at the cost of possibly enlarging
the constant $C$ in the statement of the proposition.
\hfill$\square$

\subsection{Proof of Theorem~\ref{thm:InverseStability}}
\label{app:InverseStability}

\textbf{Step 1:} For $a > 0$, define
\[
  f_a : \R \to \R, ~x \mapsto \varrho(x+a) - 2 \varrho(x) + \varrho(x-a) .
\]
Our claim in this step is that there is some $a > 0$ with
$f_a \not\equiv \mathrm{const}$.

Let us assume towards a contradiction that this fails; that is, $f_a \equiv c_a$ for all $a > 0$.
Since $\varrho$ is Lipschitz continuous, it is at most of
linear growth, so that $\varrho$ is a tempered distribution.
We will now make use of the \emph{Fourier transform}, which we define by
$\widehat{f}(\xi) = \int_{\R} f(x) \, e^{-2 \pi i x \xi} \, d x$
for $f \in L^1(\R)$, as in \cite{GrafakosClassical,FollandRA},
where it is also explained how the Fourier transform is extended
to the space of tempered distributions.
Elementary properties of the Fourier transform for tempered distributions
(see \cite[Proposition~2.3.22]{GrafakosClassical}) show
\[
  c_a \cdot \delta_0 = \widehat{f_a} = \widehat{\varrho} \cdot g_a
  \qquad \text{with} \qquad
  g_a : \R \to \R, ~\xi \mapsto e^{2\pi i a \xi} -2 + e^{-2\pi i a \xi} .
\]

Next, setting $z(\xi) \coloneqq e^{2\pi i a \xi} \neq 0$, we observe that
\[
  g_a (\xi)
  = z(\xi) - 2 + [z(\xi)]^{-1}
  = [z(\xi)]^{-1} \cdot (z^2(\xi) -2 z(\xi) + 1)
  = [z(\xi)]^{-1} \cdot (z(\xi)-1)^2
  \neq 0 ,
\]
as long as $z(\xi) \neq 1$, that is, as long as $\xi \notin a^{-1} \Z$.

Let $\varphi \in C_c^\infty (\R)$ such that $0 \not \in \supp \varphi$
be fixed, but arbitrary.
This implies $\supp \varphi \subset \R \setminus a^{-1} \Z$ for some sufficiently small $a > 0$.
Since $g_a$ vanishes nowhere on the compact set $\supp \varphi$,
it is not hard to see that there is some smooth, compactly
supported function $h$ with $h \cdot g_a \equiv 1$ on the support of $\varphi$.
All in all, we thus get
\[
  \langle \widehat{\varrho} , \varphi \rangle_{\Schwartz',\Schwartz}
  = \langle
      \widehat{\varrho} \cdot g_a , h \cdot \varphi
    \rangle_{\Schwartz',\Schwartz}
  = \langle \widehat{f_a} , h \cdot \varphi \rangle_{\Schwartz',\Schwartz}
  = c_a \cdot h(0) \cdot \varphi(0) = 0 .
\]
Since $\varphi \in C_c^\infty (\R)$ with $0 \not \in \supp \varphi$ was arbitrary, we have
shown $\supp \widehat{\varrho} \subset \{0\}$.
But by \mbox{\cite[Corollary~2.4.2]{GrafakosClassical}}, this implies that $\varrho$ is a polynomial.
Since the only globally Lipschitz continuous polynomials are affine-linear,
$\varrho$ must be affine-linear, contradicting the prerequisites of the theorem.

\medskip{}

\noindent
\textbf{Step~2:} In this step we construct certain continuous functions $F_n : \R^d \to \R$
which satisfy $\mathrm{Lip}(F_n|_\Omega) \to \infty$ and $F_n \to 0$ uniformly on $\R^d$.
We will then use these functions in the next step to construct the desired networks $\Phi_n$.

We first note that each function $f_a$ from Step 1 is bounded.
In fact, if $\varrho$ is $M$-Lipschitz, then
\begin{equation}
   | f_a (x) |
  \leq | \varrho(x+a) - \varrho(x) | + | \varrho(x-a) - \varrho(x) |
  \leq 2M |a| .
  \label{eq:InverseStabilityHatFunctionBounded}
\end{equation}
Next, recall that $\varrho$ is Lipschitz continuous and not affine-linear.
Therefore, Lemma~\ref{lem:LocallyLipschitzHasNonzeroDerivativeSomewhere}
shows that there is some $t_0 \in \R$ such that $\varrho$ is differentiable at $t_0$
with $\varrho'(t_0) \neq 0$.
Therefore, Proposition~\ref{prop:Identity} shows
that there is a neural network $\Phi \in \cN((1,\dots,1))$ with $L-1$ layers such that
$\psi \coloneqq \Realization^{\R}_{\varrho} (\Phi)$ is differentiable at the origin
with $\psi(0) = 0$ and $\psi'(0) = 1$.
By definition, this means that there is a function $\delta : \R \to \R$
such that $\psi(x) = x + x \cdot \delta(x)$ and $\delta (x) \to 0 = \delta(0)$ as $x \to 0$.

Next, since $\Omega$ has nonempty interior, there exist $x_0 \in \R^d$
and $r > 0$ with $x_0 + [-r,r]^d \subset \Omega$.
Let us now choose $a > 0$ with $f_a \not\equiv \mathrm{const}$
(the existence of such an $a > 0$ is implied by the previous step), and define
\[
  F_n :
  \R^d \to     \R,~
  x    \mapsto \psi\left( n^{-1} \cdot f_a(n^2 \cdot (x-x_0)_1) \right)
  .
\]

Since $f_a$ is not constant, there are $b,c \in \R$ with $b < c$ and $f_a (b) \neq f_a (c)$.
Because of $\delta(x) \to 0$ as $x \to 0$,
we see that there is some $\kappa > 0$ and some $n_1 \in \N$ with
\[
  | f_a(b) - f_a(c) |
  - | f_a(b) | \cdot | \delta(f_a(b) / n) |
  - | f_a (c) | \cdot | \delta(f_a(c) / n) |
  \geq \kappa > 0
  \qquad \text{ for all } n \geq n_1 .
\]

Let us set $x_n \coloneqq x_0 + n^{-2} \cdot (b,0,\dots,0) \in \R^d$ and
$y_n \coloneqq x_0 + n^{-2} \cdot (c,0,\dots,0) \in \R^d$, and observe
$x_n, y_n \in \Omega$ for $n \in \N$ large enough.
We have $| x_n - y_n | = n^{-2} \cdot | b-c |$.
Furthermore, using the expansion $\psi(x) = x + x \cdot \delta(x)$, and noting
$f_a (n^2 (x_n - x_0)_1) = f_a(b)$ as well as
$f_a(n^2 (y_n - x_0)_1) = f_a(c)$, we get
\begin{align*}
   | F_n (x_n) - F_n (y_n) |
  & = | \psi (f_a(b) / n) - \psi (f_a(c) / n) | \\
  & = \left|
        \frac{f_a(b)}{n}
        - \frac{f_a(c)}{n}
        + \frac{f_a(b)}{n} \cdot \delta\left( \frac{f_a(b)}{n} \right)
        - \frac{f_a(c)}{n} \cdot \delta \left( \frac{f_a(c)}{n} \right)
      \right| \\
  & \geq \frac{1}{n} \cdot
         \big(
           | f_a(b) - f_a(c) |
           - | f_a(b) | \cdot | \delta(f_a(b) / n) |
           - | f_a (c) | \cdot | \delta(f_a(c) / n) |
         \big)
  \geq \kappa / n ,
\end{align*}
as long as $n \geq n_1$ is so large that $x_n,y_n \in \Omega$.
But this implies
\begin{equation*}
  \mathrm{Lip}(F_n |_\Omega)
  \geq \frac{| F_n (x_n) - F_n (y_n) |}{| x_n - y_n |}
  \geq \frac{\kappa / n}{n^{-2} \cdot | b-c |}
  =    n \cdot \frac{\kappa}{| b-c |} \xrightarrow[n\to\infty]{} \infty .
\end{equation*}

It remains to show $F_n \to 0$ uniformly on $\R^d$.
Thus, let $\eps > 0$ be arbitrary.
By continuity of $\psi$ at $0$, there is some $\delta > 0$ with $| \psi(x) | \leq \eps$
for $| x | \leq \delta$.
But Equation~\eqref{eq:InverseStabilityHatFunctionBounded} shows
$|n^{-1} \cdot f_a (n^{-2} \cdot (x-x_0)_1)| \leq n^{-1} \cdot 2M|a| \leq \delta$
for all $x \in \R^d$ and all $n \geq n_0$, with $n_0 = n_0(M,a,\delta) \in \N$
suitable.
Hence, $| F_n (x) | \leq \eps$ for all $n \geq n_0$ and $x \in \R^d$.

\medskip{}

\textbf{Step 3:} In this step, we construct the networks $\Phi_n$.
For $n \in \N$ define
\[
  A_1^{(n)} \coloneqq n^2 \cdot \left(
                           \begin{matrix}
                             1      & 0      & \cdots & 0      \\
                             1      & 0      & \cdots & 0      \\
                             1      & 0      & \cdots & 0
                           \end{matrix}
                         \right) \in \R^{3 \times d}
  \quad \text{and} \quad
  b_1^{(n)} \coloneqq \left(
                 \begin{matrix}
                   -n^2 \cdot (x_0)_1 + a \\
                   -n^2 \cdot (x_0)_1     \\
                   -n^2 \cdot (x_0)_1 - a
                 \end{matrix}
               \right) \in \R^3 ,
\]
as well as $A_2^{(n)} \coloneqq n^{-1} \cdot (1, -2, 1) \in \R^{1 \times 3}$ and
$b_2^{(n)} \coloneqq 0 \in \R^1$.
A direct calculation shows
\[
  \Realization^{\R^d}_{\varrho} (\Phi_n^{(0)}) (x)
  = n^{-1} \cdot f_a (n^2 \cdot (x-x_0)_1)
  \quad \text{ for all } x \in \R^d,
        \text{ where } \Phi_n^{(0)} \coloneqq \big(
                                                (A_1^{(n)},b_1^{(n)}),
                                                (A_2^{(n)},b_2^{(n)})
                                              \big).
\]
Thus, with the concatenation operation introduced in
Definition~\ref{def:Concatenation}, the network
$\Phi_n^{(1)} \coloneqq \Phi \, \conc \Phi_n^{(0)}$ satisfies
$\Realization^{\Omega}_{\varrho}(\Phi_n^{(1)}) = F_n|_\Omega$.
Furthermore, it is not hard to see that $\Phi_n^{(1)}$ has $L$ layers and
has the architecture $(d,3,1,\dots,1)$.
From this and because of $N_1 \geq 3$, by Lemma \ref{lem:enlarge}
there is a network $\Phi_n$ with architecture $(d,N_1,\dots,N_{L-1},1)$
and $\Realization^{\Omega}_{\varrho}(\Phi_n) = F_n|_\Omega$. By Step 2, this implies
$\Realization^{\Omega}_{\varrho} (\Phi_n) = F_n|_\Omega \to 0$ uniformly on $\Omega$,
as well as $\mathrm{Lip}(\Realization^{\Omega}_{\varrho} (\Phi_n)) \to \infty$
as $n \to \infty$.

\medskip{}

\textbf{Step 4:} In this step, we establish the final property which is stated in the theorem.
For this, let us assume towards a contradiction
that there is a family of networks $(\Psi_n)_{n \in \N}$ with architecture $S$
and $\Realization^{\Omega}_{\varrho}(\Psi_n) = \Realization^{\Omega}_{\varrho}(\Phi_n)$,
some $C > 0$, and a subsequence $(\Psi_{n_r})_{r \in \N}$ with
$\| \Psi_{n_r} \|_{\mathrm{scaling}} \leq C$ for all $r \in \N$.
In view of the last part of Proposition~\ref{prop:RealizationContinuity},
there is a constant $C' = C'(\varrho,S) > 0$ with
\[
  \mathrm{Lip} \big(\Realization^{\Omega}_{\varrho} (\Phi_{n_r}) \big)
  = \mathrm{Lip} \big( \Realization^{\Omega}_{\varrho} (\Psi_{n_r}) \big)
  \leq C' \cdot \| \Psi_{n_r} \|_{\mathrm{scaling}}^L
  \leq C' \cdot C^L ,
\]
in contradiction to $\mathrm{Lip} \big(\Realization^{\Omega}_{\varrho}(\Phi_n) \big) \to \infty$.
\hfill$\square$

\subsection{Proof of Corollary~\ref{cor:Quotient}}
\label{app:Quotient}

Let us denote the range of the realization map by $R$.
By definition (see \cite[Page~65]{LeeTopologicalManifolds}),
$\Realization^{\Omega}_{\varrho}$ is a quotient map if and only if
\[
  \forall \, M \subset R: \qquad
  M \subset R \text{ open}
  \quad \Longleftrightarrow \quad
  \left(\Realization^{\Omega}_{\varrho}\right)^{-1}(M) \subset \cN(S) \text{ open}.
\]
Clearly, by switching to complements, we can equivalently replace ``open'' by ``closed'' everywhere.

Now, choose a sequence of neural networks $(\Phi_n)_{n \in \N}$ as in
Theorem~\ref{thm:InverseStability}, and set
$F_n \coloneqq \Realization^{\R^d}_{\varrho} (\Phi_n)$.
Since $\mathrm{Lip}(F_n|_{\Omega}) \to \infty$, we have
$F_n |_{\Omega} \not\equiv 0$ for all $n \geq n_0$ with $n_0 \in \N$ suitable.
Define $M \coloneqq \{F_n |_{\Omega} \with n \geq n_0 \} \subset R$.
Note that $M \subset R \subset C(\Omega)$ is \emph{not} closed, since
$F_n |_{\Omega} \to 0$ uniformly, but $0 \in R \setminus M$.
Hence, once we show that $\left(\Realization^{\Omega}_{\varrho}\right)^{-1}(M)$ is closed,
we will have shown that $\Realization^{\Omega}_{\varrho}$ is not a quotient map.

Thus, let $(\Psi_n)_{n \in \N}$ be a sequence in
$\left(\Realization^{\Omega}_{\varrho}\right)^{-1}(M)$ and assume
$\Psi_n \to \Psi$ as $n\to \infty$.
In particular, $\| \Psi_n \|_{\mathrm{scaling}} \leq C$ for
some $C > 0$ and all $n \in \N$.
We want to show $\Psi \in \left(\Realization^{\Omega}_{\varrho}\right)^{-1}(M)$
as well.
Since $\Psi_n \in \left(\Realization^\Omega_{\varrho}\right)^{-1}(M)$, there is
for each $n \in \N$ some $r_n \in \N$ with
$\Realization^{\Omega}_{\varrho} (\Psi_n) = F_{r_n}|_{\Omega}$.
Now there are two cases:

\smallskip{}

\noindent
\textbf{Case~1:} The family $(r_n)_{n \in \N}$ is infinite.
But in view of Proposition~\ref{prop:RealizationContinuity}, we have
\[
  \mathrm{Lip}(F_{r_n} |_{\Omega})
  =    \mathrm{Lip}(\Realization^{\Omega}_{\varrho} (\Psi_n))
  \leq C' \cdot \| \Psi_n \|_{\mathrm{scaling}}^L
  \leq C' \cdot C^L
\]
for a suitable constant $C' = C'(\varrho, S)$, in contradiction
to the fact that $\mathrm{Lip}(F_{r_n}|_{\Omega}) \to \infty$ as
$r_n \to \infty$. Thus, this case cannot occur.

\medskip{}

\noindent
\textbf{Case~2:} The family $(r_n)_{n \in \N}$ is finite.
Thus, there is some $N \in \N$ with $r_n = N$ for infinitely many $n \in \N$,
that is, $\Realization^\Omega_{\varrho}(\Psi_n) = F_{r_n}|_{\Omega} = F_N|_{\Omega}$ for
infinitely many $n \in \N$.
But since $\Realization^{\Omega}_{\varrho} (\Psi_{n})
\to \Realization^{\Omega}_{\varrho}(\Psi)$
as $n \to \infty$ (by the continuity of the realization map), this implies
$\Realization^{\Omega}_{\varrho} (\Psi) = F_N|_{\Omega} \in M$, as desired.
\hfill$\square$

\end{document}